\documentclass{amsart}
\pdfoutput=1
\usepackage[utf8]{inputenc}

\usepackage[T1]{fontenc}    
\usepackage{ae,aecompl}     

\usepackage{url}

\usepackage{amsmath}
\usepackage{amsfonts}
\usepackage{amssymb}
\usepackage{amsthm} %
\usepackage{mathrsfs} %
\usepackage{enumerate}

\usepackage{color}
\usepackage{tikz}
\usepackage{pgfplots}
\usetikzlibrary{calc}
\usetikzlibrary{shapes}
\usepackage{tkz-base}
\usepackage{tkz-euclide}
\usetkzobj{all}

\usepackage{caption}
\numberwithin{equation}{section}
\theoremstyle{plain} 
\newtheorem{thm}{Theorem}[section]
\newtheorem{cor}[thm]{Corollary}
\newtheorem{prop}[thm]{Proposition}
\newtheorem{lem}[thm]{Lemma}

\theoremstyle{definition}
\newtheorem{defi}[thm]{Definition}

\newtheorem{rem}[thm]{Remark}
\newtheorem{question}[thm]{Question}
\newtheorem{ex}[thm]{Example}

\newcommand{\h}{\widehat}
\newcommand{\cone}{\operatorname{cone}}
\newcommand{\Span}{\operatorname{span}}
\newcommand{\Sp}{\operatorname{Sp}}
\newcommand{\dprec}{\,\dot\prec\,}
\newcommand{\oo}{\infty}

\newcommand{\dep}{\textnormal{dp}}
\newcommand{\sgn}{\operatorname{sgn}}
\newcommand{\hQact}{\h Q_{\textnormal{act}}}
\newcommand{\Gen}{\operatorname{Gen}}
\newcommand{\Diag}{\operatorname{Diag}}
\newcommand{\Ef}{E_{\textnormal{f}}}
\newcommand{\Efcov}{E_{\textnormal{f}}^{\textnormal{\tiny cov}}}
\newcommand{\Ecov}{E^{\textnormal{\tiny cov}}}
\newcommand{\precf}{\,\prec_{\textnormal{f}}\,}
\newcommand{\dprecf}{\,\dot\prec_{\textnormal{f}}\,}
\newcommand{\Eelem}{E_{\textnormal{elem}}}

\newcommand{\PP}{\Phi^{+}}
\newcommand{\imc}{\mathcal Z}
\newcommand{\ol}{\overline}
\newcommand{\vphi}{\operatorname{\varphi}}
\newcommand{\Acc}{\operatorname{Acc}}

\newcommand{\al}{\alpha}
\newcommand{\bt}{\beta}
\newcommand{\g}{\gamma}

\newcommand{\e}{\epsilon}
\newcommand{\wt}{\widetilde}

\newcommand{\eset}{\varnothing}
\newcommand{\sneq}{\subsetneq}

\newcommand{\Id}{\text{\rm Id}}

\newcommand{\Int}{\mathbb {Z}}

\newcommand{\mpair}[1]{B(#1)} 
\newcommand{\Mpair}{B}

\newcommand{\mset}[1]{\set{\,#1\,}}
\newcommand{\pair}[1]{\langle #1\rangle}
\newcommand{\set}[1]{\{#1\}}

\newcommand{\seq}{\subseteq}
\newcommand{\sm}{\setminus}

\newcommand{\dist}{\operatorname{\textnormal{d}}\mathopen{}}

\newenvironment{conds}{
                       
                        \begin{enumerate} }
                     {\end{enumerate} }

\newenvironment{num}{
                      
                      \begin{enumerate} }
                    {\end{enumerate} }

 \newenvironment{subconds}{
                       
                        \begin{enumerate} }
                     {\end{enumerate} }

\newcommand{\Nat}{\mathbb {N}}

\newcommand{\real}{\mathbb{R}}

\newcommand\mc{\mathcal}

   \DeclareMathOperator{\aff}{\mathrm{aff}}
   \DeclareMathOperator{\conv}{\mathrm{conv}}
   
   \DeclareMathOperator{\ri}{{\mathrm{relint}}}
   \DeclareMathOperator{\rb}{{\mathrm{rb}}}
   \DeclareMathOperator{\ray}{{\mathrm{Ray}}}

\newcommand\G{{\Gamma}}

\newcommand\sreq{{\supseteq}}

 \newcommand\ek{{\epsilon'}}
\newcommand{\Eext}{E_{\mathrm{ext}}}

\newcommand{\Eexp}{E_{\mathrm{exp}}}

\newcommand{\tmax}{t_{\mathrm{max}}}
\newcommand{\tmin}{t_{\mathrm{min}}}
\newcommand{\wh}{\widehat}

\author[M.~Dyer]{Matthew Dyer}
\address[Matthew Dyer]{University of Notre Dame\\ Department of Mathematics\\
   255 Hurley Hall \\46556-4618, USA}
\email{dyer.1@nd.edu}

\author[C. Hohlweg]{Christophe~Hohlweg}
\address[Christophe Hohlweg]{Universit\'e du Qu\'ebec \`a Montr\'eal\\
LaCIM et D\'epartement de Math\'ematiques\\ CP 8888 Succ. Centre-Ville\\
Montr\'eal, Qu\'ebec, H3C 3P8\\ Canada}
\email{hohlweg.christophe@uqam.ca}
\urladdr{http://hohlweg.math.uqam.ca}

\author[V. Ripoll]{Vivien~Ripoll}
\address[Vivien Ripoll]{Fakult\"at f\"ur Mathematik, Universit\"at Wien\\ Oskar-Morgenstern-Platz 1\\ 1090 Wien\\Austria}
\email{vivien.ripoll@univie.ac.at}
\urladdr{http://www.normalesup.org/~vripoll}

\title[Imaginary cones and limit roots of infinite Coxeter groups]{Imaginary cones and limit roots of\\ infinite Coxeter groups}


\begin{document}
\thanks{The second author is supported by a NSERC grant and the third author is supported by a postdoctoral fellowship from CRM-ISM and LaCIM}

\begin{abstract}
Let $(W,S)$ be an infinite Coxeter system. To each geometric representation of $W$ is associated a root system. While a root system lives in the positive side of the isotropic cone of its associated bilinear form, an imaginary cone lives in the negative side of the isotropic cone. Precisely on the isotropic cone, between root systems and imaginary cones, lives the set $E$ of limit points of the  directions of roots. In this article we study the close relations of the imaginary cone with the set~$E$, which leads to new fundamental results  about the structure of geometric representations of infinite Coxeter groups. In particular, we show that the $W$-action on $E$ is minimal and faithful, and that $E$ and the imaginary cone can be approximated arbitrarily well by sets of limit roots and imaginary cones of universal root subsystems of $W$, i.e., root systems for Coxeter groups without braid relations (the free object for Coxeter groups). Finally, we discuss open questions as well as  the possible relevance of our framework in other areas such as  geometric group theory.
\end{abstract}

\date{\today}


 \maketitle


\section{Introduction}\label{se:Intro}

Root systems are fundamental in the theory of Coxeter groups. 
Finite root systems and their associated finite Coxeter groups have received a lot of attention because of their fundamental role in the theories of semisimple complex Lie algebras and Lie groups, algebraic groups, quantum groups, regular polytopes, singularities, representations of quivers etc;  see for instance~\cite{bourbaki,humphreys,geck-pf,bjorner-brenti} and the references therein. 
This article rather focuses on infinite root systems (and their associated infinite Coxeter groups), for which many natural questions remain unexplored. 
Important results have been obtained  on the geometry and topology of infinite Coxeter groups; see for instance~\cite{davis,abramenko-brown} and the references therein. In particular, we mention the strong Tits' alternative of Margulis-Noskov-Vinberg (see \cite{noskovvinberg}), according to which any subgroup of a Coxeter group has a finite index subgroup which is either abelian or surjects onto a non-abelian free group. The approach used in this study  is very often related to the Tits cone, Coxeter complex or Davis complex, which are dual objects to root systems. On the other hand, root systems are natural objects to consider and they provide tools that are not provided by their dual counterparts. 
Infinite crystallographic root systems and Coxeter  groups have been  studied  
because of their natural association  with Lie algebras, Kac-Moody algebras 
and their generalizations; see for instance
 \cite{bourbaki,kac,moodypianzola,loos-neher:loc}.  Root systems of general 
 Coxeter groups  are also at the heart of fundamental work  such as B.~Brink 
 and R.~Howlett's proof that  Coxeter groups have an automatic
  structure~\cite{BH:aut} or D.~Krammer's work on the conjugacy problem for Coxeter groups~\cite{krammer}.

\smallskip

One of the main goals of  this article is  to better understand the geometric representations, and especially the  associated root systems, of infinite Coxeter groups.  A principal motivation  for that goal is the hope it  will lead to progress in  the study of reflection orders of Coxeter groups and their initial sections, which play a  significant role in relation to  Bruhat order and Iwahori-Hecke algebras  (see \cite{bjorner-brenti} for more details) and   conjecturally are  important  for associated representation categories. Despite many important potential applications, reflection orders and their initial sections are  poorly understood in  general, and   many  of their basic properties remain conjectural. For example, Conjecture 2.5 in  \cite{dyer:weak} suggests that  the initial sections, ordered by inclusion,  form a complete lattice  that may be viewed as a natural  ortholattice completion of weak order.
   In efforts to refine and prove    these conjectures for general Coxeter groups,  one fundamental difficulty is that  not  much   is known about how the roots of an infinite root system are geometrically distributed over   the space, and  it is our intention to begin to fill this gap. 
   Another motivation  is to study  discrete subgroups of isometries in quadratic spaces; for instance modules associated to geometric representations of $W$ are quadratic spaces and $W$ is itself a discrete subgroup of isometries generated by reflections. The case of Lorentzian spaces is discussed in~\cite{HPR} but the results here suggest such a  study may  be of considerable interest more generally.

\smallskip
In recent years, several studies about infinite root systems of Coxeter groups have been conducted (see for instance \cite{bonnafe-dyer,dyer:rig,dyer:weak,fu:dom}). One of the notions introduced is a nice generalization of the \emph{imaginary cone}, which first appears in the context of root systems of Kac-Moody algebras (see~\cite{kac}), to root systems of Coxeter groups in general; see~\cite{hee:imc,fu:imc,dyer:imc,edgar}.  While a root system lives in the positive side of the isotropic cone of its associated bilinear form, an imaginary cone lives in the negative side of the isotropic cone. Precisely on the isotropic cone, between root systems and imaginary cones, lives the set of limit points of the directions of roots, which we call \emph{limit roots}.

In~\cite{HLR}, the second and third author, together with
J.-P.~Labb\'e, initiated a study of the set $E(\Phi)$ of limit roots of a based root system $(\Phi,\Delta)$, with associated Coxeter system $(W,S)$.  In this second article we study the close relations of the set~$E(\Phi)$ with the imaginary cone studied by the first author~\cite{dyer:imc}, which leads to new fundamental results  about the structure of geometric representations of infinite Coxeter groups.

\smallskip

To study a root system $\Phi$ in the geometric $W$-module $V$, the
approach used in~\cite{HLR} is to consider a projective version of
$\Phi$ by cutting the cone $\cone(\Delta)$, in which the positive
roots live, by an affine hyperplane $V_1$. We obtain this way the
so-called \emph{normalized root system $\h\Phi$} that is the
intersection of the rays spanned by the roots with $V_1$. By doing so,
we obtain that $\h\Phi$ is contained in the polytope $\conv(\h\Delta)$
and therefore $\h\Phi$ (when infinite) has a non-empty set of
accumulation points denoted by $E(\Phi)$. The following properties of
$\h\Phi$ and $E(\Phi)$ were brought to light in~\cite{HLR}: $E(\Phi)$
is contained in the isotropic cone $Q$ (the red curve in
Figure~\ref{fig:intro}) of the bilinear form associated to the
geometric representation of $(W,S)$; $W$ acts on $\h\Phi \sqcup
E(\Phi)$ by projective transformations and has a nice geometric
interpretation that can be seen on Figure~\ref{fig:intro}; and
$E(\Phi)$ is the closure of the set $E_2(\Phi)$ of the limit points
obtained from dihedral reflection subgroups.  Independently, the first
author showed in~\cite{dyer:imc} that the closure of the imaginary cone
is the convex cone $\cone(E(\Phi))$ spanned by the elements $E(\Phi)$
seen as vectors in $V$.
 
  In \S\ref{se:imc}, we recall the definition of $E(\Phi)$, of the $W$-action, of the imaginary cone $\imc(\Phi)$ and bring together, with slight improvements, the frameworks and results from \cite{HLR} and \cite{dyer:imc}.  In particular we extend in \S\ref{se:imc} the projective $W$-action to include the \emph{imaginary convex set} $Z(\Phi)$ that is an   affine section of $\imc(\Phi)$, see~Figure~\ref{fig:ICB}.  Then in \S\ref{ss:minimal}, we prove our first fundamental fact: 
 the $W$-action on $E(\Phi)$ is \emph{minimal}, i.e.,  for any $x$ in $E(\Phi)$, the orbit $W\cdot x$ of $E(\Phi)$
under $W$ is dense in $E(\Phi)$ (Theorem~\ref{cor:minimal}). In order to do so, 
we study the convexity properties of $Z(\Phi)$ and give fundamental results  on the set of extreme points and exposed faces of  the closure $\ol{Z(\Phi)}$ of $Z(\Phi)$.

\begin{figure}
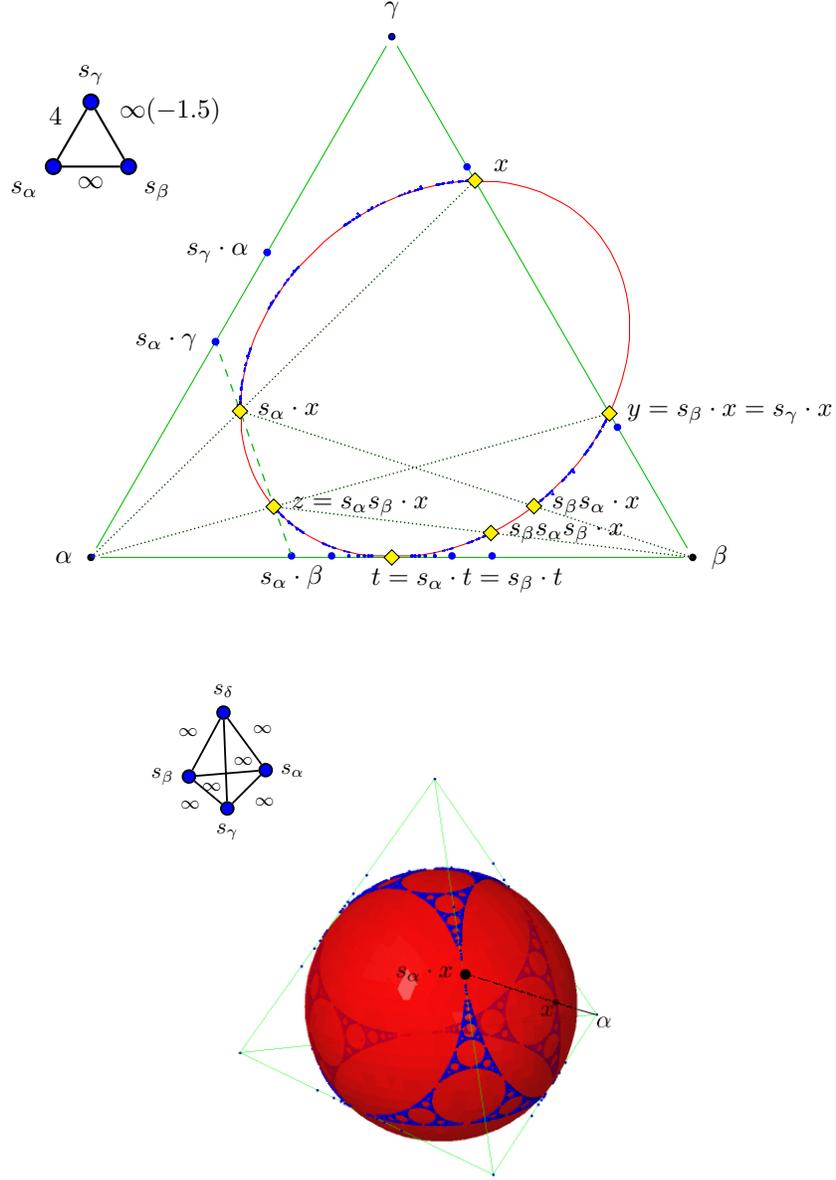

\begin{minipage}[t]{\linewidth}
\centering
\scalebox{1}{\input{FigIntro1.tex}}
\end{minipage}
\hfill
\\
\begin{minipage}[b]{\linewidth}
\centering
\scalebox{0.85}{\input{FigIntro2.tex}}
\end{minipage}%
\caption{Pictures in rank $3$ and $4$ of the normalized isotropic cone
  $\h{Q}$ (in red), the first normalized roots (in blue dots, with
  depth~$\leq 8$) for the based root system with diagram given in the upper
  left of each picture. The set $E(\Phi)$ of limit roots is the limit set of the normalized roots. It is acted on by $W$, as explained in
  \S\ref{ss:limitroots}: for example, in the upper picture, the limit root $x$ is sent to $y$ by $s_\beta$, which is then sent to $z$ by $s_\alpha$.}
\label{fig:intro}
\end{figure}

In \S\ref{se:fractal} we turn our attention to two `fractal conjectures' stated in
\cite[\S3.2]{HLR} about  fractal, self-similar descriptions of the
set of limit roots $E(\Phi)$. We use the minimality of the $W$-action
from \S\ref{se:imc}, as well as some additional work on the case
where $(\Phi,\Delta)$ is weakly hyperbolic, in order to  completely prove in any ranks 
\cite[Conjecture 3.9 ]{HLR}, as well as to prove in the weakly hyperbolic case the conjecture stated just above Conjecture~3.9. 
This will turn out to give us in~\S\ref{se:univ} another fundamental result, which has a ``hyperbolic discrete  group taste'':   the accumulation set of the $W$-orbit of any $z\in Z$ contains $E(\Phi)$ (see Corollary~\ref{cor:acc-imc}).  We have been made aware while preparing this article that those two fractal conjectures have been solved independently for root systems of signature  $(n-1,1)$  in \cite{rank3} with a different approach, see Remark~\ref{rem:rank3}.   

In \S\ref{se:dom} we explore the question of the restriction of $E(\Phi)$ to a face $F_I$ of $\conv(\Delta)$. In particular we prove that $E_2(\Phi)$ behaves well with the restriction to standard facial subgroups (those are exactly the standard parabolic subgroups when $\Delta$ is a basis of~$V$).  By doing so, we will be brought to give a useful interpretation of \emph{the dominance order} and \emph{elementary roots}  in our affine normalized setting.
 
In \S\ref{se:univ}, we  study the geometry of $E(\Phi)$ and $Z(\Phi)$ in detail. We prove  two other fundamental results under the assumption that the root system is irreducible and neither finite nor affine. Firstly, the $W$-action on $E(\Phi)$ is faithful (Theorem~\ref{thm:faithful}).  Secondly,   the set of limit roots $E(\Phi)$ (resp., each  face of the closed imaginary cone) can be approximated arbitrarily closely
(in a Hausdorff-type metric)  by the sets of limits roots (resp., closed imaginary cones) of the universal root subsystems of $\Phi$ i.e., root systems for Coxeter groups without braid relations (which are the free objects for Coxeter groups, called universal Coxeter groups in \cite{humphreys}).  This second  result may be viewed as asserting that $\Phi$ contains ``large'' universal root subsystems. As  the subgroup of even length elements in a universal Coxeter group is a free group, this is a result   very much in the spirit of the Tits alternative for $W$.

 In \S\ref{se:questions}  we collect open  questions and discuss further avenues of research. 
 In particular, in \S\ref{sse:RelationGeo}, we discuss the relations  of  our framework with hyperbolic geometry and geometric group theory. The link with Kleinian groups, which is precisely discussed in another article~\cite{HPR}, is outlined, and we also explain for instance how the convex core of $W$ is related to the imaginary convex set and the limit roots. Considering that our framework and results apply more generally to discrete subgroups of isometries on  quadratic spaces, an important question we raise is what part of the theory of Kleinian groups and discrete subgroups of hyperbolic isometries can be generalized to quadratic spaces.

 In  a final   appendix, we discuss the relation of the set of limit roots  defined here with a notion of limit set of  a Zariski dense subgroup of the group of $k$-points of   a connected reductive group defined  over    a local field $k$  (only $k=\mathbb{R}$ here)  as  studied by Benoist in~\cite{benoist}. 

\smallskip

Sections \ref{se:fractal}, \ref{se:dom} and \ref{se:univ} can be read independently of each other. In view of the length of this article, we made sure to treat each section as a small chapter by writing a short introduction and stating the main results it contains as soon as possible.  

\subsection*{Figures} The pictures of normalized roots and the
imaginary convex body were computed with the computer algebra system
\emph{Sage} \cite{sage}, and drawn using the \TeX-package TikZ.

\subsection*{Acknowledgements} The authors wish to thank Jean-Philippe
Labb\'e who made the first version of the Sage and TikZ functions used to
compute and draw the normalized roots. The third author gave, in France in
November 2012, several seminar talks about a preliminary version of these results; he is grateful to the organizers of these seminars and
to the participants for many useful comments. In particular, he would
like to thank Vincent Pilaud for valuable discussions. 

The authors also wish to thank an anonymous referee for  valuable suggestions that improved the present manuscript, especially    \S\ref{sse:RelationGeo},  and  that led to  the Appendix.

\section{Imaginary cone, limit roots, and action of $W$}
\label{se:imc}

The aim of this section is to bring together, and slightly improve, the frameworks and results from \cite{HLR} and \cite{dyer:imc}. This  will lead  to  our first main result in \S\ref{ss:minimal}.

\medskip

Let $V$ be a real vector space of dimension $n$ equipped with a
symmetric bilinear form (inner product)~$\Mpair$.  Let $(\Phi,\Delta)$ be a
\emph{based root system} in $(V,\Mpair)$ with associated Coxeter
system $(W,S)$, i.e., $\Delta$ is a \emph{simple system}, $W$ is
generated by the set of simple reflections
${{S:=\{s_\alpha\,|\,\alpha\in\Delta\}}}$, where
 \[
s_\alpha(v)=v-2\mpair{\alpha,v}\alpha,\ \text{for } v\in V,
\]
 and $\Phi:=W(\Delta)$ is the associated \emph{root system}. The set $\Phi^+:=\cone(\Delta)\cap \Phi$ is the set of \emph{positive roots}. We recall\footnote{{\bf Note to the reader.} This  article follows directly~\cite{HLR}. In   this spirit, we chose not to rewrite in details an introduction to    based root systems. We refer the unfamiliar reader to  \cite[\S1]{HLR} for a more detailed introduction to this framework, which generalizes the classical geometric representation of Coxeter groups.} from \cite[\S1]{HLR} that a simple system  $\Delta$ is a finite subset of $V$ such that:
\begin{enumerate}[(i)]
\item $\Delta$ is positively independent: if $\sum_{\alpha \in
    \Delta} \lambda_{\alpha} \alpha =0$ with all $\lambda_\alpha \geq
  0$, then all $\lambda_\alpha=0$;
\item for all $\alpha, \beta \in \Delta$, with $\alpha \neq \beta$,
  $\displaystyle{\mpair{\alpha,\beta} \in \ ]-\oo,-1] \cup
    \{-\cos\left(\frac{\pi}{k}\right), k\in \mathbb Z_{\geq 2} \} }$;
\item for all $\alpha \in \Delta$, $\mpair{\alpha,\alpha}=1$.
\end{enumerate}
The \emph{rank of $(\Phi, \Delta)$} is the cardinality $|\Delta|$ of $\Delta$. The \emph{signature of the based root system~$(\Phi,\Delta)$} is the signature of the quadratic form $q_{\Mpair,\Delta}$ associated to 
 the restriction of $\Mpair$ to the subspace $\Span(\Delta)$. In the  case where~$\Delta$ spans $V$, the signature of $(\Phi,\Delta)$ is the signature of $q_B=B(\cdot,\cdot)$.

 Throughout the article we always assume that $\Delta$
is finite, i.e., $W$ is a finitely generated Coxeter group. Some of
the results may be extended to the general case. To lighten the
notations, we will often shorten the terminology ``\emph{based root system}''
and use ``\emph{root system}'' instead.

\subsection{Normalized roots and limit roots}
\label{ss:limitroots}

Let $V_1$ be a hyperplane that is \emph{transverse to $\PP$}, i.e.,
such that each ray $\mathbb R_{>0}\alpha$, for $\alpha\in\Delta$, intersects
$V_1$ in one point, denoted by~$\h\alpha$, see for
instance~\cite[Figure~2, Figure~3 and \S5.2]{HLR}. 

 Denote by $V_0$ the linear hyperplane directing $V_1$. For any $v\in V\setminus V_0$,
the line $\mathbb Rv$ intersects the hyperplane $V_1$ in one point,
that we denote also by $\h v$ (we also use the analog notation~$\h P$
relatively to a subset $P$ of $V\setminus V_0$).  More precisely, denote by $\vphi$
the linear form associated to~$V_0$ such that the equation of $V_1$ is
$\vphi(v)=1$. Then we have
\[ \h v= \frac{v}{\vphi(v)} ,\  \forall v \in  V\setminus V_0 .\]
For instance, if $\Delta$ is a basis for $V$, we can take for $V_1$ the
affine hyperplane spanned by $\Delta$ seen as points, so
$\h\Delta=\Delta$ and $\vphi(v)$ is simply the sum of the coordinates
of $v$ in $\Delta$ (see \cite[\S2.1 and \S5.2]{HLR} for more details).

\smallskip

Since $V_1$ is transverse to $\PP$, for any root $\rho\in \Phi$ we can
define its associated \emph{normalized root} $\h \rho$ in $V_1$. We
denote by $\h \Phi$ the set of normalized roots. It is contained in
the convex hull $\conv(\h\Delta)$ of the \emph{normalized simple roots} $\h\alpha$ in
$\h\Delta$, and it can be seen as the set of representatives of the
directions of the roots, i.e., the roots seen in the projective space
$\mathbb{P}V$. In~Figure~\ref{fig:intro}, normalized roots are the
blue dots, while the edges of the polytope $\conv(\h\Delta)$ are in
green. Note that since $\Phi=\PP \sqcup (-\PP)$, we also have
\[ \h \Phi = \h{\PP} = V_1 \cap \bigcup_{\rho \in \PP} \real_{>0}\rho.\]
The \emph{set of limit roots} is the accumulation set of $\h\Phi$:
\[
E(\Phi) = \Acc(\h \Phi) .
\]
In~Figure~\ref{fig:intro}, this is the (Apollonian gasket-like) shape to which  the blue dots tend. It is well known that $\Phi$, and therefore, $\h \Phi$, are discrete (see for instance \cite[Cor.~2.9]{HLR}); so $E(\Phi)$ is also the complement of $\h \Phi$ in the closure of $\h \Phi$. Since the elements of $E(\Phi)$ are limit points of normalized roots, we call them for short the \emph{limit roots} of $\Phi$.

In \cite[Theorem 2.7]{HLR}, it was shown that  $E(\Phi)\subseteq \h{Q} \cap \conv(\h\Delta)$, where
\[Q:=\{v\in V\,|\, \mpair{v,v}=0\}\]
 is the \emph{isotropic cone} of $\Mpair$, and $\h Q=Q\cap V_1$; $\h Q$ is represented in red in Figure~\ref{fig:intro}. Recall also that the natural geometric action of $W$ on $V$ induces a \emph{$W$-action on $\h\Phi\sqcup E(\Phi)$}:
\begin{equation} \label{eq:Wact}
w\cdot x=\h{w(x)}=\frac{w(x)}{\vphi(w(x))} \text{ for }w\in W, x\in \h\Phi\sqcup E(\Phi),
\end{equation}
where $\vphi$ is, as above, the linear form such that $\ker \vphi=V_0$ is the
direction of the transverse hyperplane $V_1$, see~\cite[\S3.1]{HLR} for more details.
 This action has a nice  geometric interpretation on $E(\Phi)$: for $\beta\in \Phi$ and $x\in E(\Phi)$ denote by   $L(\h\beta,x)$ the line in $V_1$ passing through the points $x$ and $\h\beta$, then either $s_\beta\cdot x=x$ if  $L(\h\beta,x)$ is tangent  to $\h Q$, or $s_\beta\cdot x$ is the other point of intersection of $L(\h\beta,x)$ with $\h Q$, see Figure~\ref{fig:intro}.

\medskip
Finally, it is interesting to notice that the signature of $(\Phi,\Delta)$ is intimately linked to the shape of $E(\Phi)$. If $q_{\Mpair,\Delta}$ is:
\begin{itemize}
  \item positive definite: $(\Phi,\Delta)$ is said to be of
    \emph{finite type}; in this case $\Phi$ and $W$ are finite, and
    $E(\Phi)$ is empty;
  \item positive semi-definite (and not definite): $(\Phi,\Delta)$ is said to be of \emph{affine type}; in this case $E(\Phi)$ is finite non-empty; if in addition $(\Phi,\Delta)$ is irreducible, then $\h Q$ is a singleton and $E(\Phi)=\h Q$ (see \cite[Cor.~2.15]{HLR});
  \item not positive semi-definite: $(\Phi,\Delta)$ is said to be
    of \emph{indefinite type}. In this case $E(\Phi)$ is infinite.
\end{itemize}
Some special cases of  root systems of indefinite type are the  root systems of hyperbolic type; they will be discussed in \S\ref{ss:facial}.

\smallskip
 In the following, we will write $E$ instead of $E(\Phi)$ if there is no possible confusion with more than one root system.

\subsection{The convex hull of limit roots and the imaginary cone}
\label{ss:imc}

The imaginary cone has been introduced by Kac (see \cite[Ch.~5]{kac})
in the context of Weyl groups of Kac-Moody Lie algebras: its name comes from the fact that it was 
defined as the cone pointed on $0$ and spanned by the positive imaginary roots of
the Weyl group. This notion has been generalized afterwards to
arbitrary Coxeter groups, first by H\'ee \cite{hee:imc,hee:thesis}, then
by the first author \cite{dyer:imc} (see also Edgar's thesis
\cite{edgar}).  The definition we use here applies to any finitely generated Coxeter group (see Remark~\ref{rem:SingularB} below).  

\smallskip
Let $(\Phi,\Delta)$ be a   based root system in $V$ with  associated Coxeter group $W$. The \emph{imaginary cone} $\imc(\Phi)$ of  $(\Phi,\Delta)$ is the union of the cones in the $W$-orbit of the cone
  \[
\mathcal K(\Phi) :=\{v\in \cone(\Delta)\,|\, \mpair{v,\alpha}\leq 0,
\ \forall \alpha\in\Delta\}. \]
The imaginary cone $\imc(\Phi)$ is by definition stable by the action of $W$. Observe that for each $\alpha\in \Delta$,  the reflecting hyperplane $H_\alpha=\{v\in V\,|\, B(v,\alpha)=0\}$ associated to the simple reflection~$s_\alpha$ supports a facet of $\mathcal K(\Phi)$. In~\cite[Prop.~3.2.(c)]{dyer:imc}, it is shown that  $\imc(\Phi)$ is contained in the cone $\cone(\Delta)$ and  for $x,y \in \imc(\Phi)$, $\mpair{x,y} \leq 0$. In particular, letting ${Q^-}:=\{ v\in V\,|\, B(v,v)\leq 0\}$, we have: 
\[
\imc(\Phi) \subseteq \cone(\Delta)\cap {Q^-}.
\]

The imaginary cone is intimately linked to the set of limit roots: it is proven in \cite{dyer:imc} that the closure $\ol{\imc}$ of
$\imc(\Phi)$ is equal to the convex cone spanned by the ``limit rays of roots''. These limit rays in the sense of \cite{dyer:imc} are  the rays spanned by limit roots in the sense of~\cite{HLR}  (see~\cite[\S5.6]{dyer:imc} for more details). We get from~\cite[Theorem~5.4]{dyer:imc} that the set $E$ of limit roots and the imaginary cone $\imc(\Phi)$ have the  following relation:
  \[ \ol{\imc(\Phi)}=\cone(E(\Phi)).\]
  
\smallskip

We now ``\emph{normalize}'' these notions. Let $V_1$ be an affine hyperplane transverse to~$\Phi^+$ and let
\[
 K(\Phi):=\h{\mathcal K(\Phi)}=\mathcal K(\Phi)\cap V_1\quad \textrm{and}\quad  Z(\Phi):=\h{\mathcal Z(\Phi)}=\mathcal Z(\Phi)\cap V_1.
\]
In Figure~\ref{fig:ICB} we draw two examples in rank 3, and an example in rank 4 is in Figure~\ref{fig:gasket}(b)-(c) at the end of this article. Similarly to
the case of the cone $\mathcal K$, we observe that for each $\alpha\in
\Delta$, the trace in $V_1$ of the reflecting hyperplane $H_\alpha$
associated to the simple reflection $s_\alpha$ supports a facet of
$K(\Phi)$. The converse is not always true: a facet of $K(\Phi)$ may
be rather contained in a facet of $\conv(\Delta)$, see the example on the right in Figure~\ref{fig:ICB}. Moreover,
\[
Z(\Phi)\subseteq \conv(\h\Delta) \cap {Q^-}.
\]
 and
 \begin{equation}
 \label{eq:ineqICB}
 \mpair{x,y}\leq 0, \quad\text{for }x,y\in \ol{Z}.
 \end{equation}
 
\begin{defi}
  \label{def:imc}
   The closure $\ol{Z(\Phi)}$ of $Z(\Phi)$  is called the \emph{imaginary convex body} of $(\Phi,\Delta)$.
\end{defi}
As for the set $E$, when the root system is unambiguous we will write simply $\mathcal K$, $K$, $\imc$ and $Z$ instead of $\mathcal K(\Phi)$, $K(\Phi)$, $\imc(\Phi)$ and $Z(\Phi)$.  We sometimes refer to the set $Z$ as the \emph{imaginary convex set}.  We get this normalized version of \cite[Theorem~5.4]{dyer:imc}. 

\begin{thm}
\label{thm:imc-closure}
 The convex hull of $E$ equals the imaginary convex body $\ol Z$:
\[\conv(E)=\ol{Z}.\]
\end{thm}

Using this equality, we also get the following
very nice description of  $\conv(E)$, which was mentioned in \cite[Remark 3.3]{HLR} and proved in \cite[Thm.~5.1]{dyer:imc}.

\begin{thm}
 \label{thm:imc-inter}
\[ \conv(E)= V_1 \cap \bigcap_{w\in W} w(\cone(\Delta)) .\]
\end{thm}

\begin{figure}[!h]
\begin{tabular}{cc}

\scalebox{0.6}{\input{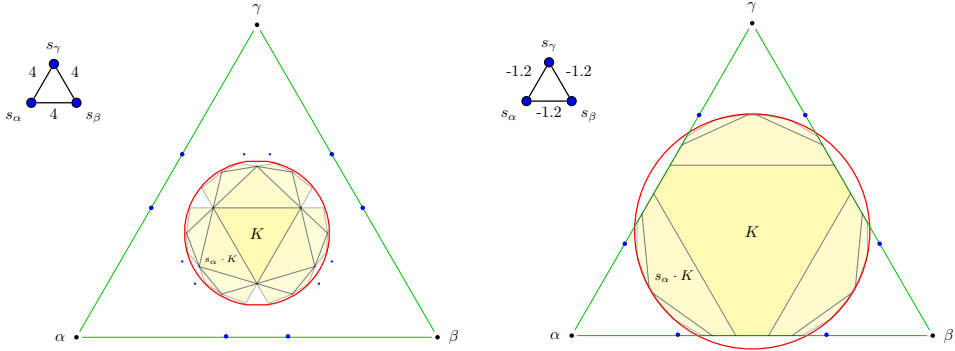}}
&
\scalebox{0.6}{\begin{tikzpicture}
	[scale=2,
	 q/.style={red,thick},
	 racine/.style={blue},
	 racinesimple/.style={black},
	 racinedih/.style={blue},
	 sommet/.style={inner sep=2pt,circle,draw=black,fill=blue,thick,anchor=base},
	 rotate=0]

\def\grosseur{0.0125}
\def\grosseursimple{0.025}

\def\grosseurdih{0.0075}



\draw[q] (2, 1.15470053837925) circle (1.30267789455786) ;

\node[label=left :{$\alpha$}] (a) at (0.000000000000000,0.000000000000000) {};
\fill[racinesimple] (0.000000000000000,0.000000000000000) circle (\grosseursimple);\node[label=right :{$\beta$}] (b) at (4.00000000000000,0.000000000000000) {};
\fill[racinesimple] (4.00000000000000,0.000000000000000) circle (\grosseursimple);\node[label=above :{$\gamma$}] (g) at (2.00000000000000,3.46410161513775) {};
\fill[racinesimple] (2.00000000000000,3.46410161513775) circle (\grosseursimple);
\draw[green!75!black] (a) -- (b) -- (g) -- (a);
\fill[racine] (1.17647058823529,0.0100000000000000) circle (\grosseursimple);
\fill[racine] (1.41176470588235,2.44524819892077) circle (\grosseursimple);
\fill[racine] (2.82352941176471,0.0100000000000000) circle (\grosseursimple);
\fill[racine] (3.41176470588235,1.01885341621699) circle (\grosseursimple);
\fill[racine] (0.588235294117647,1.01885341621699) circle (\grosseursimple);
\fill[racine] (2.58823529411765,2.44524819892077) circle (\grosseursimple);

\coordinate (ancre) at (-0.5,2.6);
\node[sommet,label=below left:$s_\alpha$] (alpha) at (ancre) {};
\node[sommet,label=below right :$s_\beta$] (beta) at ($(ancre)+(0.5,0)$) {} edge[thick] node[auto] {-1.2} (alpha);
\node[sommet,label=above:$s_\gamma$] (gamma) at ($(ancre)+(0.25,0.43)$) {} edge[thick] node[auto,swap] {-1.2} (alpha) edge[thick] node[auto] {-1.2} (beta);

\filldraw[draw= black ,fill= yellow ,opacity= 0.300000000000000 ]
(2.909090908, 1.8895099722) --
(1.0909090908, 1.8895099722) --
(0.909090909, 1.5745916434) --
(1.8181818182, 0.0) --
(2.181818182, 0.0) --
(3.09090909, 1.5745916426) --
cycle ;


\filldraw[draw= black ,fill= yellow!80 ,opacity= 0.300000000000000 ]
(0.855614973194145, 0.555738227281955)  -- 
(0.779220779101243, 1.34964998007078)  -- 
(0.909090909001093, 1.57459164340189)  -- 
(1.81818181816364, 0.000000000000000)  -- 
(1.55844155834879, 0.000000000000000)  -- 
(0.909090909200340, 0.463115189191957)  -- 
cycle ;

\filldraw[draw= black ,fill= yellow!80 ,opacity= 0.300000000000000 ]
(3.22077922077631, 1.34964997879825)  -- 
(3.14438502678509, 0.555738227066197)  -- 
(3.09090909093551, 0.463115189208216)  -- 
(2.44155844154916, 0.000000000000000)  -- 
(2.18181818163636, 0.000000000000000)  -- 
(3.09090909053703, 1.57459164166984)  -- 
cycle ;

\filldraw[draw= black ,fill= yellow!80 ,opacity= 0.300000000000000 ]
(2.90909090835412, 1.88950997158664)  -- 
(1.09090909044588, 1.88950997158664)  -- 
(1.22077922060431, 2.11445163513458)  -- 
(1.94652406417647, 2.44524819892077)  -- 
(2.05347593588235, 2.44524819892077)  -- 
(2.77922077811584, 2.11445163554275)  -- 
cycle ;


\filldraw[draw= black ,fill= yellow!60 ,opacity= 0.300000000000000 ]
(2.88170407005843, 0.197647486673624)  -- 
(3.05271199391595, 0.396955876465156)  -- 
(3.09090909093583, 0.463115189208773)  -- 
(2.44155844156772, 0.000000000000000)  -- 
(2.55161787368442, 0.000000000000000)  -- 
(2.85271933305869, 0.171898651017698)  -- 
cycle ;

\filldraw[draw= black ,fill= yellow!60 ,opacity= 0.300000000000000 ]
(1.57522826510843, 2.38457808676480)  -- 
(1.27580893667477, 2.20976589928114)  -- 
(1.22077922060375, 2.11445163513362)  -- 
(1.94652406416578, 2.44524819892077)  -- 
(1.87012987010259, 2.44524819892077)  -- 
(1.61201977942229, 2.39680518749896)  -- 
cycle ;

\filldraw[draw= black ,fill= yellow!60 ,opacity= 0.300000000000000 ]
(0.947288006310731, 0.396955876200963)  -- 
(1.11829592999777, 0.197647486608127)  -- 
(1.14728066692370, 0.171898651013539)  -- 
(1.44838212635149, 0.000000000000000)  -- 
(1.55844155853432, 0.000000000000000)  -- 
(0.909090909358289, 0.463115188918381)  -- 
cycle ;

\filldraw[draw= black ,fill= yellow!60 ,opacity= 0.300000000000000 ]
(2.72419106255821, 2.20976589972896)  -- 
(2.42477173484066, 2.38457808679453)  -- 
(2.38798022062887, 2.39680518749691)  -- 
(2.12987012986740, 2.44524819892077)  -- 
(2.05347593577540, 2.44524819892077)  -- 
(2.77922077784185, 2.11445163601732)  -- 
cycle ;

\filldraw[draw= black ,fill= yellow!60 ,opacity= 0.300000000000000 ]
(0.855614973298298, 0.555738227101555)  -- 
(0.779220779281918, 1.34964998038372)  -- 
(0.724191063174665, 1.25433571593848)  -- 
(0.722508932116326, 0.907624877326688)  -- 
(0.730315709398947, 0.869648941020573)  -- 
(0.817417876158433, 0.621897540049636)  -- 
cycle ;

\filldraw[draw= black ,fill= yellow!60 ,opacity= 0.300000000000000 ]
(3.22077922059563, 1.34964997911119)  -- 
(3.14438502668094, 0.555738226885798)  -- 
(3.18258212374129, 0.621897539719493)  -- 
(3.26968429060517, 0.869648941040586)  -- 
(3.27749106788816, 0.907624877348487)  -- 
(3.27580893678228, 1.25433571494016)  -- 
cycle ;

\draw  (2, 1.15) node{$K$};
\draw  (1.13827349116821, 0.657182506657764) node{{\small $s_{\alpha}\cdot K$}};

\end{tikzpicture}}

\end{tabular}
\caption{Two examples of pictures of $K$ and of its first images by
  the group action (in shaded yellow), giving the first steps to
  construct the imaginary convex set $Z$. The red circle is the
  normalized isotropic cone $\h Q$; in black and blue are the first
  normalized roots. The example on the left is a group of hyperbolic
  type (see \S\ref{ss:facial}): $K$ is simply a triangle and $Z$ turns
  out to be the whole open disk inside $\h Q$. The example on the
  right is weakly hyperbolic but not hyperbolic: $K$ is a truncated
  triangle and $Z$ is stricly contained in the open disk (see
  \S\ref{se:fractal}).}
\label{fig:ICB}
\end{figure}

Before extending the $W$-action on $\h\Phi\sqcup E$ to include $\ol
Z$, we discuss the affine space $\aff(\ol Z)$ spanned by the imaginary
convex body $\ol Z$. Obviously we have
\[\aff(K)\subseteq \aff(Z)\subseteq \aff(\h\Delta).\]
Moreover, we have the following particular situations. 
\begin{itemize} 
\item In the case where the root system if finite, then we know that $E=\eset$ and therefore $\ol Z=\conv(E)=\eset=\aff(\ol Z)$. 
\item  If $(\Phi,\Delta)$ is affine irreducible, then $E=\{x\}$ is a singleton and therefore $\{x\}=\conv(E)=\ol Z=Z=K$.

\item If  $(\Phi,\Delta)$ is non-affine infinite dihedral, so $\Delta=\{\al,\bt\}$, then we obtain by direct computation that $E=\{x,y\}$ is of cardinality $2$, $Z=]x,y[$ and $K\subsetneq ]x,y[$ (see also \cite{HLR} and \cite[\S9.10]{dyer:imc}). 

\item  When the root system is of indefinite type and irreducible, then  $\aff(K) = \aff(\h\Delta)$. The essential point to prove this fact is the following result (mentioned without proof in \cite[\S4.5]{dyer:imc} and which goes back to Vinberg \cite{vinberg}).
\end{itemize}

\begin{lem}
  \label{lem:sphere}
  Let $(\Phi,\Delta)$ be an irreducible based root system of indefinite
  type. Then there exists a vector $z$ in the topological interior, for the induced topology on $\Span(\Delta)$, of
  $\cone(\Delta)$ such that $\mpair{z,\alpha}<0$, for all
  $\alpha\in\Delta$.    
  Equivalently, $K$ has non-empty interior for the induced topology on
  $\aff(\h \Delta)$.
\end{lem}

We give a proof here for convenience.

\begin{proof}
  Let $\alpha_1,\dots, \alpha_p$ be the simple roots and  $A$
  be the matrix $(\mpair{\alpha_i,\alpha_j})_{1\leq i,j \leq p}$. For
  any $X={}^t(x_1,\dots, x_p)$ column matrix of size $n$, define
  $v_X=\sum_i x_i\alpha_i$. If $X\in(\mathbb R_{>0})^p$, then $v_X$ is
  in the interior of $\cone(\Delta)$. Moreover, $\mpair{v_X,\alpha}<0$
  for all $\alpha\in\Delta$ if and only if $AX\in (\mathbb
  R_{<0})^p$. Thus if we prove that there is $Z\in (\mathbb R_{>0})^p$
  such that $AZ\in (\mathbb R_{<0})^p$, the lemma follows by setting
  $z=v_Z$.

  Set $M:=I_p-A$. Since $(\Phi,\Delta)$ is irreducible, $A$ is
  indecomposable, and so is $M$. The matrix $M$ has nonnegative
  coefficients, because $A$ has $1$'s as diagonal coefficients and
  nonpositive coefficients elsewhere. So the Perron-Frobenius theorem
  implies the following two facts:
  \begin{itemize}
  \item define the spectral radius of $M$: $r:=\max \{|\lambda|,
    \lambda \in \Sp(M)\}$. Then $r\in \Sp(M)$. Moreover, since the
    root system is of indefinite type, the signature of $\Mpair$ has at
    least one $-1$. So Sylvester's law of inertia implies that $A$ has
    at least one negative eigenvalue, so $M=I_p-A$ has an eigenvalue
    that is strictly greater than $1$.  Therefore $r>1$.
  \item The eigenspace associated to the eigenvalue $r$ is a line
    spanned by a vector $Z$ of $M$ with strictly positive
    coefficients.
  \end{itemize}
  Therefore, there is $Z\in (\mathbb R_{>0})^p$ such that
  $MZ=rZ$. Hence we obtain that $AZ=(1-r)Z\in (\mathbb R_{<0})^p$,
  since $1-r<0$.
\end{proof}

We deduce easily that $K$ (and also $E$) affinely spans the same space as $\h \Delta$:

\begin{prop}
  \label{prop:aff}
  Suppose $(\Phi,\Delta)$ is an irreducible based root system of indefinite type. Then:
  \[ \aff(E)=\aff(\ol{Z})=\aff(Z)=\aff(K)=\aff(\h\Delta).\]
\end{prop}

In particular, we get $\Span(E)=\Span(\Delta)$. We will prove later
(in Theorem~\ref{thm:faithful2}) that actually any non empty open subset of $E$
is sufficient to (affinely) span $\aff(\h\Delta)$.

\begin{proof} The first equality $\aff(E)=\aff(\ol{Z})$ is straightforward using 
$\ol{Z}=\conv(E)$, and the second is clear since an affine span is closed. We also have by 
definition $\aff(K)\subseteq \aff(Z)\subseteq \aff(\h\Delta)$, so it will suffice to show that
 $\aff(K)=\aff(\h\Delta)$. By Lemma~\ref{lem:sphere}, the interior of $K$ (for the induced 
 topology on $\aff(\h \Delta)$) is not empty. So~$K$ contains an open ball of $\aff(\h \Delta)$ 
 (of nonzero radius), and $\aff(K)=\aff(\h \Delta)$.
\end{proof}

\begin{rem}
\label{rem:SingularB}
Many of the results in \cite{dyer:imc} involving the Tits cone and its
relationship to the imaginary cone require the assumption that
$\Mpair$ should be non-degenerate. For
reasons explained in \cite[\S12]{dyer:imc}, this assumption is not
necessary for results on the imaginary cone itself.  Consider a
based root system $(\Phi,\Delta)$ in $(V,\Mpair)$, with associated
Coxeter system $(W,S)$.  If $V'$ is any subspace of $V$ containing
$\Delta$ and $\Mpair$ is the restriction of $\Mpair$ to $V'$, we may
regard $(\Phi,\Delta)$ as a based root system in $(V',B')$, which we
say \emph{arises by restriction} (of ambient vector space).  The associated
Coxeter systems of these two based root systems are canonically
isomorphic and we identify them.  The definitions show that the
imaginary cones of these two based root systems are equal (as
$W$-subsets of $\Span(\Delta)$).  We also say that  $(\Phi,\Delta)$, as a based
root system in $(V,\Mpair)$, is an \emph{extension} of the based root system
$(\Phi,\Delta)$ in $(V',\Mpair')$, so the above shows that the imaginary
cone is unchanged by extension or restriction.  Similar facts apply to
the limit roots.

 Say that a based root system $(\Phi,\Delta)$ in $(V,\Mpair)$ is \emph{spanning} if 
 $V=\Span( \Delta)$  and is \emph{non-degenerate} if $\Mpair$ is non-degenerate. 
  Observe that any based root system has  a  restriction which is spanning, and also has some non-degenerate extension.
 The results we give in this paper are insensitive to restriction or extension, so, whenever convenient, we shall assume that a based root system under consideration is  non-degenerate
 (so that results from \cite{dyer:imc}  proved for non-degenerate root systems apply) or 
 spanning. Note however, we cannot always assume that it is simultaneously spanning and non-degenerate, as this would exclude affine Weyl  groups, for instance, from consideration. 
\end{rem}

\subsection{The $W$-action on the imaginary convex body}
\label{ss:WactICB}

We want now to extend the $W$-action on $\h\Phi\sqcup E$ to include the imaginary convex body $\ol Z$.  Recall that $\vphi$ denotes the linear form 
such that $\ker \vphi=V_0$ is the direction of the transverse hyperplane $V_1$.We know from \cite[\S3.1]{HLR} 
that the $W$-action on $\h\Phi\sqcup E$ defined in Equation \ref{eq:Wact} is  well defined on the set 
\[
D^+=\bigcap_{w\in W} w(V_0^+)\cap V_1,
\]
where $V_0^+$ is the open halfspace defined by  $\vphi(x)>0$. It is proven in \cite[Prop.~3.2]{HLR} that $E\subseteq D^+$. 
Since $D^+$ is convex, we have necessarily that 
\[
  \ol Z=\conv(E)\subseteq D^+.
\]  
So this $W$-action is also well-defined on $\ol Z$, and therefore $Z$. Note that $W$ acts on $\ol{Z}$  by  (restrictions of) projective transformations of $V_{1}$, and not by
  restrictions of affine maps. However, it does preserve convex  closures. 
We get therefore the following result, whose  illustration can be seen in Figure~\ref{fig:ICB}. 

\begin{prop}
  \label{rk:action}
The $W$-action from Equation \ref{eq:Wact} is an action on $\h\Phi\sqcup (E\cup \ol Z)$. More precisely:
\begin{enumerate}

\item $Z=W\cdot K$ is stable under this $W$-action;

\item $\ol Z$ is stable\footnote{This action may be identified with the restriction
  to the rays in $\ol{\imc}$ of the natural $W$-action on the set
  $\ray(V)$ of rays of $V$ (with origin $0$), as in
  \cite{dyer:imc}.} under this $W$-action.

\end{enumerate}
Moreover  the $W$-action on $\ol Z = \conv(E)$  is the restriction of projective 
transformations that preserve  convex closures, in the sense that
$\conv(w\cdot X)=w\cdot \conv(X)$ for $X\seq \ol Z$.
\end{prop}
\begin{proof} Almost all the statements follow from the previous discussion. We only need to show  that  $Z=W\cdot K$ and  the statement for convex combinations. 

We know that $\mathcal Z=W(\mathcal K)$. 
Take $z\in Z=\mathcal Z\cap V_1$, so there is $w\in W$ and $x\in \mathcal K\setminus\{0\}$ such that $z=w(x)$.  Since $\mathcal K\subseteq \cone(\Delta)\subseteq V_0^+$, $V_1$ cuts $\mathbb R x$ and therefore the normalized version $\h x$ of $x$ exists. We know that  $\h{w(x)}=\h{w(y)}$ for all nonzero $y\in \mathbb Rx$.  Since $\h x \in \mathbb Rx$ we have
\[
w\cdot \h x = \h{w\cdot  \h x}  =\h{w(x)}=\h z= z .
\]

The action is clearly by restrictions of projective transformations, as already noted.  We show the statement about convex closures.  
 It is sufficient to consider the case where $X=\set{x_{1},x_{2}}$, with $x_1,x_2\in \ol Z$.  Note that 
 \[
 \conv(X)=\mset{\lambda x_{1}+(1-\lambda)x_{2}\mid 0\leq \lambda \leq 1}.
 \]
  Let $z=  \left( \lambda x_1+(1-\lambda)x_{2}\right)\in  \conv(X)$, with $\lambda \in \real$ and $0\leq \lambda\leq 1$.
 Take  $w\in W$,  then:
  \[ w\cdot z = w \cdot \left( \lambda x_1+(1-\lambda)x_{2}\right) =  \lambda' \ w\cdot x_1+(1-\lambda')\ w \cdot x_{2} \ ,\]
  where $\lambda'=\displaystyle{\frac{\lambda\vphi(w(x_1))}{\lambda \vphi (w(x_{1}))+(1-\lambda)\vphi (w(x_{2}))}}$ by~\eqref{eq:Wact}. Since $\vphi(w(x_{i}))>0$ for $i=1,2$, we have $\lambda'\in [0,1]$ and so $w\cdot z \in \conv(w\cdot X)$. Therefore $w\cdot \conv (X) \subseteq \conv(w\cdot X)$. For the reverse inclusion we use this result with $w\cdot X$ instead of $X$ and $w^{-1}$ instead of $w$: 
$
w^{-1}\cdot \conv(w\cdot X)\subseteq  \conv(w^{-1}\cdot (w\cdot X))=   \conv(X).
$
Therefore
$
\conv(w\cdot X)=w\cdot (w^{-1}\cdot \conv(w\cdot X))\subseteq w\cdot  \conv(X),
$ which concludes the proof.
 \end{proof}

 We end this section with this  fundamental property of the convex hull of an orbit in $\ol{Z}$ shown in~\cite{dyer:imc}.
 
\begin{thm}
\label{thm:Wact}
Let $(\Phi,\Delta)$ be an irreducible  based root system.
 Then for any  $z\in
\ol{Z}$, one has $\ol{\conv(W\cdot z)}=\conv(\ol{W\cdot z})=\ol{Z}
$.
 If $z\in Z\cup (\ol{Z}\cap \h Q)$, then $\Acc(W\cdot z)\subseteq \h Q$.  In particular, $\ol{Z}$ is the only non-empty, closed,
$W$-invariant convex set contained in $\ol{Z}$.
\end{thm}

\begin{proof}
 The first two assertions are equivalent, by general facts as stated in
 \cite[A11]{dyer:imc}, to \cite[Theorem~7.5(b) and Lemma 7.4]{dyer:imc}. The third
 assertion follows from the first, and is a slightly weaker version of
 \cite[Theorem 7.6] {dyer:imc}.
\end{proof}

\section{The $W$-action on $E$ is minimal}
\label{ss:minimal}

The aim of this section is to prove  the first main result of this article:  the $W$-action on $E(\Phi)$ is \emph{minimal}, i.e., \emph{every}
$W$-orbit $W\cdot x$ in $E(\Phi)$ is dense in $E(\Phi)$.

\begin{thm} \label{cor:minimal}
  Let $(\Phi,\Delta)$ be an irreducible based root system.
  \begin{num}
  \item If $z\in \ol{Z}$, then $\ol{W\cdot z} \supseteq E$.
  \item If $x\in E$, then $\ol{W\cdot x}= E$, i.e., the action of $W$ on
    $E$ is minimal.
  \item    If  $\al\in \Phi$, then
    $ E=\Acc(W\cdot \al) = \ol{W\cdot \al}\setminus W\cdot \al.$
  \end{num}
\end{thm}

Part (b) of this theorem is a huge improvement of the only density result we had on $E(\Phi)$. In~\cite[Theorem 4.2]{HLR}, it is shown that $E(\Phi)$ is the closure of the set $E_2(\Phi)$ of the limit points obtained from dihedral reflection subgroups:  
\begin{equation}
\label{eq:E2}
E_2(\Phi):=\bigcup_{\alpha,\beta\in\Phi} L(\h\alpha,\h\beta)\cap \h Q,
\end{equation}
where $L(\h\alpha,\h\beta)$ denotes the line in $V_1$ passing through the points $\h\alpha$ and $\h\beta$. Even if Theorem~\ref{cor:minimal} is  much stronger, the density of $E_2(\Phi)$ into $E(\Phi)$ remains a very important result. Indeed, it is the main ingredient to prove Theorem~\ref{thm:Wact}, which is a main ingredient to prove  this new stronger density property.

\smallskip

While writing this article, we have been made aware that  part (b) of this theorem in the case of   root systems of signature  $(n-1,1)$, and with $\Delta$ linearly independent,  was proven in~\cite{rank3}.  See Remark~\ref{rem:rank3} for more details.  

\smallskip

A remarkable consequence of Theorem~\ref{cor:minimal}
is the fact that any orbit of the $W$-action on $\ol{Z}=\conv(E)$ can get
arbitrarily close to any face of $\ol{Z}$. Let $C$ be a convex set. 
Recall that a \emph{face} of $C$ is a convex subset $F$
of $C$ such that whenever $tc'+(1-t)c''\in F$ with $c',c''\in C$ and
$0<t<1$, one has $c'\in F$ and $c''\in F$.  

\begin{cor} \label{cor:face}
  Let $(\Phi,\Delta)$ be an irreducible  based root system.  Let $x\in
  \overline{Z}$. Then for any non-empty face $F$ of $\ol{Z}$, for any
  open subset $U$ of $\ol{Z}$ which contains $F$, the $W$-orbit of $x$
  meets $U$: $W\cdot x\cap U\neq \eset$.
\end{cor}

\begin{proof}
  Choose a point $z\in F$. Since  $\ol{Z}=\conv(E)$ we can express $z$ as a convex combination
  $\sum_{i=1}^{p}\lambda_{i}z_{i}$ of points of $E$   
  where $p>0$, $0<\lambda_{i}\leq 1$, $z_{i}\in E$, and
  $\sum_{i}\lambda_{i}=1$. 
 Fix any $i$ in $\{1,\dots,p\}$. The point
$z_i$ is in $E$, so for any $x\in \ol Z$, $z_{i}\in \ol{W\cdot x}$ by
Theorem~\ref{cor:minimal}(a). Moreover, since $F$ is a face and $z\in
F$, one also has $z_{i}\in F$.  Since $U$ is an open subset of
$\ol{Z}$ containing $F$, hence $z_{i}$, we conclude that $U\cap W\cdot
x\neq \eset$ as required. 
\end{proof}

\begin{rem}
\label{rem:irred}
~
\begin{enumerate}
\item     Suppose $(\Phi,\Delta)$ is reducible, and
  denote its irreducible components by $(\Phi_i,\Delta_i)$, for
  $i=1,\dots, p$.  Then $\Delta=\bigsqcup_i \Delta_i$,
  $\Phi=\bigsqcup_i\Phi_i$ (where $\bigsqcup$ denotes the disjoint union), 
  $\Span(\Delta)=\sum_{i}\Span(\Delta_i)$ (a sum of orthogonal subspaces, not necessarily direct) and $W=W_1\times \dots \times W_p$. We have $E(\Phi)=\bigcup_1^p E(\Phi_i)$ where the union is not necessarily disjoint (see \cite[Example 8.2]{dyer:imc}) but the sets $E(\Phi_{i})$ are pairwise orthogonal and so any limit root in $E(\Phi_i)\cap E(\Phi_j)$, where $i\neq j$, is in the radical of the restriction of  the bilinear form $B$  to  $\Span(E(\Phi_i)\cup E(\Phi_j))$.
 Each subset $E(\Phi_{i})$ of $E(\Phi)$ is $W$-invariant and   the $W$-action on $E(\Phi_{i})$ is the pullback of the natural $W_{i}$-action on this set by the projection $W\to W_{i}$.  On the other hand,  we have from \cite{dyer:imc}, $\imc (\Phi)=\imc
  (\Phi_1)+\ldots+\imc (\Phi_p)$ (sum of cones with pairwise
  orthogonal linear spans) and hence $\ol{\imc (\Phi)}=\ol{\imc
    (\Phi_1)}+\ldots+\ol{\imc (\Phi_p)}$.
    Consequently, 
    \[
    \ol{Z(\Phi)}= \conv\left(\ol{Z (\Phi_1)}\cup\ldots \cup \ol{Z(\Phi_p)}\right).  
    \]

 \item If $\Delta$ is linearly independent,  then  $E(\Phi)=\bigsqcup_1^p E(\Phi_i)$,
  the $W$-action on $E(\Phi)$ is the cartesian product of the actions of each $W_i$ on $E(\Phi_i)$ (see
  \cite[Prop.~2.14]{HLR}) and 
    \[
    \ol{Z(\Phi)}= \ol{Z (\Phi_1)}*\ldots * \ol{Z(\Phi_p)},  
    \]
    where $*$ denotes the join  of two (disjoint) spaces: 
    \[
     A*B:=\{(1-t)a + tb \ | \ a\in A, b\in B,  t\in [0,1] \}.
     \]
  \item  Part (a) of the theorem  implies that $\Acc (W\cdot z) \supseteq  E$, for $z\in \ol{Z}\sm E$ (see  
  Corollary~\ref{cor:acc-imc} for a  stronger statement).     
    
\end{enumerate}
\end{rem}

\subsection{Extreme limit roots and the proof of Theorem~\ref{cor:minimal}}

 The main ingredients for the proof of Theorem~\ref{cor:minimal} are  Theorem~\ref{thm:Wact}, as we just mentioned,   
 along with a detailed study of convexity relations between the imaginary cone and the set of limit roots. 
We show in particular that the set of extreme points of $\conv(E(\Phi))$ is dense in $E(\Phi)$. Let us introduce this result.  

\medskip

Given a convex set $C$, recall that the \emph{extreme points} of $C$ are the points in $C$ which cannot be written as a convex combination of other points of $C$, or, equivalently, the $x$ in $C$ such that $C\setminus \{x\}$ is convex. Thus, a point $c$ in $C$
is an extreme point of $C$ if and only if $\set{c}$ is a face of $C$, i.e., $c$ does not lie in the
interior of any segment with extremities in $C$. If $C$ is compact, Minkowski's theorem (finite-dimensional
Krein-Milman Theorem) asserts that the set of extreme points of $C$ is
the unique inclusion-minimal subset of $C$ with its convex hull equal
to $C$, see \cite{webster} for more details.

Denote by $\Eext(\Phi)$ (or simply $\Eext$, when there is no possible confusion) the set of extreme points of the imaginary convex
body $\ol{Z}=\conv(E)$; so $\Eext\subseteq E$ since $\ol Z$ is compact. The elements of $\Eext$ are called \emph{extreme limit roots}. 
 By Theorem~\ref{thm:Wact} we have $\ol Z = \conv(\ol{W\cdot z})$ for any $z\in \ol{Z}$, so by Minkowski's theorem we have 
 $\Eext\subseteq \ol{W\cdot z}$ for any $z\in \ol{Z}$. Thus, the statement that the closure of any orbit contains  the whole of $E$, which is Theorem~\ref{cor:minimal}(a), is a consequence of the following theorem.

\begin{thm}
  \label{thm:Eext}
Let $(\Phi,\Delta)$ be a based root system. Assume  that either $\Phi$ is irreducible or $\Delta$ is linearly independent. Then
$
\ol{\Eext(\Phi)}=E(\Phi) .
$
\end{thm}

\begin{rem}
 \label{rk:Eext}
~
\begin{enumerate}
 \item  The  equality $\Eext(\Phi)=E(\Phi)$ holds in some cases (see Corollary  \ref{cor:Eexthyp}).  
However,    the based root system in \cite[Example~5.8]{HLR} is irreducible and has linearly independent simple roots, but $\frac{\alpha+\beta
    + \delta + \epsilon}{4}\in E \setminus \Eext$.
 \item 
The following example shows the assumption in the statement of the theorem cannot be omitted.   Suppose $\Delta$ has three irreducible affine components $\set{\al_{i},\bt_{i}}$ of type $\wt A_{1}$,  for $i=1, 2, 3$, where, setting $\delta_{i}:=\al_{i}+\bt_{i}$, the space of linear relations on $\Delta$ is spanned by $\delta_{1}-\delta_{2}+\delta_{3}=0$.   Then $E(\Phi)=\mset{\h\delta_{i}\mid i=1,2,3}$ but $\ol{\Eext(\Phi)}=\Eext(\Phi)= \mset{\h\delta_{i}\mid i=1,3}$. \end{enumerate}
 \end{rem}

The proof of Theorem~\ref{thm:Eext} needs some more detailed study  of the convex geometry of $E$; we postpone it to \S\ref{ss:proofext} in order to present right now the proof of Theorem~\ref{cor:minimal}. 

\begin{proof}[Proof of Theorem~\ref{cor:minimal}]
  (a) Let $z\in \ol Z$. As explained above, $\ol{W\cdot z}$ contains the set $\Eext$ of
  extreme points of $\ol{Z}$ since $\conv(\ol{W\cdot z})=\ol{Z}=\conv(E)$ by
  Theorem \ref{thm:Wact} and Theorem~\ref{thm:imc-closure}. Hence $\ol{W\cdot z}\supseteq \ol{\Eext}=E$ by
  Theorem \ref{thm:Eext}.

  (b) For $x\in E$, the inclusion $\ol{W\cdot x}\supseteq E$ holds by (a) and the reverse inclusion holds since $E$ is closed and $W$-invariant.

  (c) For $\al \in \Phi$, since $W\cdot \al \subseteq \h \Phi$ and $\h
  \Phi$ is discrete, we get
  \[\ol{W\cdot \al}\setminus W\cdot \al=\Acc(W\cdot \al) \subseteq \Acc (\h \Phi)= E.\]
  Thus $\Acc(W\cdot \al)$ is a closed, $W$-stable subset of $E$. It is non-empty if $E\neq \eset$, so (b)  implies that it is equal to $E$.
\end{proof}
 
\medskip

\begin{rem}[What it means in the imaginary cone setting from~\cite{dyer:imc}]
~
 \begin{enumerate}
 \item Part (a) of Theorem~\ref{cor:minimal} is equivalent to the assertion that for $z\in
  \ol{\imc}\setminus \{0\}$, every limit ray of positive root rays 
  (in the space $\ray(V)$ of rays of $V$, as defined in
  \cite[5.2]{dyer:imc}) is
  contained in the closure of the set of rays in the $W$-orbit of the
  ray spanned by $z$; the special case
  (\cite[Theorem~7.5(a)]{dyer:imc}) in which $z\in {\imc}\setminus
  \{0\}$ was a key step in the proof  of Theorem~\ref{thm:Wact}.

\item Theorem~\ref{thm:Eext} corresponds,  in the setting of \cite{dyer:imc}, to the following statement: 
  the set of limit rays of positive roots is equal to the closure
  of the set of extreme rays of the closed imaginary cone $\ol{\imc}$.
    This result amounts to a new
  description of the set of limit rays of roots. More precisely, 
  Theorem~\ref{thm:Eext} is equivalent to the assertion in
  \cite[Remark 7.9(d)]{dyer:imc} that
  $\ol{\mathrm{R}_{\mathrm{ext}}}=\mathrm{R}_{0}$, which was unproved
  there.  It leads to substantial strengthenings of \cite[Corollary
    7.9(c)--(d) and Remark (2)]{dyer:imc}; for example, the inclusions in
  \cite[Corollary 7.9(c)]{dyer:imc} are actually equalities.

  \end{enumerate}
\end{rem}

\subsection{Exposed faces of $\ol{Z}$}

The aim of the rest of this section is to prove Theorem~\ref{thm:Eext}. In order to do so, we need to carefully study the convexity properties that are enjoyed by the imaginary convex body $\ol Z$. We refer to \cite[Chap.2]{webster} for more details
on convexity theory.

\medskip

A \emph{supporting half-space} of a convex set $C$ is a closed (affine) half-space
in $V$ which contains $C$ and has a point of $C$ in its boundary; the
boundary (which is an affine hyperplane) of the half-space is then
called a \emph{supporting hyperplane} of $C$.  An \emph{exposed face}
of $C$ is defined to be a subset of $C$ which is either the
intersection of $C$ with a supporting hyperplane of $C$, empty, or
equal to $C$.  Exposed faces of $C$ are faces of $C$. A face or exposed face $F$ of $C$ is
said to be proper if $F\neq C$.  It is known that any proper face of
$C$ is contained in the relative boundary $\rb(C)$ of $C$. 

Assume  for notational convenience in this subsection that $\Mpair$ is non-singular (see Remark \ref{rem:SingularB}). For any linear hyperplane $H$ of $V$ there is  $x\in V\setminus\{0\}$ such that
\[
x^\perp=\{v\in V\,|\, \mpair{v,x}=0\}=H.
\]
Since $V_1$ is an affine hyperplane that does not contain $0$, any affine hyperplane~$\mathcal H$ of $V_1$ is the intersection of $V_1$ with a linear hyperplane of $V$: there is $x\in V$ such that $\mathcal H=x^\perp\cap V_1$. In particular, any (affine) half-space
in $V$ that contains the imaginary convex body $\ol Z=\conv(E)$ has a boundary of this form. 
 The next proposition, which  refines parts of \cite[Proposition 7.10]{dyer:imc}, describes special properties of certain exposed faces of $\ol Z$.

\begin{prop}
\label{prop:isoface}
 Let $x\in \ol{Z}$ and  $F:=\ol{Z}\cap x^{\perp}$. Then:
 \begin{num}
 \item $F$ is an exposed face of $\ol{Z}$.
 \item If $U$ is an open subset of $\ol{Z}$ which contains $F$, then for some $\e>0$ one has:
\[ U \supseteq \mset{z\in \ol{Z}\mid \mpair{x,z}>-\epsilon} \supseteq
   F.\]
 \item If $x\notin \h Q$ (i.e. is non-isotropic), then $x\not\in F$, so $F$ is a
   proper face of $\ol{Z}$.
 \item  If $x\in \h Q$ (i.e. is isotropic), then $x\in F$, so $F$ is a non-empty face of $\ol{Z}$.
 \end{num}
\end{prop}

\begin{proof} 
  (a) If $x\in V^\perp$, then $F=\ol{Z}$ is an exposed face of
  $\ol{Z}$ (by convention). Assume now that $x\notin
  V^\perp$. For any $z,z'\in \ol{\imc}$, one has
  $\mpair{z,z'}\leq 0$ by Equation~\eqref{eq:ineqICB}.  This
  implies that $\ol{Z}$ is contained in the half-space $\mset{z\in
    V\mid \mpair{x,z}\leq 0}$ of $V$.  If $\ol{Z}\cap
  x^{\perp}=\eset$, then $F=\eset$ is an exposed face of $\ol{Z}$ by
  convention. Otherwise, by definition, $x^\perp$ is a supporting
  hyperplane of $\ol{Z}$, so $F=\ol{Z}\cap x^{\perp}$ is an exposed
  face of $\ol{Z}$.

  \smallskip

  (b) Let $U$ be an open subset of $\ol{Z}$ containing $F$. Since
  $\ol{Z}\sm U$ is compact, we can define $\epsilon:=
  \min(-\mset{\mpair{x,z}\mid z\in \ol{Z}\sm U})$. For $z\in
  \ol Z$, one has $\mpair{x,z}\leq 0$, and $\mpair{x,z}=0$ if and only if
  $z \in F=\ol Z \cap x^{\perp}$ (which is contained in $U$). So
  $\epsilon>0$. Then $F \subseteq \mset{z\in
    \ol{Z}\mid\mpair{x,z}>-\epsilon} \subseteq U$. Parts  (c)--(d) follow directly from the definitions. 
\end{proof}

\subsection{Extreme limit roots and exposed limit roots}
\label{ss:extpoint}

 An \emph{exposed point} of a convex set $C$ is a point $c\in C$ such
that $\set{c}$ is an exposed face of $c$. In particular, any exposed point of $c$ is an
extreme point of $C$. The converse is not true, see for example the picture
 \cite[Fig.~2.10]{webster}.  Recall  that if $C$ is compact, the set of extreme points of $C$  is the
minimal subset of $C$ whose convex hull is equal to $C$. Moreover,
Strascewicz' theorem asserts that every extreme point of $C$ is in the
closure of the set of exposed points of $C$. 

\smallskip

Let $\Eexp$ denote the set of  exposed points of $\ol{Z}=\conv(E)$: we
 call its points  \emph{exposed limit roots}.
 Since  $\ol{Z}$, is convex and compact, we have the following inclusions:
\begin{equation*}\label{eq:extpt1}
\Eexp\subseteq\Eext\subseteq \ol{\Eexp}=\ol{\Eext}, \quad \text{ and} \quad
\ol{\Eext}\subseteq E=\ol{E} \subseteq  \ol{Z}\cap \widehat{Q}.
\end{equation*}
Moreover,
\begin{equation}\label{eq:extpt2}
 \conv(\Eext)=\conv(\,\ol{\Eexp}\,)=\conv(E)=\ol{Z}.
\end{equation}

It was already known that the $W$-action preserves $E$, $Z$, $\ol Z$
and $\widehat Q\cap \ol Z$. From Proposition~\ref{rk:action}, one sees
that the $W$-action on $\ol{Z}$ sends convex sets to convex sets and
preserves $\Eexp$, $\Eext$, and $\ol{\Eexp}$. It also (obviously)
preserves the signs of inner products, in the sense that $\mpair{x,y}$
and $\mpair{w\cdot x,w\cdot y}$ have the same sign (positive, negative
or zero) for all $w\in W$ and $x,y \in \ol Z$.

\begin{lem}
  \label{lem:exp}
  Let $(\Phi,\Delta)$ be irreducible of indefinite type. Then:
  \begin{num}
  \item 
    For any $x\in \cone(\Delta)\sm\set{0}$,
    there exists $y\in E_{\exp}$ with $\mpair{x,y} \neq 0$.
    \item  
      If $x\in \ol{Z}$ in (a), then $\mpair{x,y}< 0$.
  \end{num}
\end{lem}

\begin{proof}
  (a) Let $x\in \cone(\Delta) \sm\set{0}$ be arbitrary.  By Proposition~\ref{prop:aff}, we have
   $\aff(Z)=\aff(\widehat \Delta)$.  By Equation~\eqref{eq:extpt2},
  $\aff(\Eexp)=\aff(\Eext)=\aff(\widehat \Delta)$ as well.  We claim
  that there is some $y\in \Eexp$ with $\mpair{x,y}\neq 0$.  For
  otherwise, the above would imply that $\mpair{x,\widehat\Delta}=0$
  and so $\mpair{x,\Delta}=0$ also. Since $x\in \cone(\Delta)\setminus
  \set{0}$ and $\Phi$ is irreducible of indefinite type, this is
  impossible by Lemma~\ref{lem:sphere} (see also \cite[Prop.~4.8]{HLR} or \cite[Lemma 7.1 and
    \S4.5]{dyer:imc}). For (b), since $y \in \Eexp \subseteq \ol{Z}$, if in addition $x\in
  \ol{{Z}}$, then one has $\mpair{x,y}\leq 0$ by
 Equation~\eqref{eq:ineqICB}.
\end{proof}

\begin{prop}
  \label{prop:isoface2}
  Let $x\in \ol{Z}$ and  $F:=\ol{Z}\cap x^{\perp}$. If $\Phi$ is irreducible of indefinite type and $x\in Q$ is isotropic, then $\eset\neq F\subsetneq \ol{Z}$ is a proper, non-empty, exposed  face of $\ol{Z}$.
\end{prop}

\begin{proof} 
  By Proposition~\ref{prop:isoface}(d), it suffices to  show that $F\subsetneq \ol{Z}$. By Lemma~\ref{lem:exp}(b), there exists $y\in \Eexp$ such that $\mpair{x,y}<0$. Thus $y \in \ol{Z}\sm x^\perp$, and the result follows.
\end{proof}

\begin{rem}\label{rk:extcomp}
~
\begin{enumerate}
\item   \label{rk:Eextreduc}
   It is easily seen that when $A$ and $B$ are disjoint
  convex sets, the join $A*B$ of $A$ and $B$ is convex and the set of extreme points of $A*B$ is the disjoint union
  of the set of extreme points of $A$ and of the set of extreme points of $B$. It follows directly from Remark~\ref{rem:irred} that if $\Delta$ is linearly independent, then   the
   set $\Eext(\Phi)$ of extreme limit  roots of $\Phi$ is the disjoint union of the sets $\Eext(\Phi_i)$. 
  
   \label{rk:cutEext}
 \item  As the sets $\Eext$ and $\Eexp$ can be constructed from cones,
  their properties do not depend on the choice of the transverse
  hyperplane ($V_1$) used to define the normalization map. If $H$ and
  $H'$ are two different affine hyperplanes, both transverse to $\PP$,
  denote by $\pi_H$, $\pi_{H'}$ the associated normalization maps (see
  \cite[Prop.~5.3]{HLR}), such that $\pi_H$ sends
  $\conv(\pi_{H'}(\Delta))$ to $\conv(\pi_H(\Delta))$ and $E(\Phi,H')$
  to $E(\Phi,H)$. Then $\pi_H$ also maps $\Eext(\Phi,H')$ to
  $\Eext(\Phi,H)$ and $\Eexp(\Phi,H')$ to $\Eexp(\Phi,H)$.
\end{enumerate}
 \end{rem}

\subsection{A fractal property and proof of Theorem \ref{thm:Eext}}

\label{ss:proofext}

We want to prove that the extreme limit roots (or, equivalently, the
exposed limit roots) are dense in the limit roots. Before proving Theorem~\ref{thm:Eext}, and concluding this section, 
we state a last  theorem that concisely encapsulates certain aspects of the
``fractal'' (self-similar) nature of the boundary of $\ol{Z}$ (see also \S\ref{se:fractal} for other fractal-like properties).
Only the weaker part (b) will be needed in the proof of Theorem~\ref{thm:Eext}, 
but (a) will be used in \S\ref{se:univ}.  Of course, we write $w\cdot X$ for the set $\h{w(X)}=\{w\cdot x \ | \ x \in X\}$,  where $w \in W$ and $X\subseteq V\setminus V_0$.

\begin{thm}
  \label{thm:fractact}
  Suppose that $(\Phi,\Delta)$ is irreducible of indefinite type.  
  
  Let $x\in
  E$ and $(\al_{n})$ be a
  sequence in $\PP$ such that $\widehat{\al_{n}}\to x$ as
  $n\to\infty$.
  
  Let $F:= \ol{Z}\cap x^{\perp}$ (which is a proper face of
  $\ol{Z}$ containing $x$, by Proposition~\ref{prop:isoface2}). Let $U$ be
  an open subset of $\ol{Z}$ containing $x$ and $P$ be a closed subset
  of $\ol{Z}$ such that $F\cap P=\eset$.  
 \begin{num}
 \item There exists $N\in \Nat$ such that for $n\geq N$,
   $s_{\al_{n}} \cdot P \subseteq  U$, or equivalently
   $s_{\al_{n}} \cdot U\supseteq P$.
 \item For any $z\in\ol{Z}\sm F$, one has
  $s_{\al_{n}} \cdot z \to x$ as $n\to \infty$.
 \end{num}
\end{thm}

The theorem implies the following self-similarity property: given a
point $x\in E$ and its associated face $F:= \ol{Z}\cap x^{\perp}$, if
we consider any open neighborhood of $x$ inside $\ol{Z}$ and any
closed subset of $\ol{Z}$ disjoint from $F$, we can send the latter
inside the former by the action of some element of $W$ (see Figure~\ref{fig:self}).

\begin{figure}[!h]
\begin{minipage}[b]{\linewidth}
\centering
\scalebox{0.5}{\input{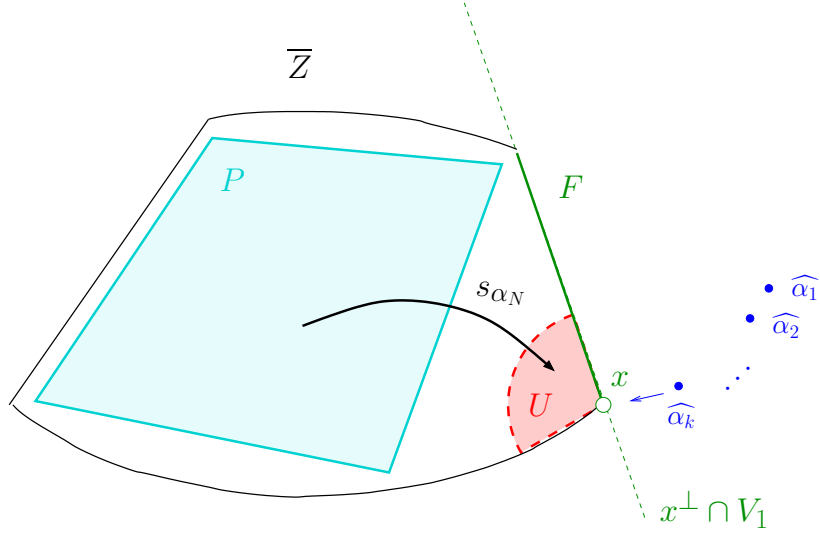}}
\end{minipage}%
\caption{A schematic visualization of Theorem~\ref{thm:fractact},
  displaying the self-similar nature of the boundary of
  $\ol{Z}$. Consider a limit root $x\in E$ (say
  $x=\lim(\h{\alpha_n})$), an open neighborhood $U$ of $x$ inside
  $\ol{Z}$ and a closed subset $P$ of $\ol{Z}$ disjoint from the face
  $F$ of $\ol{Z}$ containing $x$ ($F=\ol{Z}\cap x^\perp$). Then for
  $N$ large enough, $s_{\alpha_N}$ sends $P$ inside $U$.}
\label{fig:self}
\end{figure}

\begin{proof}
  Recall that $\vphi$ denotes the linear form such that $\h v =
  v/\vphi(v)$ for any $v \in V \setminus V_0$. Since $x\in E$, by
  definition there is a sequence $(\al_{n})$ of positive roots with
  $\widehat {\al_{n}}=\al_{n}/\vphi(\al_{n})\to x$ as $n\to\infty$.
  Since $x$ is isotropic and $\mpair{\al_{n},\al_{n}}=1$ for all $n$,
  $\vphi(\al_{n})\to \infty$ as $n\to \infty$.    Note first that the 
property (a) obviously implies (b) (take
  $P=\{z\}$). The equivalence of the conditions $s_{\al_{n}} \cdot P
  \subseteq U$ and $s_{\al_{n}} \cdot U\supseteq P$ in (a) is clear
  since $(w,z)\mapsto w\cdot z$ is a $W$-action on $\ol{Z}$ and 
$s_{\al_{n}}$ is an involution.

   To prove the inclusion $s_{\al_{n}} \cdot P \subseteq U$, it will 
suffice to
  show that $s_{\al_{n}} \cdot z\to x$ uniformly on $P$.
   Applying Proposition \ref{prop:isoface}(b) to the open subset
  $\ol{Z}\sm P\supseteq F$ of $\ol{Z}$ shows that there is some
  $\epsilon>0$ such that $P \subseteq \mset{z\in \ol{Z}\mid
    \mpair{x,z}\leq -\epsilon}$.   Since $\Mpair$ is bilinear,  the function 
$\Mpair:\conv(\h\Delta)\times P\to \mathbb R$ is uniformly continuous.
   By compactness of $P$, the function $f:\conv(\h\Delta)\to\mathbb R$ 
defined by  $f(y)= \sup_{z\in P} \mpair{y,z}$ is well-defined and continuous.
  So  $U'=f^{-1}(]-\infty,-\epsilon/2[)$ is   an open neighbourhood of 
$x$ in $\conv(\h\Delta)$ such that:
  $\forall z \in P, \forall y \in U', \mpair{y,z} \leq - \epsilon/2$.  
Moreover, for $z\in P$, one has
  $s_{\al_{n}}(z)=z-2\mpair{z,\al_{n}}\al_{n}$. Hence for $n$ large
  enough so that $\h\alpha_n=\alpha_n/\vphi(\al_{n})\in U'$, one has
  $\vphi(s_{\al_{n}}(z))=1-2\mpair{z,\al_{n}}\vphi(\al_{n}) \geq 1+
  \vphi(\al_{n})\epsilon$. So :
  \begin{equation}
      \label{cvu} \vphi(s_{\al_{n}}(z)) \to \oo \text{ as } 
n\to\infty\text{, uniformly on } P.
  \end{equation}
  Moreover,
  \[
   s_{\al_{n}} \cdot z= \frac{z}{\vphi(s_{\al_{n}}(z))} - 2 
\frac{\mpair{z,\al_{n}}}{\vphi(s_{\al_{n}}(z))}\al_{n}
  =\frac{z}{\vphi(s_{\al_{n}}(z))} +
  \frac{\vphi(s_{\al_{n}}(z))-1}{\vphi(s_{\al_{n}}(z))}\ \h{\al_{n}} \ .
  \]
  Since $P$ is compact,  $P$ is bounded and therefore,   using 
\eqref{cvu}, we have $
  \frac{z}{\vphi(s_{\al_{n}}(z))} \to 0$ and
  $\frac{\vphi(s_{\al_{n}}(z))-1}{\vphi(s_{\al_{n}}(z))}\to 1$,
  uniformly on $P$. Since $\h{\al_{n}}\to x$, we
  conclude that $s_{\al_{n}} \cdot z$ converges to the same limit as  
$\h\alpha_n$, that is converges to $x$,
  uniformly on $P$.
\end{proof}

We can now prove the density of $\Eext$ in $E$.

\begin{proof}[Proof of Theorem~\ref{thm:Eext}]
First we treat the case in which 
  $(\Phi,\Delta)$ is an irreducible root system. 
     If $\Phi$ is finite, then there are no limit roots and no extreme
  limit roots.  If $\Phi$ is affine, there is one limit root (see
  \cite[Cor.~2.15]{HLR}), and the closed imaginary cone consists of the
  ray of this root alone.  Hence the desired conclusion holds in these
  two cases.  Suppose henceforward that $(\Phi,\Delta)$ is of
  indefinite type.  Let $x\in E$.  We may choose a sequence
  $(\al_{n})$ in $\PP$ such that $\widehat{\al_{n}}\to x$ as $n\to
  \infty$.  By Lemma~\ref{lem:exp}(b), there exists $y\in E_{\exp}$
  with $\mpair{x,y}<0$. In particular, $y$ does not lie in the face
  $\ol{Z}\cap x^\perp$, so by Theorem \ref{thm:fractact}(b),
  $s_{\al_{n}}\cdot y$ converges to $x$ as $n\to \infty$. Since
  $\Eexp$ is stable by the $W$-action (see Proposition~\ref{rk:action}),
  $s_{\al_{n}}\cdot y$ lies in $\Eexp$ for any $n$. Hence $x$ is in
  the closure $\ol{\Eexp}$. Since exposed points are extreme, this
  concludes the proof in case $\Phi$ is irreducible.
  
  It remains to deal with the case in which 
  $(\Phi,\Delta)$ is a root system such that $\Delta$ is linearly independent. Denote its irreducible components by
  $(\Phi_{i},\Delta_{i})$ for $i=1,\ldots, p$. Then
  $E(\Phi)=\bigsqcup_{i} E(\Phi_{i})$, and, by  Remark \ref{rk:extcomp}(\ref{rk:Eextreduc}),
   the set of extreme limit roots satisfies  $\Eext(\Phi)=\bigsqcup_{i} \Eext(\Phi_{i})$.   This reduces the proof to the case in which the root system is irreducible, which is already known.
  \end{proof}

\section{Fractal properties}
\label{se:fractal}

We already explained in Theorem~\ref{thm:fractact} a fractal property of $E$.
 In this section, we turn our attention to  two conjectures about 
fractal descriptions of the set of limit roots $E(\Phi)$ that are stated in the prequel of this article, see
\cite[\S3.2]{HLR} and  notice Figure~\ref{fig:fractalConj} below.  We use the minimality of the
$W$-action from Theorem~\ref{cor:minimal}(b), as well as some additional
works on the case where $\Phi$ is weakly hyperbolic, to completely settle 
 \cite[Conjecture 3.9]{HLR} and, in the case of weakly hyperbolic Coxeter groups,  
 to settle the conjecture stated just above Conjecture 3.9 in \cite[\S3.2]{HLR}.

\begin{figure}[!ht]
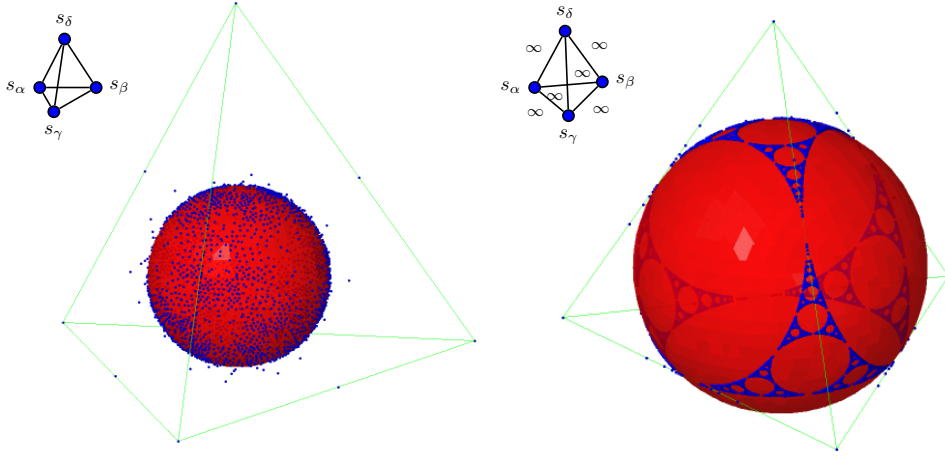

\begin{minipage}[b]{.48\linewidth}
\centering
\captionsetup{width=\textwidth}   
\scalebox{0.75}{\input{Img_3d-333333.tex}}
\end{minipage}
\hfill
\begin{minipage}[b]{.48\linewidth}
\centering
\captionsetup{width=\textwidth}   
\scalebox{0.75}{\input{Img_3d-oooooo.tex}}
\end{minipage}
\caption{The normalized isotropic cone $\h{Q}$ and the first normalized roots (with depth $\leq 8$) for two weakly hyperbolic root systems, whose diagrams are in in the upper left of each picture. In the left picture, the root system is of hyperbolic type and  $E(\Phi)$ is the whole $\h Q$, whereas in the right picture the root system is strictly weakly hyperbolic and the roots converge toward a fractal form.}
\label{fig:fractalConj}
\end{figure}

\medskip
Let $(\Phi,\Delta)$ be a based root system in $(V,\Mpair)$, with associated Coxeter group $(W,S)$.  For simplicity, we assume  throughout this section that $\Span(\Delta)=V$, and we denote by $n$ the dimension of $V$.  The interplay between $\h Q$, $\conv(\h\Delta)$ and the faces of the polytope $\conv(\h\Delta)$ is  at the heart of the fractal properties of $E$. 

 Let us recall the existing link between the faces of $\conv(\h\Delta)$ and some standard parabolic subgroups of $(W,S)$. Recall that a \emph{standard parabolic subgroup} of $W$ is a subgroup
$W_I$ of $W$ generated by a subset $I$ of $S$. It is well known that:
\begin{itemize}
\item $(W_I,I)$ is a Coxeter system;
\item $(\Phi_I,\Delta_I)$ is a  based root system in\footnote{Note that $(\Phi_I,\Delta_I)$ can also be seen as a based root system in $(V,\Mpair)$, since we do not require that the simple roots generate the whole space in the definition of based root system.} $(V_I,\Mpair_{|V_I})$ with associated Coxeter group $W_I$, where:
\[
\Delta_I:=\{\alpha\in\Delta\,|\, s_\alpha\in I\},\quad \Phi_I:=W_I (\Delta_I)\quad\textrm{and}\quad V_I:=\Span(\Delta_I).
\]
\end{itemize}
This allows us  to define easily the subset $E(\Phi_I)$ of $E(\Phi)$, consisting of the accumulation points of $\h{\Phi_I}$ (see \cite[\S5.4]{HLR}).

\medskip

We say that $I\subseteq S$ (or, $\Delta_I \subseteq \Delta$) is
\emph{facial for the based root system $(\Phi,\Delta)$} if
$\conv(\h{\Delta_I})$ is a face of the polytope $\conv(\h\Delta)$. The
corresponding standard parabolic subgroup~$W_I$ is then called a
\emph{standard facial subgroup for $(\Phi,\Delta)$}, and $(\Phi_I,\Delta_I)$ is called a \emph{facial root subsystem}.  
If $\Delta$ is a basis for $V$ then obviously any standard parabolic subgroup is a
standard facial subgroup.  But if $(\Phi,\Delta)$ is, for instance, a rank~$4$
based root system in $V$ of dimension $3$ (see for example
\cite[Example~5.1]{HLR}), then the subsets of $S$ corresponding to the
diagonals of the quadrilateral $\conv(\h\Delta)$ are obviously not
facial.

\begin{rem}
 The  notion of standard   facial reflection subgroup which we use 
 in this paper differs  from that in \cite{dyer:imc}. Their relationship  may be characterized 
 as follows:  the  family of standard  facial reflection subgroups as defined in this paper is unchanged,  as a family of  subgroups of $W$,
 by  extension or restriction of ambient vector space (as in Remark \ref{rem:SingularB})  and coincides 
 with that in \cite{dyer:imc}  for  based root systems $(\Phi, \Delta)$ in $(V,\Mpair)$ for which $\Mpair$ is non-singular.  
  We refer the reader 
to~\cite[\S2,\S8]{dyer:imc} for more details and properties of standard facial subgroups.
\end{rem}

\subsection{Facial subgroups, hyperbolicity and a self-similar dense subset of the set of limit roots}
\label{ss:facial}

We first completely characterize the root systems that have the same property as the one in the left picture in Figure~\ref{fig:fractalConj} shows, i.e., such that $\h Q = E$. This settles \cite[Conjecture 3.9(i)]{HLR}.  
 
\begin{defi}  
  \label{rk:def-hyp}
  We say that~$(\Phi,\Delta)$ is \emph{weakly hyperbolic} if the signature of the bilinear form $\Mpair$ is  $(n-1,1)$, where $n$ is the dimension of $\Span(\Delta)$.  We say that a weakly hyperbolic based root system $(\Phi,\Delta)$ is \emph{hyperbolic} if  every proper  facial root subsystem of $(\Phi,\Delta)$ has all its irreducible components  of finite or affine type.
\end{defi}

\smallskip

\begin{rem}
  ~
  \begin{enumerate}
  \item When $|\Delta|=2,3$, then any root system $(\Phi,\Delta)$ of
    indefinite type is weakly hyperbolic. In higher ranks, there are
    still many families of weakly hyperbolic root system. For example,
    this is the case when all the inner products
    $\mpair{\alpha,\beta}$ are the same (and non-zero) for any
    $\alpha\neq \beta \in \Delta$ (see the examples of
    Figure~\ref{fig:fractalConj}). In particular, any universal
    Coxeter group (where the labels of the Coxeter graph are all
    $\oo$) can be associated with a root system of weakly hyperbolic
    type. It is not true in general that all the based root systems
    associated to any universal Coxeter groups are of signature
    $(n-1,1)$; see Figure~\ref{fig:nonhyp} for such an example in rank
    $4$ with signature $(2,2)$, see also \cite[Example~1.4]{dyer:imc}.
  \item If $(\Phi,\Delta)$ is weakly hyperbolic and reducible, then
    all but one of its irreducible components are of finite type, and
    the remaining one is weakly hyperbolic. Also, if $(\Phi,\Delta)$
    is hyperbolic, then it is irreducible.  For details on these
    different notions of hyperbolicity, we refer to
    \cite[\S9.1-2]{dyer:imc} and the references therein.
\end{enumerate}

\end{rem}
\begin{thm}
\label{thm:Qincl}
Assume $(\Phi,\Delta)$ is irreducible of indefinite type. Then the following properties are equivalent:
\begin{enumerate}[(i)]
 \item $(\Phi,\Delta)$ is hyperbolic;
  \item $\h Q \subseteq \conv(\h\Delta)$;
  \item $E(\Phi)=\h Q$.
 \end{enumerate}
\end{thm}

The proof of this theorem is postponed to \S\ref{ss:fractal1}, but 
we use the theorem now to  explain the qualitative appearance of  the right picture of Figure~\ref{fig:fractalConj}. 
 The idea of Conjecture~3.9 in \cite{HLR} is to describe $E(\Phi)$ 
 by acting with $W$ only on the limit roots of parabolic root 
 subsystems $\Phi_I$ such that $
 \h Q \cap \Span (\Delta_I) \subseteq \conv (\h{\Delta_I})$. By 
 Theorem~\ref{thm:Qincl}, we know that in this case  
 $E(\Phi_I)=\h{Q_I}$. This will explain why the  set of limit 
 roots in Figure~\ref{fig:fractalConj}, or in Figure~\ref{fig:intro},
  looks like a self-similar union of circles.
  
 In general, say that a subset  $\Delta_I\subseteq \Delta$ is generating if $\h{Q_I}:=\h Q \cap \Span (\Delta_I) \subseteq \conv (\h{\Delta_I})$.   Denote the set of  irreducible generating subsets $\Delta_I$ of $\Delta$ such that $\Phi_I$ is not finite as
\[ \Gen(\Phi,\Delta)=\{ \Delta_I\subseteq \Delta \ | \ (\Phi_I,\Delta_I) \text{ is irreducible and } \eset \neq \h{Q_I} \subseteq \conv(\Delta_I)\}\] 
(note that  $\Phi_I$ is infinite if and only if $\h{Q_I} \neq \eset$). 
 Using Theorem~\ref{thm:Qincl}, we have:
 \begin{equation*} \begin{split}\Gen(\Phi,\Delta)&=\{ \Delta_I\subseteq \Delta \ | \ (\Phi_I,\Delta_I) \text{ is irreducible and of hyperbolic or affine type}\}\\
&= \{ \Delta_I\subseteq \Delta \ | \ (\Phi_I,\Delta_I) \text{ is irreducible and }\h{Q}_{I}=E(\Phi_{I})\neq \eset. \}\end{split}\end{equation*}
\smallskip

This theorem settles the discussion at the end of section 2.2 in
\cite{HLR}, and proves \cite[Conjecture~3.9(i)]{HLR}.  Indeed, one
easily sees that if $\Delta$ is linearly independent\footnote{In
  \cite{HLR}, it is assumed that $\Delta$ is a basis of $V$; in this
  case any parabolic based root subsystem is facial.}, a subset
$\Delta_I\subseteq \Delta$ is generating if and only if all its
components are generating and it has at most one component of infinite
type; this implies by the above that if $\Delta_{I}$ is generating,
one has $\h{Q}_{I}=E(\Phi_{I})$.

The second item of the next corollary, which is a consequence of Theorem~\ref{cor:minimal} and Theorem~\ref{thm:Qincl}, settles  \cite[Conjecture~3.9(ii)]{HLR}.

\begin{cor}[{\cite[Conj.~3.9]{HLR}(ii)}]
\label{cor:fractal1}
 Let $(\Phi,\Delta)$ be an irreducible root system in
 $(V,\Mpair)$. Then:
\begin{enumerate}[(i)]
\item $\Gen(\Phi,\Delta)$ is empty if and only if $\Phi$ is finite;
\item the set $E$ is the topological closure of the subset $\mathcal F_0$ of
  $\h Q$ defined by:
    \[\mathcal F_0:=W\cdot\left( \bigcup_{\Delta_I\in \Gen(\Phi,\Delta)} \h{Q_I}\right).
    \]
\end{enumerate}
\end{cor}

In the example of right-hand side of Figure~\ref{fig:fractalConj}, the set
$\mathcal F_0$ is the self-similar fractal constituted by the circles
appearing on the facets (which are the $\h{Q_I}$'s) and all their
$W$-orbits, the first of them are the smaller visible circles; thus
$\mathcal F_0$ is an infinite (countable) union of circles, similar to
an Apollonian gasket drawn on a sphere; see also the example of
Figure~\ref{fig:gasket}(a), where the only generating subsets of
$\Delta$ are $\{\alpha,\beta,\gamma\}$ and
$\{\alpha,\gamma,\delta\}$, which produce, respectively, the single
  limit root of the bottom face and the ellipse of limit roots of the
  left face of the tetrahedron.

\medskip

\begin{rem} 
  ~
  \begin{enumerate}
  \item Define a based root system $(\Phi,\Delta)$ to be 
    \emph{compact hyperbolic} if every proper  facial root subsystem 
    of $(\Phi,\Delta)$ has all its irreducible components  of finite  
    type.   The notion of hyperbolic  (resp., compact hyperbolic)  
    Coxeter group defined in \cite[\S6.8]{humphreys} corresponds in 
    our setting to a Coxeter group with a hyperbolic (resp., compact 
    hyperbolic)  based root system such that  the simple system is a 
    basis of $V$. Using the theorem and its corollary, it is easy to 
    deduce the following addition: an irreducible root system of indefinite type 
    $(\Phi,\Delta)$ is compact hyperbolic  if and only if 
    $\h Q \subseteq \ri(\conv(\h\Delta))$ where 
    $\ri(X)$ denotes relative interior of $X$, see for instance \cite[Appendix A]{dyer:imc} where it is denoted by ``ri$(X)$''.
    
  \item   By looking carefully at 
    the proof below, we note that the corollary still holds if we 
    replace $\Gen(\Phi,\Delta)$ by the set  $\Gen'(\Phi,\Delta)$ of
    $\Delta_I\subseteq \Delta$ such that $(\Phi_I,\Delta_I)$ is 
    irreducible of affine type or compact hyperbolic type.  
  \end{enumerate}
\end{rem}

\begin{figure}
\captionsetup{width=\textwidth} 
\begin{minipage}[b]{\linewidth}
\centering
\scalebox{1}{\input{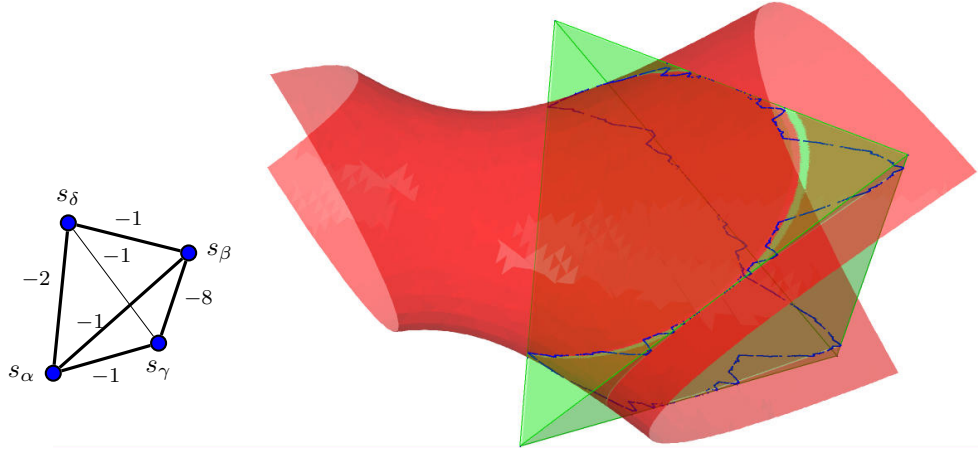}}
\end{minipage}%

\caption{A based root system which is not weakly hyperbolic, but whose
  associated Coxeter group is the universal Coxeter group of rank
  $4$. The inner products between the simple roots are indicated in
  the diagram on the left. The signature of the bilinear form is
  $(2,2)$, and $\h Q$ is a hyperboloid of one sheet. The normalized
  roots are drawn until depth $8$. Note that the geometry of the limit
  shape $E$ looks very different from the case of weakly hyperbolic
  root systems.}
\label{fig:nonhyp}
 \end{figure}

\begin{proof}[Proof of Corollary~\ref{cor:fractal1}]
(i) Clearly $\Gen(\Phi,\Delta)$ is empty when $\Phi$ is
  finite. Suppose now that $\Phi$ is irreducible infinite. We prove
  that $\Gen(\Phi,\Delta)\neq \eset$ by induction on the rank of the
  root system. When $(\Phi,\Delta)$ has rank 2, $\Delta\in
  \Gen(\Phi,\Delta)$. Suppose now that the property is true until some
  rank $n-1\geq 2$, and take $(\Phi,\Delta)$ irreducible infinite of
  rank $n$. If $(\Phi,\Delta)$ has a proper facial root subsystem
  $\Phi_I$ which is infinite, we can conclude by applying the induction
  hypothesis on one infinite irreducible component of
  $\Phi_I$. Suppose now that every proper facial root subsystem is
  finite. Then the positive index in the signature of $\Mpair$ is
  at least $n-1$, so $(\Phi,\Delta)$ is (irreducible) of finite,
  affine, or weakly hyperbolic type. It is not finite by hypothesis,
  and if it is affine, then $\Delta$ is in $\Gen(\Phi,\Delta)$. If
  $(\Phi,\Delta)$ is weakly hyperbolic, since every proper facial
  root subsystem is finite, $(\Phi,\Delta)$ is hyperbolic by
  definition (it is even compact hyperbolic), so $\Delta$ is in
  $\Gen(\Phi,\Delta)$.

(ii) If $\Delta_I\in \Gen(\Phi,\Delta)$, then $\h{Q_I}=E(\Phi_I)\subseteq
  E(\Phi)$: indeed, the affine case is clear (see \cite[Cor.~2.15]{HLR}), and
  the hyperbolic case is one implication of
  Theorem~\ref{thm:Qincl}. Moreover $\h{Q_I}\neq \eset$, so by (i), $\mathcal F_0$
  is not empty and contained in $E$. It is also stable by $W$, and
  from Corollary~\ref{cor:minimal}(b), the orbit of any point of $E$ is dense
  in $E$, so $\mathcal F_0$ is dense in $E$.
\end{proof}

\subsection{Fractal description of $E$ using parts of the isotropic cone $\h Q$}
\label{ss:fractal2}

In this subsection, we study the  conjecture mentioned in \cite[\S3.2]{HLR} before Conjecture~3.9. We start by constructing a natural subset $\mathcal F $ of $\h Q$ by removing the parts of $\h Q$ that cannot belong to $E$ for straightforward reasons\footnote{We use the letter $\mathcal F $ for the sake of consistency
with \cite{HLR}, and because $\mathcal F $ can be thought of as a fractal set.}. We conjecture below that $\mathcal F $ is actually equal to $E$, and prove this conjecture when the root system is weakly hyperbolic.

\smallskip

The construction of the set $\mathcal F $ was roughly described in \cite[\S3.2]{HLR}. Here we need a more precise definition, taking care of the fact that the $W$-action is not defined everywhere on $V_1$. Let $(\Phi,\Delta)$ be an irreducible   root system. We denote by $D$ the part of $V_1$ on which $W$ acts (see \cite[\S3.1]{HLR}):
\[ D= V_1 \cap \bigcap_{w\in W} w(V\setminus V_0).\]
Is clear that the domain $D^+$ considered in \S\ref{ss:WactICB} is contained in $D$.

\begin{defi}
Denote by $\hQact:=\h Q \cap D$ the part of $\h Q$ where $W$ acts. Then 
\[ \mathcal F:=\hQact \setminus \bigcup_{w\in W} w\cdot (\hQact \setminus \hQact \cap \conv(\h\Delta)).\]
\end{defi}

The idea of the construction of $\mathcal F $ is actually mostly
naive. We know the limit roots are in $\hQact=\h Q \cap D$, but there
cannot be any limit root outside $\conv(\h\Delta)$. Since $E$ is
$W$-stable, there cannot be any limit root also in the orbits of the
parts of $\h Q$ that are outside $\conv(\h\Delta)$. We remove all these
parts from $\h Q$ to construct $\mathcal F $.  In the example pictured
in Figure~\ref{fig:fractalConj}(right), $\mathcal F $ is the complement, in the
red sphere, of the union of the open spherical caps associated to the
circles used to define $\mathcal F_0$. Similarly, in
Figure~\ref{fig:gasket}(a), $\mathcal F$ is the complement of the
union of the open ellipsoidal caps which are the images of the one cut
by the left face.

We give in Proposition~\ref{prop:F} other characterizations of
$\mathcal F$, in particular $\mathcal F=\conv(E)\cap \h Q$.

\begin{rem} 
As explained in \cite[\S3.2]{HLR}, $E$ is contained in $\mathcal F $. Indeed, it is contained in $Q$ (\cite[Theorem~2.7]{HLR}), contained in $\conv(\h\Delta)$ (clear), contained in~$D$ and stable by $W$ (\cite[Prop.~3.1]{HLR}).
\end{rem}

\begin{question}
\label{conj:fractal2}
For any irreducible root system, is $E$ equal to $\mathcal F $ ?
\end{question}

We are able to answer this question for any weakly hyperbolic root system, using the specific geometry of $\h Q$ in this case. We do not know if this result extends to more general root
systems. 

\begin{thm}
\label{thm:fractal2}
Let $(\Phi,\Delta)$ be an irreducible root system. Assume that
it is \emph{weakly hyperbolic}. Then:
\begin{enumerate}[(i)]
\item $\mathcal F = E=\h Q \cap \conv(E)$;
\item $E$ is the unique non-empty, closed subset of $\hQact \cap \conv(\h\Delta)$, which is stable by $W$.
\end{enumerate}
\end{thm}

\begin{rem}
  \label{rem:rank3}
  ~
  \begin{enumerate}
  \item In \cite{rank3}, Higashitani, Mineyama and Nakashima prove both conjectures of \cite[\S3.2]{HLR} in the case of rank $(n-1,1)$ root systems with $\Delta$ linearly independent. They also obtain, as a by-product of their work, Theorem~\ref{cor:minimal}(b) under these assumptions.  Their proof is different from ours and based on a careful analysis of the limit roots  by looking at~$\h Q$ as a metric space $(\h Q, d_\Mpair)$. It would be interesting to see which part of their work could be generalized to based root systems of arbitrary ranks that are not necessarily weakly hyperbolic. 
  \item In the particular case where the root system is \emph{hyperbolic}, the theorem is implied by Theorem~\ref{thm:Qincl}, and $E=\mathcal F=\h Q$ is  homeomorphic to an $(n-2)$-sphere (see Remark \ref{rk:Fhyp}).
\end{enumerate}
\end{rem}

The proof of Theorem \ref{thm:fractal2} is postponed to \S\ref{ss:fractal2.1}.
The theorem implies that $\Acc (W\cdot z)\seq E$ for $z\in Z$ if $W$ is weakly hyperbolic; we show equality in  Corollary 
\ref{cor:acc-imc}.

\begin{rem}
  \label{rk:klein}
  Suppose that the rank of $\Phi$ is $3$ or $4$ and $\Phi$ is weakly hyperbolic.  In this case, we  can always describe $E$ using Theorem \ref{thm:fractal2}. The  Coxeter group~$W$ acts on $E$ as a group generated by hyperbolic
  reflections, so can be seen as a Fuchsian or Kleinian group. Using
  this point of view, the set $E$ is no other than the limit set of a
  Kleinian group, which explains the shape of Apollonian gasket
  obtained in \cite[Fig.~9]{HLR}. These relations are explored in the general context of Lorentzian spaces in the article \cite{HPR}, and are outlined in \S\ref{sse:RelationGeo} at the end of this article.
\end{rem}

\subsection{The normalized isotropic cone for weakly hyperbolic groups}
\label{ss:sphere}

Before proving Theorem~\ref{thm:Qincl} and Theorem~\ref{thm:fractal2}, we first describe what $\h Q$ looks like when the root system $(\Phi,\Delta)$ is weakly hyperbolic. Recall that we assume (without loss of generality) that $\Span (\Delta)=V$ and denote by $n$ the dimension of $V$, which is smaller than or equal to the rank $|\Delta|$ of $(\Phi,\Delta)$. Recall that a hyperplane of $V$ is said to be \emph{transverse to $\PP$} if for any $\rho \in \PP$, $H$ intersects the ray $\real_{>0}\rho$ in one point. We prove below that we can find a hyperplane $H$ of $V$, transverse to $\PP$, and such that $Q\cap H$ is an $(n-2)$-dimensional sphere. This will be useful in the following subsections.

\begin{prop}
  \label{prop:sphere}
  Let $(\Phi,\Delta)$ be an irreducible based root system  in
  $(V,\Mpair)$ of dimension $n$. Suppose that $(\Phi,\Delta)$ is weakly
  hyperbolic. Then there exists a basis $(e_1,\dots, e_n)$ for $V$
  such that:
  \begin{enumerate}[(i)]
  \item the restriction of $\Mpair$ to the hyperplane
    $H_0:=\Span(e_1,\dots, e_{n-1})$ is positive definite;
  \item the matrix of $\Mpair$ in this basis is
    $\Diag(1,\dots, 1,-1)$;
  \item $Q$ intersects the affine hyperplane $H:=e_n+H_0$ in an
    $(n-2)$-dimensional sphere; if $x_1,\dots, x_n$ are the
    coordinates in the basis,
\[ Q\cap H = \{(x_1,\dots, x_{n-1},1) \, | \, x_1^2+\dots +
  x_{n-1}^2=1 \}\ ;\]
 \item the vector $e_n$ satisfies $\mpair{e_n,\alpha}<0$, for all
   $\alpha\in\Delta$;
\item $H$ is transverse to $\PP$.
\end{enumerate}
\end{prop}

\begin{rem}
  In the case $(\Phi,\Delta)$ is weakly hyperbolic but not
  irreducible, then only one of its irreducible components (say
  $(\Phi',\Delta')$) is infinite, and $(\Phi',\Delta')$ is weakly
  hyperbolic (see Remark~\ref{rk:def-hyp}). So $Q$ lives in $\Span
  (\Delta')$ and Proposition~\ref{prop:sphere} implies that there is a
  transverse hyperplane $H$ such that $Q\cap H$ is a sphere of
  dimension $\dim(\Span (\Delta'))-2$.
\end{rem}

Note that items (i)-(iii) are straightforward. Indeed, as the
signature of $\Mpair$ is $(n-1,1)$, there exists a basis such
that the equation of $Q$ is $x_1^2+\dots + x_{n-1}^2-x_n^2=0$, and $Q$
is a conical surface on $\mathbb S^{n-2}$. The fact that $H$ is
transverse to $\PP$ (item (v)) is not direct and will follow from
(iv) and Lemma~\ref{lem:sphere}. 

\begin{proof}
  Let $(\Phi,\Delta)$ be an irreducible root system of weakly
  hyperbolic type. Take~$z$ as in Lemma~\ref{lem:sphere}. Since
  $\mpair{z,\alpha}<0$ for all $\alpha \in \Delta$, and $z$ is
  in\footnote{In this proof we do not need the stronger statement that
    $z$ is in the topological interior of $\cone(\Delta)$.}
  $\cone(\Delta)\setminus\{0\}$, we have $\mpair{z,z}<0$. Denote
  $e_n:=z/\sqrt{-\mpair{z,z}}$, so that $\mpair{e_n,e_n}=-1$, and
  $e_n$ satisfies item (iv).

  Since the signature of $\Mpair$ is $(n-1,1)$, we can complete
  $\{e_n\}$ in a basis $(e_1,\dots, e_n)$ such that the matrix of
  $\Mpair$ in this basis is $\Diag(1,\dots, 1,-1)$.  Let $H_0$ be the
  hyperplane spanned by $e_1,\dots, e_{n-1}$. The restriction of
  $\Mpair$ to $H_0$ is positive definite, so $(H_0, \Mpair_{|H_0})$ is a
  Euclidean plane with $(e_1,\dots, e_{n-1})$ as an orthonormal
  basis. Note that $\mpair{e_n,v}=0$ for any $v\in H_0$. Consider
  $H:=e_n+H_0$ the affine hyperplane directed by $H_0$ and passing
  through the point $e_n$.

  This proves items (i) and (ii). Item (iii) follows since the
  equation of $Q$ in the chosen basis is $x_1^2+\dots +
  x_{n-1}^2-x_n^2=0$, so:
  \[ Q\cap H = \{(x_1,\dots, x_n) \, | \, x_n=1 \text{ and } x_1^2+\dots +
  x_{n-1}^2=1 \} . \]

  We are left to proving item (v), that is, $H$ is transverse to $\PP$:
  we have to show that $\mathbb R_{>0}\alpha\cap H$ is nonempty for
  all $\alpha \in \Delta$. Since $0\notin H$, the line $\mathbb
  R\alpha$ cannot be contained in $H$, and therefore neither can
  $\mathbb R_{>0}\alpha$. Assume by contradiction that $\mathbb
  R_{>0}\alpha\cap H=\eset$ for some $\alpha \in \Delta$. We have two
  cases:
  \begin{itemize}
  \item either $\mathbb R_{>0}\alpha$ is in $H_0$ and so is the line
    $\mathbb R\alpha$. Therefore $\alpha\in H_0$. So
    $\mpair{\alpha,e_n}=0$ contradicting $\mpair{\alpha,e_n}<0$;
  \item or $\mathbb R_{<0}\alpha\cap H$ is a point $\lambda \alpha =
    e_n+u$, with $\lambda<0$ and $u\in H_0$. Since $\mpair{e_n,u}=0$
    we have $ -1= \mpair{e_n,e_n}= \mpair{\lambda\alpha-u,e_n}=
    \lambda \mpair{\alpha,e_n}$. But $\lambda<0$ and
    $\mpair{\alpha,e_n}<0$, so we get a contradiction.
  \end{itemize}
  Thus, overall, $\mathbb R_{>0}\alpha\cap H$ must be a point.
\end{proof}

Proposition \ref{prop:sphere} will be used in the two following
subsections. We can already deduce an interesting consequence on the
set $\Eext$ of extreme points of $\ol{Z}=\conv(E)$: as they live on a
sphere, no limit root can be written as a convex combination of other
limit roots.

\begin{cor}
\label{cor:Eexthyp}
Let $(\Phi,\Delta)$ be an irreducible based root system. If
$(\Phi,\Delta)$ is weakly hyperbolic, then $\Eext(\Phi)=E(\Phi)$.
\end{cor}

After Theorem~\ref{thm:Eext}, it is always true that $\ol{\Eext}=E$. The equality $\Eext=E$ is not always valid (see Remark~\ref{rk:Eext}(3)). We do not know  how to characterize  root systems such that $\Eext=E$.

\begin{proof}
  From Proposition~\ref{prop:sphere}, we can choose the transverse hyperplane $V_1$ such that $\h Q=Q\cap V_1$ is a sphere. This does not change the properties of $E$ and $\Eext$, as explained in Remark~\ref{rk:cutEext}. Suppose there exists $x\in E \setminus \Eext$. Then $x$ is a linear combination with positive coefficients of points $x_1,\dots, x_p$ in $E$ (with $p>1$). Since $E\subseteq \h Q$, $x_1,\dots, x_p$ lie on the sphere $\h Q$. So $x$ can not be in $\h Q$ (since every point in $\h Q$ is an extreme point of the ball $\conv(\h Q)$), which contradicts the inclusion $E\subseteq \h Q$.
\end{proof}

\subsection{Proof of Theorem~\ref{thm:Qincl}}
\label{ss:fractal1}

To prove Theorem~\ref{thm:Qincl} we will use the following (technical but
elementary) lemma, which answers the question: when can $\h Q$ be
bounded? Note that the letters $Q$ and $B$ below are specific to the
lemma and its proof, and more general than in the framework of the
article.

\begin{lem}
\label{lem:bounded}
Let $B$ be a symmetric bilinear form on a $n$-dimensional vector space~$V$ (over a field of characteristic $0$). Define the associated quadric
\[ Q:=\{v\in V \ | \ B(v,v)=0\}.\]
Let $H$ be an affine (nonlinear) hyperplane in $V$.  If $Q\cap H$ is
bounded, then we have one of the following:
\begin{itemize}
\item $B$ is positive (definite or not);
\item $B$ is negative (definite or not);
\item $B$ has signature $(n-1,1)$ or $(1,n-1)$.
\end{itemize}
 More precisely, if $B$ does not satisfy any of these conditions, then
$Q\cap H$ contains an affine line.
\end{lem}

\begin{rem}
  \label{rk:bounded}
  ~
  \begin{enumerate}
  \item The converse is obviously false: for example when the signature of $B$ is $(n-1,1)$, $Q$ is a cone on a sphere, so the boundedness of $Q\cap H$ depends on the choice of the hyperplane $H$.
  \item Of course $Q$  is unchanged if we replace $B$ by $(-B)$, so if we assume the signature of $B$ to be $(p,q)$ with $p\geq q$, the lemma is equivalent to
    \[ Q\cap H \text{ bounded } \implies q=0 \text{ or }(p,q)=(n-1,1).\]
  \end{enumerate}
\end{rem}

\begin{proof}
Denote $\sgn B=(p,q)$, and set $r=n-(p+q)$. Let us choose an adapted basis for $V$ such that $X=(x_1,\dots, x_p,y_1,\dots, y_q, z_1,\dots, z_r) \in Q\cap H$ if and only if:
\begin{equation}
\left\{
\begin{alignedat}{11}
x_1^2 & {}+{} & \dots & {}+{} & x_p^2 \ & {}-{} & y_1^2 &{}-{} & \dots &{}-{} & y_q^2  &&&&&& & {}={} & 0 & \qquad & (Q)\\
a_1 x_1 &{}+{} & \dots &{}+{} & a_p x_p \ &{}+{} & b_1 y_1 &{}+{} & \dots &{}+{} & b_q y_q \ &{}+{} & c_1 z_1 &{}+{} & \dots &{}+{} & c_r z_r & {}={} &1 & \qquad & (H)
\end{alignedat}
\right.
\label{eq:Q}\end{equation}
where $(a_1,\dots, b_1,\dots, c_1,\dots)\neq (0,\dots,0)$. Suppose
that $Q\cap H$ is bounded. For simplicity we assume that $p\geq q$
(see Remark \ref{rk:bounded}(2)), and we want to prove that $q=0$ or
$(p,q)=(n-1,1)$. Supposing this is not the case, we will reach a
contradiction by constructing an affine line contained in
$Q\cap H$.

(1) Suppose $r=0$ and $q\geq 2$.  By reordering the coordinates if
needed, we can assume that $(a_1,b_1)\neq (0,0)$. For any
$s,t\in \real$, define $X(s,t)=(x_1,\dots, x_p,y_1,\dots, y_q)$ where
$x_1=s$, $y_1=\epsilon s$ (for some $\epsilon =\pm 1$), $x_2=t$,
$y_2=t$, and $x_i=0$, $y_j=0$ for all $i,j\geq 3$. Then we clearly
have $X(s,t)\in Q$, and $X(s,t)\in H$ if and only if
\begin{equation} (a_1+b_1\epsilon)s + (a_2+b_2)t=1. \label{eq:st}\end{equation}
Since $(a_1,b_1)\neq (0,0)$ we can choose $\epsilon$ such that $a_1+b_1\epsilon \neq 0$. Thus for any $t$ there is a unique solution $s(t)$ in Equation \ref{eq:st}. Hence $(X(s(t),t))_{t\in \real}$ is an affine line  contained in $Q\cap H$.

(2) Suppose now $r\geq 1$ and $q\geq 1$. By reordering the coordinates, we can suppose that $(a_1,b_1,c_1)\neq (0,0,0)$. For any
$s,t\in \real$, define $Y(s,t)=(x_1, \dots, x_p, y_1, \dots, y_q,\\ z_1, \dots, z_r)$ where
$x_1=s$, $y_1=\epsilon s$ (for some $\epsilon =\pm 1$), $z_1=t$, and $x_i=0$, $y_j=0$, $z_k=0$ for all $i,j,k\geq 2$. Then we clearly
have $Y(s,t)\in Q$, and $Y(s,t)\in H$ if and only if
\begin{equation}(a_1+b_1\epsilon)s + c_1t=1 .\label{eq:st2}\end{equation}
Since $(a_1,b_1,c_1)\neq (0,0,0)$ we can choose $\epsilon$ such that $(a_1+b_1\epsilon,c_1) \neq (0,0)$. Thus Equation \ref{eq:st2} has an affine line (L) of solutions for $(s,t)$. Hence we obtain an affine line $(Y(s,t))_{(s,t)\in L}$ living in $Q\cap H$.
 
\end{proof}

By definition of finite, affine, and weakly hyperbolic type,
Lemma~\ref{lem:bounded} automatically implies the following property,
which is the key point in the proof of Theorem~\ref{thm:Qincl}.

\begin{prop}
  \label{prop:bounded}
  Let $(\Phi,\Delta)$ be an irreducible  root system, and $Q$ be the isotropic cone   of its associated bilinear form. Let $H$ be an affine   nonlinear hyperplane.   If the intersection $Q\cap H$ is   bounded, then $(\Phi,\Delta)$ is of finite, affine or weakly   hyperbolic type.  If $(\Phi, \Delta)$ is of another type, then $Q\cap H$ contains an affine line. 
\end{prop}

\begin{proof}[Proof of Theorem~\ref{thm:Qincl}]
 
  $(iii) \implies (ii)$ is straightforward, since $E\subseteq \conv(\h\Delta)$.

  $(ii) \implies (i)$: $\h Q\subseteq \conv(\h\Delta)$, so in
  particular $\h Q=Q\cap V_1$ is bounded. By
  Proposition~\ref{prop:bounded}, the root system $(\Phi,\Delta)$ (which
  is assumed to be of indefinite type) is necessarily weakly
  hyperbolic. Let us choose the transverse hyperplane $V_1$ as in
  Proposition~\ref{prop:sphere} such that $\h Q=Q\cap V_1$ is an
  $(n-2)$-dimensional sphere. Let $I$ be a facial subset of $\Delta$, and
  consider the facial root subsystem $(\Phi_I,I)$. Its normalized
  isotropic cone $\h{Q_I}$ is $\h Q\cap \Span (I)$, so it is either
  (1) empty, or (2) a singleton, or (3) an $(|I|-2)$-dimensional
  sphere of positive radius. Since $\h Q\subseteq \conv(\h\Delta)$,
  $\h Q$ cannot cross nontrivially the faces of $\conv(\h\Delta)$, and
  case (3) cannot arise. So any component of $(\Phi_I,I)$  is of  finite  or  affine type. Hence, $(\Phi,\Delta)$ is hyperbolic.

  $(i) \implies (iii)$: Suppose $(\Phi,\Delta)$ is hyperbolic. Choose
  the transverse hyperplane $V_1$ as in Proposition~\ref{prop:sphere} such that
  $\h Q=Q\cap V_1$ is an $(n-2)$-dimensional sphere. From the specific
  study of the imaginary cone $\imc$ for hyperbolic groups in
  \cite{dyer:imc}, we have that $\ol{\imc}\cap V_1$ is equal to the
  ball $\conv(\h Q)$ (this corresponds to the statement
  $\ol{\imc}=\ol{\mathcal L}$ in \cite[Prop.~9.4(c)]{dyer:imc}). From
  Theorem~\ref{thm:imc-closure}, we know that $\ol{\imc}\cap
  V_1=\ol{Z}=\conv(E)$, so we get $\conv(E)=\conv(\h Q)$, i.e., $E$
  contains the set of extreme points of $\h Q$. But since $\h Q$ is a
  sphere, any point in $\h Q$ is an extreme point of the ball
  $\conv(\h Q)$. Hence $E\supseteq \h Q$, and $E=\h Q$.
\end{proof}

\subsection{Proof of Theorem~\ref{thm:fractal2}}
\label{ss:fractal2.1}

We start by giving several equivalent descriptions for the set $\mathcal F $, arising from the characterization of the closed imaginary cone in Theorem~\ref{thm:imc-inter}.

\begin{prop}
\label{prop:F}
Let $\mathcal F $ be defined as above. We have:
\begin{enumerate}[(i)]
\item $\displaystyle{\mathcal F = \bigcap_{w\in W} w \cdot (\hQact \cap \conv(\h\Delta))}$ ;
\item $\mathcal F = Q \cap \ol{\imc} \cap V_1$;
\item $\mathcal F =\h Q \cap \conv(E)$;
\item $\mathcal F $ is the maximal closed subset of $\hQact \cap
  \conv(\h\Delta)$, which is stable by $W$.
\end{enumerate}
\end{prop}

\begin{rem}
\label{rk:Fhyp}
In the case where $\h Q \subseteq \conv(\h\Delta)$ (i.e., finite, affine, or hyperbolic type, according to Theorem~\ref{thm:Qincl}), we have $\mathcal F =\hQact$. But by the same theorem, $E=\h Q$, so $\h Q\subseteq D$, $\hQact=\h Q$ and $\mathcal F =\h Q=E$.
\end{rem}

\begin{proof}
Recall that we denote the $W$-action inside $V_1$ (defined in $D$) as ``$w\cdot v$'', whereas the geometric action of $W$ on $V$ is denoted by ``$w(v)$''.

(i) is clear since $\hQact$ is stable by the $W$-action. For (ii), note that
since $\hQact$ is stable,
\begin{equation*}
 \mathcal F  = \bigcap_{w\in W} w \cdot (\hQact \cap \conv(\h\Delta)) = \hQact \cap \bigcap_{w\in W}  w \cdot (\conv(\h\Delta) \cap D) \ .
\end{equation*}
Moreover $\bigcap_{w\in W}  w (\cone(\h\Delta))\cap V_1 \subseteq D$, so
\begin{align*}
\mathcal F  &= \hQact \cap  \bigcap_{w\in W}  w (\cone(\h\Delta))\cap V_1 \\
&= \h Q \cap  \bigcap_{w\in W}  w (\cone(\h\Delta))\cap V_1  \\
&= \h Q \cap  \bigcap_{w\in W}  w (\cone(\Delta))\cap V_1
\end{align*}
Now, from \cite[Thm.~5.1]{dyer:imc} (see also \cite[Def.~3.1]{dyer:imc}), we have
\[ \bigcap_{w\in W}  w (\cone(\Delta)) = \ol{\imc} \ ,\]
so (ii) is proved. Since $\ol{\imc}\cap V_1=\ol{Z}=\conv(E)$ (see
Theorem~\ref{thm:imc-closure}), the equality (iii) follows (see also
Theorem~\ref{thm:imc-inter}). Finally, the characterization (iv) is clear from equality (i).
\end{proof}

\begin{proof}[Proof of Theorem~\ref{thm:fractal2}]
  Let us prove the equality (i), which is equivalent to $E=\mathcal F $ from Proposition~\ref{prop:F}(iii). The inclusion $E \subseteq \h Q \cap \conv(E)$ is always true. From Proposition~\ref{prop:sphere}, we can choose the transverse hyperplane  $V_1$ such that $\h Q$ is a sphere. Let $x$ be a point in $\conv(E)\cap \h Q$, and suppose $x\notin E$. Then $x$ is a convex combination of some points in $E$, which are points in the sphere $\h Q$. So $x$ cannot lie on the sphere, i.e., $x\not\in \h Q$ (same argument as in the proof of Corollary~\ref{cor:Eexthyp}), which is contradictory. Thus $ \h Q \cap \conv(E) \subseteq E$.

(ii) Let $G$ be a non-empty, closed subset of $\hQact \cap \conv(\h\Delta)$, which is stable by $W$. Then $G\subseteq \mathcal F$ by Proposition~\ref{prop:F}(iv). So $G\subseteq E$. As it is non-empty, closed and $W$-stable, Corollary~\ref{cor:minimal}(b) implies that $G=E$.
\end{proof}

\section{On facial restrictions of subsets of $E$ and the dominance order}
\label{se:dom}
 
Let $(\Phi,\Delta)$ be a based root system, with associated Coxeter
system $(W,S)$. Take~$I\subseteq S$ a facial subset, i.e., 
$F_I:=\conv(\h{\Delta_I})$ is a face of  the polytope $\conv(\h\Delta)$. 
 Recall that $(\Phi_I,\Delta_I)$ is facial root subsystem, as recalled 
in the introduction of \S\ref{se:fractal}, and the set $E(\Phi_I)$ is the set of
limit roots which are accumulation points of $\h{\Phi_I}$.  A natural
question to ask is whether it is possible to describe $E(\Phi_I)$ from
$E(\Phi)$. Clearly~$E(\Phi_I)$ is contained in $E(\Phi)\cap F_I$;
however, the equality is not true in general, and a 
counterexample was provided in \cite[Ex.~5.8]{HLR}. It is interesting to note that the imaginary convex set behaves well with facial restriction:
 $K(\Phi_I)=K(\Phi)\cap F_I$ and $Z(\Phi_I)=Z(\Phi)\cap F_I$, see~\cite[Lemma~3.4]{dyer:imc}.

\smallskip
In this section we explore the question of the restriction of some subsets of $E(\Phi)$ to a face $F_I$ of $\conv(\h\Delta)$.
 By doing so, we will be brought to interpret \emph{the dominance order} and \emph{elementary roots}  in our geometrical setting.  
 \smallskip

A natural, (countable and dense) subset of $E(\Phi)$ that we start to consider for facial restriction is  
the set of \emph{dihedral limit roots} already considered in Equation~\eqref{eq:E2}:
\[ 
  E_2(\Phi)= \bigcup_{\alpha,\beta\in\Phi} L(\h\alpha,\h\beta)\cap \h Q, 
\]
where $L(\h\alpha,\h\beta)$ denotes again the line containing $\h\al$ and
$\h\bt$. One of the main results of this section is that the set of dihedral limit roots respects the facial structure. 

\begin{thm} 
\label{thm:E2facial} 
For all $I\subseteq S$ facial, $E_2(\Phi_I)=E_2(\Phi)\cap F_I$.
\end{thm}

The proof will be given in \S\ref{ss:intersection}. The question of
characterizing subsets of $E$ that verify the same facial
restriction equality is open and is discussed a little bit more in
\S\ref{ss:facialinter}.  Still, we are able to give more examples of
such subsets in this section. They are all subsets of $E_2(\Phi)$
built from the notion of the \emph{dominance order}, \emph{elementary
  roots} and \emph{the root poset}. We also provide, along the way,
useful geometric interpretations of the \emph{``normalized version''}
of these important combinatorial tools using the geometry of the
normalized isotropic cone $\h Q$.

\subsection{Dominance order, elementary roots, elementary limit roots}
\label{ss:elem}

We collect here some definitions used throughout this section.

\begin{defi}
\label{def:dom}
Let $(\Phi,\Delta)$ be a based root system.
\begin{itemize}
\item The \emph{dominance order} is a partial order on $\Phi$ defined by:
\[
\alpha\preceq \beta \quad \textrm{if and only if}\quad \forall w\in W, \quad w(\beta)\in \Phi^-\Longrightarrow w(\alpha)\in\Phi^-
\]
(we say that $\beta$ \emph{dominates} $\alpha$).
\item A positive root $\beta$ is called \emph{elementary}\footnote{These
  roots are also called \emph{humble} or \emph{small} in the
  literature. We adopt here the terminology of \cite{BH:aut}. See
  \cite[Notes, p.130]{bjorner-brenti} for more detail.} when $\beta$ dominates
no other positive roots than itself: \[\forall \alpha\in\PP,\ \alpha\preceq \beta \implies \alpha=\beta.\]
\item We denote by $\Sigma(\Phi)$ (or $\Sigma$ when $\Phi$ is clear)
  the set of elementary roots.
\end{itemize}
\end{defi}

For instance, the simple roots are elementary: $\Delta\subseteq
\Sigma$, but there can be other elementary roots than the simple roots. For example, in Figure~\ref{fig:intro}, the elementary roots are $\alpha$, $\beta$, $\gamma$, $s_\alpha \cdot \gamma$ and $s_\gamma \cdot \alpha$; in Figure~\ref{fig:elem} they are the points in purple.

\begin{figure}[!h]
\centering
\scalebox{1}{\input{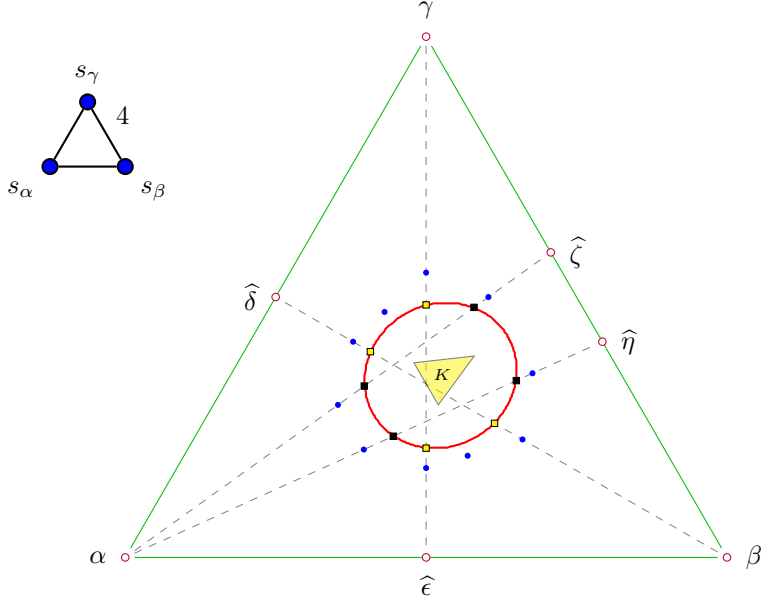}}
\caption{Example of elementary roots and elementary limit roots for
  the rank $3$ root system with diagram in the upper left corner. As in
  Figure~\ref{fig:intro}, $\h{Q}$ is in red. In blue and purple are
  the first normalized roots (with depth $\leq 4$): the elementary
  (normalized) roots are in purple (small circles), while the
  non-elementary ones are in blue (full dots) (see \S\ref{ss:short}
  for an interpretation). The elementary limit roots are the small
  squares on $\h Q$ (yellow and black). In yellow are the ones in
  $\Efcov(\Phi)$, constructed from fundamental covers of dominance
  (see \S\ref{ss:fund}); for example, $-\delta \dprecf \beta$ since
  $\beta + \delta \in \mathcal K$. The polytope $K$ is in shaded
  yellow, and illustrates Remark~\ref{rk:barycenter}.}
\label{fig:elem}
\end{figure}

We give in \S\ref{ss:domgeom} other characterizations of the dominance
order, including a geometric interpretation. The dominance order and
elementary roots are a fundamental tool allowing one to build a finite
state automaton for the language of words in Coxeter groups, as shown
by B.~Brink and R.~Howlett~\cite{BH:aut}. The key point in their
construction is the property that $\Sigma$ is a finite set (see
\cite{BH:aut}, or also \cite[\S4.7]{bjorner-brenti}). Translating this
fact into our framework, we construct the set of \emph{elementary
 root limits}:
\begin{equation}
\label{eq:elem}
\Eelem(\Phi):=\bigcup_{\alpha,\beta\in\Sigma(\Phi)}
L(\h\alpha,\h\beta)\cap \h Q,
\end{equation}
which is finite since $\Sigma$ is finite and $|L(\h\alpha,\h\beta)\cap
\h Q|\leq 2$ for all $\alpha,\beta\in \PP$. Note that by definition,
one has: $\Eelem(\Phi)\subseteq E_2(\Phi) \subseteq E(\Phi)$. 
By Theorem~\ref{cor:minimal}(b) we get immediately the following result.

\begin{prop}\label{prop:ElemLimits} The union of the   $W$-orbits of points in the finite set $\Eelem(\Phi)$ is dense in $E$:
 \[E(\Phi)=\ol{W\cdot \Eelem(\Phi)}.\]
\end{prop}

\begin{rem} 
In an early version of this research,  before we had the idea of Theorem~\ref{cor:minimal}(b), $\Eelem(\Phi)$ allowed us to prove the existence of a dense subset of $E$ that is a finite union of $W$-orbits of points of $E$. We think the techniques in this  direct elementary proof of Proposition~\ref{prop:ElemLimits} are still of  interest and it will be presented in \S\ref{ss:density}. 
\end{rem}

An example is shown in Figure~\ref{fig:elem}. We will explain the
picture in more detail once we have described a way to compute
algorithmically the elementary roots. In order to do so, we need an
alternative characterization of $\Sigma$, using another partial order
on $\PP$, which we introduce now (see \cite[\S4.6]{bjorner-brenti} for
details).

\subsection{Geometric construction of elementary roots and the root poset}
\label{ss:short}

We consider $\PP$ with the partial order defined as the transitive
closure of the following relation: let $\alpha\in \Delta$ and $\beta
\in \PP \setminus \{\alpha\}$, we write $\beta \leq
s_\alpha(\beta)$ if $\mpair{\alpha,\beta}\leq 0$ and $\beta \geq
s_\alpha(\beta)$ otherwise.  The poset $(\PP,\leq)$ is called the
\emph{root poset on $\PP$}. This poset is an interesting tool to
build roots algorithmically by means of the depth. Recall that the depth of
$\beta\in\PP$ is:
\[
\dep(\beta)= 1 + \min\{k \,|\, \beta=s_{\alpha_1} s_{\alpha_2}\dots
s_{\alpha_k}(\alpha_{k+1}), \text{ for } \alpha_1,\dots, \alpha_k, \alpha_{k+1} \in \Delta\}.
\]
Note that for $\alpha\in \Delta$ and $\beta \in \PP \setminus
\{\alpha\}$, $\beta \leq s_\alpha(\beta)$ precisely when the depth
increases (see \cite[\S4.6]{bjorner-brenti}). The root poset is ranked by the
depth, i.e., the length of a chain from $\beta$ to $\gamma$ in $\PP$ such
that $\beta\leq \gamma$ is $\dep(\gamma)-\dep(\beta)$.

The root poset is intimately linked with elementary roots: given
$\alpha\in\Delta$ and $\beta\in\PP$, the edge
$\beta$\,---\,$s_\alpha(\beta)$ in the root poset of $\PP$ is called
\emph{a short edge} if $|\mpair{\alpha,\beta}|<1$; otherwise it is
called \emph{a long edge}. With this terminology, one has the following characterization of elementary roots.

\begin{prop}[{see \cite[\S4.6]{bjorner-brenti}}]
  \label{prop:rootposet}
  Let $\beta\in \PP$, then $\beta \in \Sigma$ (i.e., $\beta$ is an
  elementary root) if and only if there is a chain from a simple root
  to $\beta$ constituted only in short edges.
\end{prop}

This statement gives a useful way to construct elementary roots. First of all, recall from \cite{HLR} that the dihedral reflection subgroup generated by $s_\alpha$ and $s_\beta$ (with $\alpha,\beta\in \Phi$) is finite if and only if $|\mpair{\alpha,\beta}|<1$, if and only if the line $L(\h\alpha,\h\beta)$ does not intersect $\h Q$. We obtain the following geometric characterization of short and long edges: let $\alpha\in\Delta$ and $\beta\in\PP$, then:
\begin{itemize}
\item $\beta$\,---\,$s_\alpha(\beta)$ is a \emph{short edge} if and
  only if $L(\h\alpha,\h\beta)\cap Q=\varnothing$;
\item $\beta$\,---\,$s_\alpha(\beta)$ is a \emph{long edge} if and only
  if $L(\h\alpha,\h\beta)\cap Q\not=\varnothing$.
\end{itemize}
We build $\Sigma$ by induction:
\begin{description}
\item[Initial step] Take $\Sigma_1:=\Delta$;
\item[Inductive step] Assume $\Sigma_k$ is built.  Draw the lines
  between the roots in $\Sigma_k$ and the simple roots and select
  those that do not cut $\h Q$:
  \[
  \Sigma_{k+1}:=\Sigma_k\cup\{s_\alpha(\beta)\,|\, \alpha\in
  \Delta,\ \beta\in \Sigma_k,\ L(\alpha,\h\beta)\cap Q=\varnothing\}.
  \]
  In Figure~\ref{fig:intro}, we only have to consider the line
  $L(\alpha,\gamma)$ to build $\Sigma_2$.
\item[Final step] By Proposition~\ref{prop:rootposet}, $\Sigma_k$ is
  composed of elementary roots for any~$k$; moreover,
  $\Sigma_k\setminus \Sigma_{k-1}$ is constituted of roots of depth
  $k$. So $\bigcup_k \Sigma_k = \Sigma$. Since $\Sigma$ is finite,
  there exists $N\geq 1$ such that
  ${{\Sigma=\Sigma_N=\Sigma_{N+1}}}$. In the example of Figure~\ref{fig:intro},
  $\Sigma=\Sigma_2$.
\end{description}

\subsection{Geometric interpretation of dominance order}
\label{ss:domgeom}

The ingredients here are not new, and already appear in
\cite{BH:aut}. However, we provide here a geometric interpretation of
the dominance order inside $\conv(\h\Delta)$, using the normalized
isotropic cone $\h Q$ and the particular geometry of infinite dihedral
reflection subgroups. The idea is the following: given $\rho, \gamma
\in \PP$, $\rho \preceq \gamma$ if and only if when looking at $\h Q$,
$\h\rho$ is able to see $\h\gamma$ (see Figure~\ref{fig:geodomin}). As
far as we know, such a description never appeared in the literature,
we feel that to get a geometric intuition of the properties of the
dominance order could be very useful in future works (for example,
many of the properties proved in \cite{fu:imc} have a natural
interpretation in this setting).

\medskip

We need first to recall basic properties of the dominance order, from the seminal work in \cite{BH:aut}.

\begin{prop}[{see \cite[\S2]{BH:aut}}]
\label{prop:rappelDomin}
Let $\alpha,\beta\in\Phi$.
\begin{enumerate}[(i)]
\item There is a dominance relation between $\alpha$ and $\beta$ of and only if $\mpair{\alpha,\beta}\geq 1$.
\item Let $w\in W$, then $w(\alpha)\preceq w(\beta)$ if and only if $\alpha\prec\beta$.
\item Assume that $\alpha \preceq \beta$. Then:
\begin{enumerate}[(a)]
\item if  $\alpha\in\PP$, then $\beta\in\PP$;
\item if $\beta\in\Phi^-$, then $\alpha\in\Phi^-$;
\item $-\beta\preceq -\alpha$.
\end{enumerate}
\item If $\alpha,\beta\in\PP$, then $\alpha\preceq \beta$ if and only if $\mpair{\alpha,\beta}\geq 1$ and $\dep(\alpha)\leq \dep(\beta)$.
\item If $\alpha\in\Phi^-$ and $\beta\in\PP$, then $\alpha\preceq \beta$ if and only if $\mpair{\alpha,\beta}\geq 1$.
\item Let $(\Phi',\Delta')$ be a   root subsystem of $(\Phi,\Delta)$ (i.e. it is a based root system such that $\Delta'\subseteq \PP$ and therefore $\Phi'\subseteq \Phi$). Denote by $\preceq'$ the
dominance order of the root system $(\Phi', \Delta')$. Then:
\[\forall \rho,\gamma \in\Phi', \rho \preceq \gamma
\Leftrightarrow \rho \preceq' \gamma .\]
\end{enumerate}
\end{prop}

\begin{proof}
The proof of \emph{(i)-(v)} can be found in \cite[\S2]{BH:aut}. For
\emph{(vi)}, Let $(\Phi',\Delta')$ be a   root subsystem of
$(\Phi,\Delta)$, and denote by $W'$ the associated reflection subgroup
of $W$. Let $\rho \neq \gamma\in\Phi'$.

First suppose that $\rho\preceq \gamma$. For $w\in W'$, if $w(\gamma)\in
\Phi'^-\subseteq \Phi^-$, then $w(\rho)\in \Phi^-$ by hypothesis. But
as $W'$ preserves $\Phi'$, we also have $w(\rho)\in \Phi'$, so
$w(\rho)\in \Phi^-\cap \Phi'=\Phi'^-$. Hence $\rho\prec'\gamma$.

Conversely, suppose now that $\rho\preceq'\gamma$. Then, by
\emph{(i)} applied to $\Phi'$, we have
$\mpair{\rho,\gamma}\geq 1$. So by the same property \emph{(i)}, we have either
$\rho\preceq \gamma$ or $\gamma \preceq \rho$. If we suppose $\gamma \preceq
\rho$, we get by the first implication that $\gamma \preceq'\rho$ as well, so
$\gamma=\rho$, which is a contradiction.
\end{proof}

The above proposition is a key to give a geometric interpretation of
the dominance order: for a dominance relation to be possible between
$\rho$ and $\gamma$ in $\Phi$, the line $L(\h\rho,\h\gamma)$ must
intersect $Q$. However, this does not imply that there will be a
dominance relation between $\rho$ and $\gamma$, since $\mpair{\rho,\gamma}$
could be negative.  We aim to get a criterion on the line
$L(\h\rho,\h\gamma)$, which represents the dihedral reflection
subgroup $W':=\{s_\rho,s_\gamma\}$, for a root on this line to
dominate another.  \smallskip

Let $\alpha,\beta\in \PP$ such that $\mpair{\alpha,\beta}\leq -1$ (so $L(\h\alpha,\h\beta)$ intersects $Q$). Set:
\[
\Delta'=\{\alpha,\beta\},\  W':=\{s_\alpha,s_\beta\}\textrm{ and }\Phi':=W'(\Delta').
\]
We know that in this case $(\Phi',\Delta')$ is a root system with
associated infinite dihedral reflection subgroup $W'$ (see
\cite[Proposition~1.5(2)]{HLR}). Moreover,
by~\cite[Proposition~3.5]{HLR}, we know that $L(\h\alpha,\h\beta)\cap
Q=\{x,s_\alpha\cdot x\}$ for some $x$. Choose notation so 
$x\in \conv(\set{\alpha,s_{\alpha}\cdot x})$. We recall here some basic facts on $W'$, see
for instance \cite[Example~3.8]{HLR} or \cite[Prop.~3.7]{fu:imc}.

\medskip We refer to Figure~\ref{fig:dihedraldom} for a pictorial 
representation of the following description. Let us give to the line $L(\h\alpha,\h\beta)$ the total order inherited from
$\mathbb R$ such that \[\h\alpha<_{\mathbb R} x\leq_{\mathbb R}
s_\alpha\cdot x<_{\mathbb R}\h\beta.\] So
\[
[\h\alpha,\h\beta]=[\h\alpha,x[\sqcup[x,s_\alpha\cdot x]\sqcup]s_\alpha\cdot x,\h\beta],
\]
and $\h{\Phi'}$ is ordered as follows:
\begin{align}
\label{equ1}
 \h\alpha<_{\mathbb R}s_\alpha\cdot \h\beta<_{\mathbb R}s_\alpha s_\beta \cdot \h\alpha<_{\mathbb R}\dots<_{\mathbb R}(s_\alpha s_\beta)^n\cdot \h\alpha<_{\mathbb R} (s_\alpha s_\beta)^n s_\alpha\cdot \h\beta <_{\mathbb R}\dots <_{\mathbb R}x \\
\textrm{and }\ s_\alpha\cdot x<_{\mathbb R} \dots<_{\mathbb R}s_\beta (s_\alpha s_\beta)^n\cdot \h\alpha<_{\mathbb R} (s_\beta s_\alpha)^n \cdot \h\beta<_{\mathbb R}\dots
<_{\mathbb R} s_\beta \cdot \h\alpha\cdot<_{\mathbb R} \h\beta. \notag
\end{align}

Moreover, note that:
\begin{enumerate}[(a)]
\item the function $\Mpair$ is positive on {{$[\h\alpha,x[\times [\h\alpha,x[$}} and $]s_\alpha\cdot x,\h\beta]\times]s_\alpha\cdot x,\h\beta]$; and negative on $[\h\alpha,x[\times]s_\alpha\cdot x,\h\beta]$.
\item the depth function $\dep'$ on $\Phi'$ is increasing on $[\h\alpha,x[\cap \h\Phi$ and decreasing on ${{]s_\alpha\cdot x,\h\beta]\cap \h\Phi}}$ (we obviously mean here that  $\dep'(\h\nu):=\dep'(\nu)$ for any $\nu\in\h{\Phi'}$).
\end{enumerate}

\noindent Now, let $\rho \neq \gamma\in\Phi'^+$. Using
Proposition~\ref{prop:rappelDomin}~\emph{(iv)} and \emph{(vi)}, we
have $\rho\preceq \gamma$ if and only if $\dep'(\rho)\leq \dep'(\gamma)$
and $\mpair{\rho,\gamma}\geq 1$. Note that we always have
$|\mpair{\rho,\gamma}|\geq 1$ (since $\Phi'$ is infinite), so
$\mpair{\rho,\gamma}\geq 1$ if and only if $\mpair{\rho,\gamma}\geq 0$. Then, by
(a) above we obtain:
\begin{equation}
\label{equ2}
\rho\preceq \gamma\textrm{ if and only if  }\dep'(\rho)\leq \dep'(\gamma)\textrm{ and } \h\rho,\h\gamma\in [\h\alpha,x[\textrm{ or }\h\rho,\h\gamma\in ]s_\alpha\cdot x,\h\beta].
\end{equation}
Now by (b) we get:
\begin{equation}
\label{equ3}
\rho\preceq \gamma\textrm{ if and only if  }\h\gamma\in [\h\rho,x[\textrm{ or }\h\gamma\in ]s_\alpha\cdot x,\h\rho].
\end{equation}
So we deduce that the dominance order on the positive roots corresponding to normalized roots in $[\h\alpha,x[\cap \h\Phi$ is precisely  the order $<_{\mathbb{R}}$ on  this interval:
\begin{equation}
\label{equ4}
\alpha\prec s_\alpha( \beta)\prec s_\alpha s_\beta (\alpha)\prec \dots\prec (s_\alpha s_\beta)^n(\alpha)\prec (s_\alpha s_\beta)^n s_\alpha(\beta) \prec \dots ,
\end{equation}
while it is the reverse of $<_{\mathbb{R}}$ for the positive roots corresponding to $]s_\alpha\cdot x,\h\beta]$:
\begin{equation}
\label{equ5}
  \dots\succ s_\beta (s_\alpha s_\beta)^n(\alpha)\succ  (s_\beta s_\alpha)^n (\beta)\succ \dots
\succ s_\beta s_\alpha(\beta)\succ s_\beta(\alpha)\cdot\succ  \beta.
\end{equation}

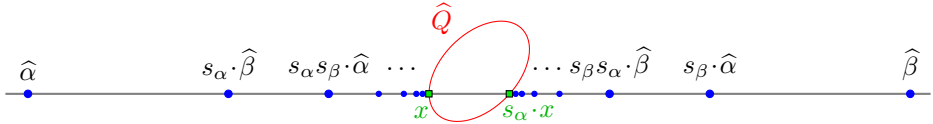
\begin{figure}[h!]

\centering
\newcommand{\ccdot}{\!\cdot \!}

\begin{tikzpicture}
	[scale=1,
	 pointille/.style={densely dashed},
	 root/.style={inner sep=1pt,circle,draw=blue,fill=blue,text=blue},
         sroot/.style={inner sep=0.7pt,circle,draw=blue,fill=blue},
         limroot/.style={inner sep=1.2pt,rectangle,draw=black,fill=green!90!black},
	 ]


\begin{scope} [rotate=45]
  
  \draw[red] (0,0) ellipse (0.8cm and 0.5cm);
\end{scope}

\draw[color=red] (-0.5,0.7) node{$\h Q$} ;
\draw[color=gray, thick] (-6.3,-0.3) -- (6,-0.3);




\node[root, label=above :{$\h\alpha$}] at ( -6.00000000000000 ,-0.3) {};
\node[root, label=above :{$s_\alpha \ccdot \h\beta$}] at (
-3.33500000000000 ,-0.3) {};
\node[root, label=above :{$s_\alpha s_\beta \ccdot \h\alpha$}] at (
-2.00250000000000 ,-0.3) {};
\node[sroot] at ( -1.33625000000000 ,-0.3) {};
\node[sroot, label={[shift={(0,0.15)}]$\dots$}] at ( -1.00312500000000 ,-0.3) {};
\node[sroot] at ( -0.836562500000000 ,-0.3) {};
\node[sroot] at ( -0.753281250000000 ,-0.3) {};

\node[root, label=above :{$\h\beta$}] at ( 5.73000000000000 ,-0.3) {};
\node[root, label=above :{$s_\beta \ccdot \h\alpha$}] at (
3.06500000000000 ,-0.3) {};
\node[root, label=above :{$s_\beta s_\alpha \ccdot \h\beta$}] at (
1.73250000000000 ,-0.3) {};
\node[sroot] at ( 1.06625000000000 ,-0.3) {};
\node[sroot, label={[shift={(0.2,0.15)}]$\dots$}] at ( 0.733125000000000 ,-0.3) {};
\node[sroot] at ( 0.566562500000000 ,-0.3) {};
\node[sroot] at ( 0.483281250000000 ,-0.3) {};

\node [limroot, 
] at (-0.67,-0.3) {};
\node[green!70!black] at (-0.77,-0.55) {$x$};
\node[limroot
] at (0.4,-0.3) {} ;
\node[green!70!black] at (0.65,-0.55) {$s_\alpha \ccdot x$};
\end{tikzpicture}












\caption{Schematic visualization of the dominance order restricted to the root subsystem with simple roots $\{\alpha,\beta\}$ (where $\alpha,\beta \in \PP$ are such that $\mpair{\alpha,\beta}<-1$, and $x,s_\alpha \cdot x$ are the two limit roots of this dihedral root subsystem). There are two chains of dominance, given in Equations \eqref{equ4} and \eqref{equ5}.}
\label{fig:dihedraldom}
\end{figure}

From this discussion, we obtain an interpretation of the relation of dominance within our framework (and therefore that can be easily seen in our affine pictures, see Figure~\ref{fig:geodomin} and Figure~\ref{fig:dihedraldom}).  We say that a point \emph{$x\in\h Q$ is visible from $v\in V_1$} if $[x,v]\cap Q=\{x\}$ (\cite[Definition~3.7]{HLR}). More generally, if $u,v\in V_1$ are two points, we say that \emph{$u$ is visible from $v$ looking at $\h Q$} if the line $L(u,v)$ cuts $\h Q$ in $x$  such that $x$ is visible from $v$ and  $u\in [v,x]$.

\begin{prop}
\label{prop:geodomin}
 Let $\rho\neq \gamma\in\PP$.  Then $\rho\prec \gamma$ if and only if there is  $x\in L(\h\rho,\h\gamma)\cap Q$ that is visible from $\h\rho$ and such that  $\h\gamma\in [\h\rho,x]$. In other words, $\rho\prec \gamma$ if and only if $\h\gamma$ is visible from $\h\rho$ looking at $\h Q$.
 
  In particular, there is a dominance relation between~$\rho$ and~$\gamma$ if and only if  {{$L(\h\rho,\h\gamma)\cap Q\not = \varnothing$}} and $[\h\rho,\h\gamma]\cap Q=\varnothing$. 
 \end{prop}

\begin{figure}[!h]
\begin{minipage}[b]{\linewidth}
\centering
\scalebox{0.5}{\input{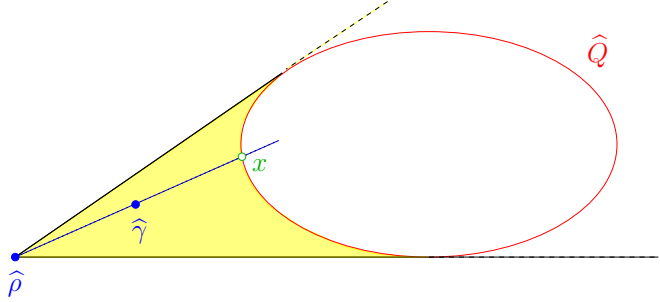}}
\end{minipage}%
\caption{Geometric interpretation of the dominance order: $\rho \prec
  \gamma$ if and only if~$\h \rho $ can see $\h\gamma$ when looking at
  $\h Q$ (see Proposition~\ref{prop:geodomin}).}
\label{fig:geodomin}
\end{figure}

\begin{ex} On the top picture of Figure~\ref{fig:intro}, we can see that $\alpha\prec s_\alpha(\beta)$, since $t$ is visible from both $\alpha$ and $s_\alpha\cdot\beta$ and $s_\alpha\cdot\beta \in [\alpha,t]$. However, there is no dominance relation between $\beta$ and $s_\alpha(\beta)$, since $t\in [s_\alpha\cdot\beta,\beta]\cap Q$.  By considering the line ${{L(s_\alpha\cdot\gamma,s_\alpha\cdot\beta)}}$, we also see that there is no dominance relation between $s_\alpha(\gamma)$ and~$s_\alpha(\beta)$ for a similar reason.
\end{ex}

\begin{proof}  If $\rho\prec\gamma$ then  $\mpair{\rho,\gamma}\geq 1$, by Proposition~\ref{prop:rappelDomin}~$(i)$ and so $L(\h\rho,\h\gamma)$ intersects~$Q$. So, for both directions of the equivalence we have to show $L(\h\rho,\h\gamma)\cap Q\not = \varnothing$. Therefore there is $\Delta':=\{\alpha,\beta\} \subseteq \PP$ that is  a simple system for the infinite dihedral reflection subgroup $W':=\{s_\rho,s_\gamma\}$; see for instance~\cite[Proposition~1.5(2)]{HLR}. Set $\Phi':=W'(\Delta')$. We are therefore in the situation of the discussion above. In particular Properties~(\ref{equ2})~and~(\ref{equ3}) hold, which completes the proof.
\end{proof}

\subsection{Fundamental dominance and cover of dominance}
\label{ss:fund}

Here we construct fundamental dominances, which is a useful tool when considering $W$-orbits since with the action of $W$ they generate all the dominances.

Recall that the imaginary cone of $(\Phi,\Delta)$ was constructed as the union of the sets in the $W$-orbit of the set $\mathcal K(\Phi)$ (see Definition~\ref{def:imc}):
\[ \mathcal K (\Phi):=\{v\in \cone(\Delta)\,|\, \mpair{v,\alpha}\leq 0, \ \forall \alpha\in\Delta\}.
\]

\begin{defi}[see \cite{fu:imc,dyer:imc}] A dominance $\alpha\prec \beta$ of distinct roots is called a \emph{fundamental dominance}, denoted by $\alpha \precf \beta$, if $\beta-\alpha\in\mathcal K(\Phi)$, $\beta\in \Phi^{+}$ and $\alpha\in \Phi^{-}$.
\end{defi}

The term \emph{fundamental dominance} is strongly suggested by the
following result, which is a direct consequence of~\cite[Theorem
  4.13]{fu:imc}:

\begin{thm}[{X.~Fu~\cite{fu:imc}
  }]
  \label{thm:fundamentalconj}
  Let $\alpha,\beta$ be two roots such that $\alpha\preceq \beta$ with $\alpha \neq \beta$, then there is $w\in W$ such that $w(\alpha)\precf w(\beta)$.
\end{thm}

Note that given the definition of the imaginary cone $\imc$, Theorem~\ref{thm:fundamentalconj} implies that if $\alpha\preceq \beta$, then $\beta - \alpha \in \imc$. It turns out that the converse of this property is true, which gives another geometric characterization of the dominance order.

\begin{thm}[{see \cite[Cor.~4.15]{fu:imc}}]
Let $\alpha,\beta \in \Phi$. Then $\alpha\preceq \beta$ if and only if
$\beta - \alpha \in \imc$.
\end{thm}

Using Theorem~\ref{thm:fundamentalconj}, we can construct easily
another subset of $E_2$, whose orbit is exactly $E_2$. Define the set
of fundamental limit roots as the set of limit roots obtained by the
lines associated to couples of fundamental dominance:
\[
\Ef(\Phi):=\bigcup_{\alpha\precf\beta} L(\h\alpha,\h\beta)\cap \h Q.
\]

Then we have the following property.

\begin{prop}
\label{prop:orbitEf}
The orbit $W\cdot \Ef(\Phi)$ of $\Ef(\Phi)$ is equal to $E_2(\Phi)$.
\end{prop}

Note that this implies (without having to use the minimality of the $W$-action from Corollary~\ref{cor:minimal}) that the orbit of $\Ef(\Phi)$ is dense in $E(\Phi)$, by \cite[Thm.~4.2]{HLR}.
\begin{proof}
Recall that for all $w\in W$, $x\in L(\h\alpha,\h\beta)$ if and only
if $w\cdot x\in L(w\cdot\h\alpha,w\cdot\h\beta)$ (\cite[Proposition
  3.5~$(ii)$]{HLR}). So by Theorem~\ref{thm:fundamentalconj} we have
\[
W\cdot \Ef(\Phi)=\bigcup_{\substack{\alpha\precf\beta\\ w\in W}}
L(w\cdot\h\alpha,w\cdot\h\beta)\cap \h Q=\bigcup_{\gamma\prec\rho}
\bigl(L(\h\gamma,\h\rho)\cap \h Q\bigr).
\]
On the other hand, $E_{2}(\Phi)=\bigcup (L(\h\g,\h \rho\,)\cap \h Q)$ where the union is over all $\g,\rho\in \Phi$ with $\g\neq \pm \rho$.  But for any such  roots, and any $\epsilon,\eta\in \set{\pm 1}$,  one has  $L(\h{\epsilon \g},\h{\eta\rho} )=L(\h \g,\h \rho\,)$ and
 $L(\h\gamma,\h\rho\,)\cap
\h Q\not = \varnothing$  if and only if $|\mpair{\gamma,\rho}
|\geq 1$. The equality  $W\cdot \Ef(\Phi)=E_{2}(\Phi)$ follows on 
noting that  $|\mpair{\gamma,\rho}|\geq 1$ if and only if 
$\epsilon \g\preceq \eta \rho$ for some $\epsilon,\eta\in \set{\pm 1}$, by
Proposition~\ref{prop:rappelDomin}(i).  
\end{proof}

\smallskip

\begin{defi}
  ~
  \begin{itemize}
  \item A relation of dominance   is a \emph{cover of dominance}, denoted by $\alpha\dprec \beta$, if there are no roots in between $\alpha$ and $\beta$, i.e., if $\alpha\not = \beta$ and 
    \[
    \forall\gamma\in\Phi,\ \alpha\preceq\gamma\preceq \beta \Longrightarrow \alpha= \gamma \textrm{ or }\gamma=\beta.
    \]
  \item A cover of dominance that is a
    fundamental dominance is called a \emph{fundamental cover of dominance}
    and is denoted by $\alpha\dprecf\beta$.
  \end{itemize}
\end{defi}

Note that, as  dominance order  is preserved by the action of
$W$, so are the covers of dominance. The following result will provide a relation between limit roots coming from fundamental covers of dominance, and elementary limit roots.

\begin{prop}[{X.~Fu~\cite[Corollary 4.17]{fu:imc}
  }]
  \label{cor:elem-dom}   If $\alpha\dprecf \beta$, then $-\alpha,\beta\in \Sigma$. In particular, there is a finite number of fundamental covers of
 dominance.
 \end{prop}

\begin{rem} 
  \label{rk:barycenter}
  In Figure~\ref{fig:elem}, we describe the pairs $(\h\rho,\h\sigma)$
  such that $-\rho\dprecf \sigma$, and draw the polytope $K=\mathcal K
  \cap V_1$. Note that for any of these pairs, we have $\sigma + \rho
  \in \mathcal K$ and $\rho,\sigma\in \Phi^+$, so there is $\lambda
  \in ]0,1[$ such that $\lambda \h\sigma + (1-\lambda) \h\rho \in
      K$. Thus a necessary condition for $(\h\rho,\h\sigma)$ to be
      such a pair is that the open interval $\h\rho$, $\h\sigma$ cuts
      $K$. From this we see already on the example of
      Figure~\ref{fig:elem} that the only candidates for elements of
      $\Efcov$ are the points in yellow, coming from the pairs $(\h
      \beta, \h \delta)$ and $(\h\gamma, \h\epsilon)$.

      More precisely, one has:
  \[ \h{\sigma + \rho} = \frac{\vphi(\sigma)}{\vphi(\sigma)+\vphi(\rho)} \h \sigma + \frac{\vphi(\rho)}{\vphi(\sigma)+\vphi(\rho)} \h \rho,
  \]
  where $\vphi$ is the linear form such that the hyperplane transverse
  to $\PP$ is $V_1= \vphi^{-1}(1)$.  A definition of fundamental
  dominance involving $K$ instead of $\mathcal K$ is therefore:
  $\rho\dprecf \sigma$ if and only if the barycenter of the system
  $(\h\rho,\vphi(\rho))$ and $(\h\sigma, \vphi(\sigma))$ is in $K$. In
  the example (where the transverse hyperplane $V_1$ is assumed to be
  the one passing through $\Delta$), one has $\delta=\alpha +\gamma$,
  so the barycenter to consider is the one of $(\h \beta, 1)$,
  $(\h\delta, 2)$, which is just in $K$.
\end{rem}

\noindent We consider:
\begin{itemize}
\item the finite set $ \Efcov(\Phi)$ of limit points obtained by the lines corresponding to fundamental covers of dominance:
\[
\Efcov(\Phi):=\bigcup_{\alpha\dprecf\beta} L(\h\alpha,\h\beta)\cap \h Q  \subseteq  \Ef(\Phi);
\]
\item the set $\Ecov(\Phi)$ of limits points obtained by the lines corresponding to covers of dominance:
\[
 \Ecov(\Phi):=\bigcup_{\alpha\dprec\beta} L(\h\alpha,\h\beta)\cap \h Q  \subseteq E_2(\Phi).
\]
\end{itemize}

Recall that $\Eelem(\Phi)$ denotes the set of elementary limits of
roots, introduced in~\S\ref{ss:elem}:
\[\Eelem(\Phi):=\bigcup_{\alpha,\beta\in\Sigma} L(\h\alpha,\h\beta)\cap \h Q .\]

The following proposition follows easily from the definitions and Proposition~\ref{cor:elem-dom}.

\begin{prop}
\label{prop:orbEcov}
 $ \Efcov(\Phi)\subseteq \Eelem(\Phi)$ and
  $\Ecov(\Phi)=W\cdot \Efcov(\Phi)\subseteq E_2(\Phi).$
\end{prop}

Figure~\ref{fig:elem} demonstrates an example where the inclusion
$\Efcov(\Phi)\subseteq \Eelem(\Phi)$ is strict.

\begin{rem} It is easy to see that
$\Efcov(\Phi)\subseteq \Ef(\Phi)\cap \Ecov(\Phi)$. However, we do not know if the equality holds.
\end{rem}

\subsection{Restriction  of the dominance order to facial root subsystems}
\label{ss:resdom}

Before discussing the restriction of the dominance order to facial root subsystems, let us briefly recall some basic properties of
cosets of standard parabolic subgroups. For a fixed subset $I
\subseteq S$, the set
\[W^I = \{ u \in W \,|\, \forall \alpha\in\Delta_I , \, u(\alpha)\in\PP\}\]
is a set of minimal length coset representatives for $W/W_I$.
It is well known and a useful fact (\cite[\S5.12]{humphreys})
that for each element $w \in W$, there is a
unique decomposition $w = w^I w_I$ where $w^I \in W^I$ and $w_I \in W_I$. Moreover:
\[w^I(\PP_I)\subseteq \PP\setminus\Phi_I\textrm{ and }w_I(\Phi_I)=\Phi_I.\]
We call the pair $(w^I, w_I)$ the parabolic
components of $w$ along $I$. 

Recall that $I\subseteq S$ is said to be facial for the   root system
  $(\Phi,\Delta)$ if $\conv(\h{\Delta_I})$ is a face of
 $\conv(\h\Delta)$ (see \S\ref{ss:facial}). By \cite[Proposition 2.4(b)]{dyer:imc}  that $\Phi_I=\Phi\cap V_I$ where $V_I=\Span(I)$.

 Denote again $F_I=\conv(\h{\Delta_I})$ the face of $\conv(\h\Delta)$ corresponding to $I$, so that $\h{\Phi_I}=\h\Phi\cap F_I$.
 We list below some other useful properties
related to the restriction of the dominance order to $W_I$, for $I$ facial.

\begin{prop}
  \label{prop:domin-parabolic}
  Let $I\subseteq S$ be facial.
  \begin{enumerate}[(i)]
  \item Let $\alpha,\beta \in \PP$ such that $\alpha\preceq \beta$. If $\beta\in\Phi_I$, then $\alpha\in\Phi_I$.
  \item If $\alpha,\beta \in \Phi_I$, then:
    \begin{enumerate}[(a)]
    \item $\alpha\prec\beta$ in $\Phi_I$ if and only if $\alpha\prec\beta$ in $\Phi$.
    \item   $\alpha\dprec\beta$ is a cover of dominance in $\Phi_I$ if and only if $\alpha\dprec\beta$ is a cover of dominance in $\Phi$.
    \item  $\alpha\precf \beta$ in $\Phi$ if and only if $\alpha\precf \beta$ in $\Phi_I$.
    \end{enumerate}
  \item Denote by $\Sigma$ the set of elementary roots and by
    $\Sigma_I=\Sigma(\Phi_I)$ the set of elementary roots of the root
    subsystem $(\Phi_I,\Delta_I)$. Then $\Sigma_I=\Sigma\cap \Phi_I$,
    so $\Sigma_I=\Sigma\cap V_I$ and $\h{\Sigma_I}=\h{\Sigma}\cap
    F_I$.
  \item Let $\alpha\in \Sigma$ and $w\in W^I$ such that $w^{-1}(\alpha)\in \Phi_I$, then $w^{-1}(\alpha)\in\Sigma_I$.
  \end{enumerate}
\end{prop}

\begin{rem}
  Properties (i), (ii)(a)-(b), and $\Sigma_I=\Sigma\cap \Phi_I$ in
  (iii) remain in fact valid even if $I$ is not facial, because they
  can be written as combinatorial properties of the group $W$ itself
  so they do not depend on the choice of a root system.
\end{rem}

\begin{proof}   $(i)$ Since $\alpha\preceq \beta$, then by Proposition~\ref{prop:geodomin} there is $x\in L(\h\alpha,\h\beta)\cap Q$ that is visible from $\h\alpha$ and such that $\h\beta\in[\h\alpha,x]$.  Assume that $\alpha\notin \Phi_I$, i.e., $\alpha\notin V_I$.  We know, since $I$ is facial, that $V_I\cap V_1$ is the affine subspace of $V$ supporting the face $F_I$ of the polytope $\conv(\h\Delta)$.  Let $H$ be the  subspace of $V_1$ spanned by $\alpha$ and $V_I\cap V_1$. So $\h\beta,\h\alpha,x \in H$. Since $\h\beta\in V_I\cap [\h\alpha,x]$, we necessarily have that  $\h\alpha$ and $x$ are separated by the affine hyperplane $V_I\cap V_1$ in $H$. Since $F_I$ is a face of $\conv(\h\Delta)\cap H$, that   means that either $\h\alpha$ or $x$ is outside $\conv(\h\Delta)$, contradicting the inclusion $\h\Phi\sqcup E(\Phi)\subseteq \conv (\Delta)$.

\smallskip
\noindent $(ii)$ $(a)$ is a particular case of
Proposition~\ref{prop:rappelDomin}\emph{(vi)}.

\smallskip
\noindent $(ii)$ $(b)$ Let $\alpha,\beta \in \Phi_I$. We just have to show  that if $\alpha\dprec\beta$ is a cover of dominance in $\Phi_I$ then $\alpha\dprec\beta$ is a cover of dominance in $\Phi$. Let $\gamma\in\Phi$ such that $\alpha\preceq \gamma\preceq \beta$. First, suppose that $\gamma\in\PP$. Then $\beta\in\Phi_I^+$, and $\gamma \in\Phi_I$ by $(i)$. Since $\alpha\dprec\beta$ is a cover of dominance in $\Phi_I$, we get $\alpha=\gamma$ or $\beta=\gamma$, which proves that $\alpha\dprec\beta$ is a cover of dominance in $\Phi$. Suppose now that $\gamma\in\Phi^-$. Then $\alpha\in\Phi^-$, so $-\gamma\preceq -\alpha$ by Proposition~\ref{prop:rappelDomin}. Then the same line of reasoning as above, with $-\alpha$ in the role of $\beta$ and $-\gamma$ in the role of $\gamma$, shows that $\alpha=\gamma$ or $\beta=\gamma$, which proves again that $\alpha\dprec\beta$ is a cover of dominance in $\Phi$.

\smallskip
\noindent $(ii)$ $(c)$ We will first show that $\mathcal K(\Phi_I)=\mathcal K (\Phi)\cap V_I$. Recall that
\[
\mathcal K (\Phi_I):=\{v\in \cone(\Delta_I)\,|\, \forall \alpha\in\Delta_I, \ \mpair{v,\alpha}\leq 0\}
\]
and
\[
\mathcal K (\Phi):=\{v\in \cone(\Delta)\,|\, \forall \alpha\in\Delta, \ \mpair{v,\alpha}\leq 0\}.
\]
So, since $\cone(\Delta)\cap V_I=\cone(\Delta_I)$, it is clear that
$\mathcal K (\Phi)\cap V_I\subseteq \mathcal K (\Phi_I)$.  Now let
$v\in \mathcal K(\Phi_I)$ and let $\alpha\in\Delta$. If
$\alpha\in\Delta_I$ we know by definition that $\mpair{v,\alpha}\leq 0$. If
$\alpha\in\Delta\setminus\Delta_I$, then $\mpair{\gamma,\alpha}\leq 0$ for
all $\gamma\in\Delta_I$, by definition of a root system. Since $v\in
\cone(\Delta_I)$, we can write
$v=\sum_{\gamma\in\Delta_I}v_\gamma\gamma$ with $v_\gamma\geq 0$, so
by linearity $\mpair{v,\alpha}\leq 0$. Hence $v\in \mathcal K (\Phi)\cap
V_I$.

Now let $\alpha,\beta\in\Phi_I$ such that $\alpha\precf\beta$ in
$\Phi$. So $\beta-\alpha\in\mathcal K(\Phi)$, $\beta\in \Phi^{+}$ and $\al\in \Phi^{-}$. Therefore
$\beta-\alpha\in\mathcal K(\Phi)\cap V_I=\mathcal K (\Phi_I)$, $\beta\in \Phi_{I}^{+}$ and $\al\in \Phi_{I}^{-}$, so 
$\alpha\precf\beta$ in $\Phi_I$. The converse implication is trivial.

   \smallskip
\noindent $(iii)$ Let $\alpha\in\Sigma_I$. Then $\alpha\in
\Phi_I^+$. Consider $\gamma\in\PP$ such that $\gamma\prec
\alpha$. By $(i)$ we know that $\gamma\in\Phi_I^+$. Since $\alpha$ is
elementary in $\Phi_I$, we obtain $\gamma=\alpha$. So $\alpha$ is also
elementary in $\Phi$, i.e., $\alpha\in\Sigma$. Hence $\Sigma_I \subseteq \Sigma \cap \Phi_I$. The reverse inclusion is straightforward.

  \smallskip
\noindent $(iv)$ Let $\alpha\in \Sigma$ and $w\in W^I$ such that
$w^{-1}(\alpha)\in \Phi_I$. We know that $w(\PP_I)\subseteq
\PP$, so $w(\Phi^-_I)\subseteq \Phi^-$. Since $\alpha \in \Sigma
\subseteq \PP$, necessarily $w^{-1}(\alpha) \notin \Phi^-_I$,
i.e., $w^{-1}(\alpha)\in\PP_I$. Now consider $\gamma\in\Phi_I^+$
such that $\gamma\preceq w^{-1}(\alpha)$. By
Proposition~\ref{prop:rappelDomin}~$(ii)$, we get $w(\gamma)\prec
\alpha$. Since $\gamma\in\Phi_I^+$ and $w \in W^I$, we have
$w(\gamma)\in \PP$. But $\alpha$ is elementary, so we obtain
$w(\gamma)=\alpha$ and therefore $w^{-1}(\alpha)=\gamma$. Hence
$w^{-1}(\alpha)\in \Sigma_I$.
\end{proof}

\subsection{Facial restriction to subsets of $E_2(\Phi)$ related to dominance}
\label{ss:intersection}

 In \cite[Example~5.8]{HLR}, it is shown that, in general, the
restriction of $E(\Phi)$ to the face~$F_I$ is not equal to
$E(\Phi_I)$. It turns out that however, this property of good facial
restrictions holds for all the subsets of $E$ that we have defined in
this section. Recall that, given a based root system $(\Phi,\Delta)$,
we have constructed in \S\ref{ss:elem} and \S\ref{ss:fund} the set
$\Ef(\Phi)$, its $W$-orbit $E_2(\Phi)$, the set $\Efcov(\Phi)$, its
$W$-orbit $\Ecov(\Phi)$, and the set $\Eelem(\Phi)$. 
The following theorem states that all these six ``functorial'' 
subsets of $E$ restrict well to facial root subsystems.
 Theorem~\ref{thm:E2facial} is the first item below.

\begin{thm}
  \label{thm:intersection}
  Let $(\Phi,\Delta)$ be a based root system with associated Coxeter
  group $(W,S)$.  Let $I\subseteq S$ be facial, and
  $F_I=\conv(\h{\Delta_I})$ denote the associated face of
  $\conv(\h{\Delta})$. Then:
  \begin{enumerate}[(i)]
  \item $E_2(\Phi_I)= E_2(\Phi)\cap F_I$;
  \item $\Ef(\Phi_I)= \Ef(\Phi)\cap F_I$;
  \item $\Ecov(\Phi_I)= \Ecov(\Phi)\cap F_I$ and $\Efcov(\Phi_I)= \Efcov(\Phi)\cap F_I$;
  \item  $\Eelem(\Phi_I)=\Eelem(\Phi)\cap F_I$ and $W_{I}\cdot\Eelem(\Phi_I)=(W\cdot\Eelem(\Phi))\cap F_I$.
  \end{enumerate}
\end{thm}

\begin{ex} In Figure~\ref{fig:intro}, take $I=\{\beta, \gamma\}$. Then $\Eelem(\Phi_I)=\{x,y\}=W_I\cdot \Eelem(\Phi_I)$. 
\end{ex}

Before getting into the proof of Theorem~\ref{thm:intersection}, we
first need the following key lemma.

\begin{lem}
\label{lem:facial}
Let $x\in E_2(\Phi)$ and $\alpha,\beta\in\Phi$ be distinct such that $x\in L(\h\alpha,\h\beta)$. Let $I\subseteq S$  be facial, and denote by $F_I$ the face $\conv(\h{\Delta_I})$.
\begin{enumerate}[(i)]
\item We have $x\in F_I$ if and only if $\h\alpha,\h\beta\in F_I$, if and only if $\alpha,\beta\in\Phi_I$.
\item Assume $\alpha,\beta\in\PP$ and let $y\in E_2(\Phi_I)$ such that $x=w\cdot y$ for some $w\in W^I$. Then $w^{-1}(\alpha),w^{-1}(\beta)\in \Phi_I^+$ and
$  y\in L(w^{-1}\cdot \h{\alpha},w^{-1}\cdot\h{\beta})$. 
\end{enumerate}
\end{lem}

\begin{proof} $(i)$ First, note that $\h\alpha,\h\beta\in F_I$ if and
  only if $L(\h\alpha,\h\beta)\subseteq V_I$, since $F_I$ is a face of $\conv(\h\Delta)$; that is, if and only if
  $\alpha,\beta\in\Phi_I=\Phi\cap V_I$. Then, note that any line
  $L(\h\alpha,\h\beta)$ contains two normalized roots
  $\h\alpha_0,\h\beta_0$ such that $\alpha_0,\beta_0$ form a simple
  system for the dihedral reflection subgroup $W'$ generated by
  $s_\alpha$ and $s_\beta$. So, $\alpha,\beta\in \Phi_I$ if and only
  if $\alpha_0,\beta_0\in\Phi_I$. Therefore, we assume without loss
  of generality that $\alpha,\beta$ is a simple system for $W'$. In
  particular, $x$ lies in the interior of the segment
  $[\h\alpha,\h\beta]$.

  We just have to show that if $x\in E_2(\Phi_I)$, then
  $\h\alpha,\h\beta\in F_I$ (the remaining implication is trivial). If
  $\h\alpha\in F_I$, then $L(\h\alpha,\h\beta)=L(\h\alpha,x)\subseteq
  V_I$. Therefore $\h\alpha,\h\beta\in F_I$. The symmetric case
  $\h\beta\in F_I$ is handled the same way.

 Suppose now that neither $\h\alpha$ nor $\h\beta$ is in $V_I$. We
 know that $V_I\cap V_1$ is an affine subspace in $V_1$ supporting the
 face $F_I$ of $\conv(\h\Delta)$.  Let $H$ be the affine subspace of
 $V_1$ spanned by $\h\alpha$ and $V_I\cap V_1$. Since $x\in V_I\cap
 [\h\alpha,\h\beta]$, necessarily $\h\alpha$ and $\h\beta$ are
 separated by $V_I\cap V_1$ in $H$. Since $F_I$ is a face of
 $\conv(\h\Delta)\cap H$, that means that either $\h\alpha$ or
 $\h\beta$ is outside $\conv(\h\Delta)$, contradicting the inclusion
 $\h\Phi\subseteq \conv (\h\Delta)$. (Remark: this last argument is
 almost the same as the one used to prove
 Proposition~\ref{prop:domin-parabolic}~$(i)$.)
\smallskip

\noindent $(ii)$ Set $\alpha'=w^{-1}(\alpha)$ and $\beta'=w^{-1}(\beta)$. Then by \cite[Proposition 3.6~(ii)]{HLR}, we have that $z=w^{-1}\cdot x\in L(\h{\alpha'},\h{\beta'})$. So by $(i)$ we have that $\alpha',\beta'\in\Phi_I$, since $y\in E_2(\Phi_I)\subseteq F_I$. Finally $\alpha',\beta'\in\Phi_I^+$ because $\alpha,\beta\in \PP$ and $w\in W^I$ (same argument as the beginning of the proof of \ref{prop:domin-parabolic}\emph{(iv)}).
\end{proof}

Now we are ready to prove Theorem~\ref{thm:intersection}, namely, that the six subsets of $E$ defined in the previous subsections all have the property of good facial restriction.

\begin{proof}[Proof of Theorem~\ref{thm:intersection}]  
\noindent (i) The inclusion $\sreq$ is a direct consequence of Lemma~\ref{lem:facial}(i) and $\seq$ is clear.

\noindent (ii) The inclusion $\seq$ follows from  follows Lemma \ref{lem:facial}(i) on recalling from the proof of \ref{prop:domin-parabolic}(ii)(c) that $\mc{K}(\Phi_{I})=\mc{K}\cap V_{I}$, while  $\sreq$ is proved using   Proposition~\ref{prop:domin-parabolic}(ii) and Lemma~\ref{lem:facial}(i) as follows. Let $x\in \Ef(\Phi)\cap F_I$, so there is $\alpha,\beta\in\Phi$ such that $x\in L(\h \alpha, \h \beta)$ and $\alpha\precf\beta$ in $\Phi$. Since $x\in F_I$, $\alpha,\beta\in\Phi_I$ and therefore  $\alpha\precf \beta$ in $\Phi_I$. So $x\in \Ef(\Phi_I)$.

\noindent (iii) is proved similarly as (ii), using Proposition~\ref{prop:domin-parabolic}(ii) and Lemma~\ref{lem:facial}(i).

\noindent (iv)  From
Proposition~\ref{prop:domin-parabolic}(iii) we know that
$\Sigma_I=\Sigma\cap \Phi_I$, so:
\[
\Eelem(\Phi_I)\subseteq \Eelem(\Phi)\cap F_I\ \textrm{ and }\
W_I\cdot \Eelem(\Phi_I)\subseteq (W\cdot \Eelem(\Phi))\cap F_I.
\]
Now let $x\in \Eelem(\Phi)\cap F_I$. So there is
$\alpha,\beta\in\Sigma$ such that $x\cap L(\h\alpha,\h\beta)$. By Lemma~\ref{lem:facial}(i) we know that
$\alpha,\beta\in\Phi_I$, since $x\in E_2(\Phi)\cap F_I$. So
$\alpha,\beta\in\Sigma_I$ and therefore $x\in \Eelem(\Phi_I)$.

Let $z\in (W\cdot \Eelem(\Phi))\cap F_I$. So there is $w\in W$ and
$x\in \Eelem(\Phi)$ such that $w\cdot z=x$.  Write $w=w^Iw_I$ with
$v=w^I\in W^I$ and $w_I\in W_I$ and set $y=w_I \cdot z$. Since
$w_I(\Phi_I)=\Phi_I$, we have $y \in (W\cdot \Eelem(\Phi))\cap
F_I\subseteq E_2(\Phi)\cap F_I= E_2(\Phi_I)$.

Let $\alpha,\beta\in\Sigma$ such that $x\in
L(\h\alpha,\h\beta)$. Since $x=v\cdot y$ with $v\in W^I$, ${{y\in
E_2(\Phi_I)}}$ and $x\in E_2(\Phi)$, we have by Lemma~\ref{lem:facial}(ii)
that $v^{-1}(\alpha),v^{-1}(\beta)\in \Phi_I^+$ and $y\in
L(v^{-1}\cdot \h{\alpha},v^{-1}\cdot\h{\beta})$.  So by
Proposition~\ref{prop:domin-parabolic}(iv) we have that
$v^{-1}(\alpha),v^{-1}(\beta) \in\Sigma_I$, since
$\alpha,\beta\in\Sigma$. Therefore $y\in \Eelem(\Phi_I)$, hence
$z=w_{I}^{-1}\cdot y \in W_I\cdot \Eelem(\Phi_I)$.
\end{proof}

\subsection{Direct proof of the density of the fundamental limit roots}
\label{ss:density}

Proposition~\ref{prop:ElemLimits} is a consequence of the following statement. 

\begin{prop}
  \label{prop:Ecovdense}
  The set $\Ecov(\Phi)$ is dense in $E(\Phi)$. Consequently, both
  $\Efcov(\Phi)$ and $\Eelem(\Phi)$ provide examples of finite subsets, the union 
 the  orbits of which is dense in $E$.
\end{prop}

Even if it is a straightforward consequence of
Theorem~\ref{cor:minimal}(b),  we will give  a direct
(geometric) proof of Proposition~\ref{prop:Ecovdense}, without using
the minimality of the action, i.e., avoiding the reliance on the
machinery of the imaginary cone developed in \cite{dyer:imc} and used in \S\ref{se:imc}.  This direct
proof is elementary and relies only on a careful study of the geometry
in the case of a root system of rank $3$, and on the density of $E_2$
proved in \cite[Theorem~4.2]{HLR}. It illustrates techniques that may be useful in the study of open questions involving the relationship of dominance order and limit roots.
\medskip

Assume for now that $(\Phi,\Delta)$ is a rank $3$, irreducible root
system in $(V,\Mpair)$. It is well known that the signature of $\Mpair$ is
then $(3,0)$, $(2,0)$ or $(2,1)$, see for instance
\S\ref{ss:limitroots}. In the case where $(\Phi,\Delta)$ is weakly hyperbolic (signature $(2,1)$), we show the following property. This will allow us to pass from dihedral limit roots coming from dominance to dihedral limit roots coming from covers of dominance.

\begin{prop}
  \label{prop:rank3}
  Let $(\Phi,\Delta)$ be a based root system of rank $3$, of weakly
  hyperbolic type. Let $\alpha,\beta,\gamma\in\PP$ such that
  $\mpair{\alpha,\beta}\leq -1$ and $\alpha\prec \gamma$ (with $\alpha
  \neq \gamma$).  Let $x\in \h Q\cap L(\h\alpha,\h\beta)$ and $y\in \h
  Q\cap L(\h\alpha,\h\gamma)$ that are visible from $\h\alpha$. Set
  $w:=s_\alpha s_\beta$. Then we have:
  \begin{enumerate}[(i)]
  \item the sequence $w^n\cdot \h\alpha$ converges to $x$ when $n$
    tends to infinity. Moreover, $w^{n+1}\cdot \h\alpha \in \left]w^n\cdot
    \h\alpha,x \right[$ for all $n\in\mathbb N$.
  \item The sequence $y_n:=w^n\cdot y$ converges to $x$ when $n$
    tends to infinity. Moreover, $y_n\in \h Q\cap L(w^n\cdot
    \h\alpha,w^n\cdot \h\gamma)$ is visible from $w^n\cdot \h\alpha$
    for all $n\in\mathbb N$.
  \end{enumerate}
\end{prop}

This property and its proof are illustrated in Figure~\ref{fig:exproof}.

\begin{proof}
  Set $W':=\langle s_\alpha, s_\beta\rangle$, so $w\in W'$, and set
  $L:=L(\h\alpha,\h\beta)$. Since $\mpair{\alpha,\beta}\leq -1$,
  $\Delta':=\{\alpha,\beta\}$ is a simple system for the root system
  $\Phi':=W'(\Delta')$.
  \medskip
  \noindent
  $(i)$ By~Equation~(\ref{equ1}) in \S\ref{ss:domgeom}, we
  know that $x$ is visible from $\h\alpha$ on the line $L$ and that
  $w^{n+1}\cdot \h\alpha \in]w^n\cdot \h\alpha,x]$. Therefore, the
  sequence $(w^n\cdot \h\alpha)_{n\in\mathbb N}$ is increasing (in
  the line $L$ ordered from $\h\alpha$ to $\h\beta$) and entirely
  contained in $[\h\alpha,x]$.  So $(w^n\cdot \h\alpha)_{n\in\mathbb
    N}$ has a limit $\ell$ in $[\h\alpha,x]$. This limit is in $E$, so
  also in $Q$ by~\cite[Theorem 2.7]{HLR}. Since $Q\cap
  [\h\alpha,x]=\{x\}$, we obtain $\ell=x$.

  \medskip
  \noindent $(ii)$ By~\cite[\S2.3]{HLR}, we can assume without loss of
  generality that $V_1$ is the transverse plane of
  Proposition~\ref{prop:sphere}, so $\h Q$ is a circle in the plane
  $V_1$. Since $y\in L(\h\alpha,\h\gamma)$ is visible from $\h\alpha$,
  we obtain by \cite[Proposition 3.6 $(ii)$ and Proposition 3.8
    $(iii)$]{HLR} that $y_n=w^n\cdot y\in L(w^n \cdot
  \h\alpha,w^n\cdot \h\gamma)$ is visible from $w^n\cdot\h\alpha$ for
  all $n\in\mathbb N$.

  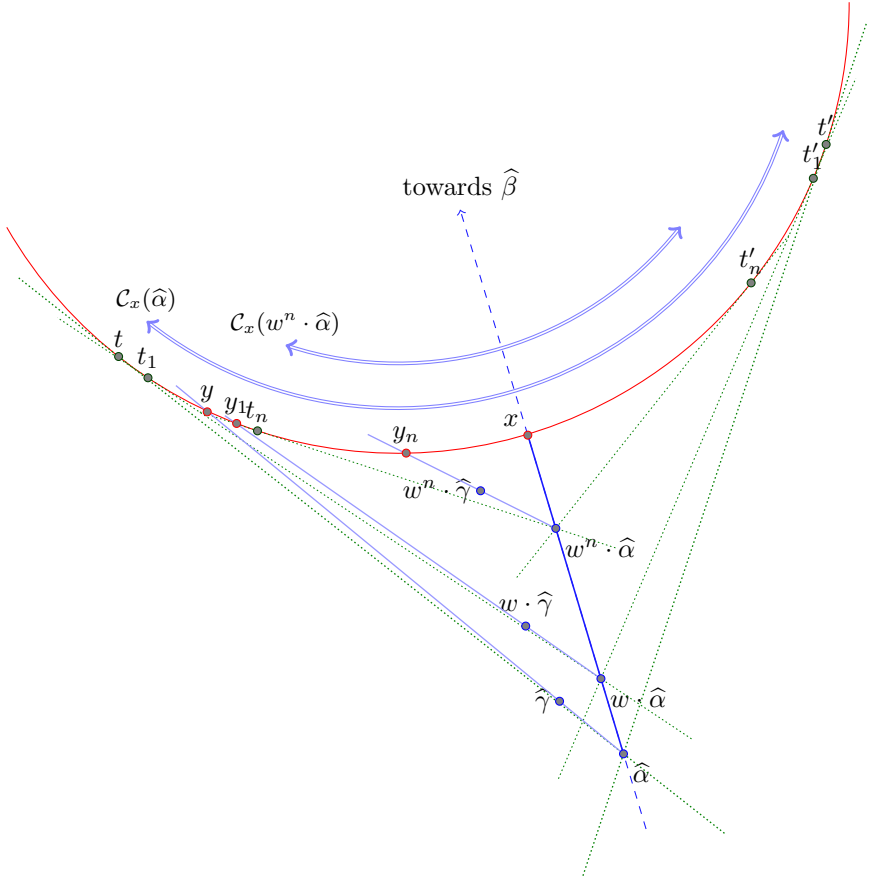
\begin{figure}[!h]
    \centering \captionsetup{width=\textwidth}
    \scalebox{1}{  \begin{tikzpicture}
  \tkzInit[xmax=5,ymax=4]

  
\tkzDefPoint(-3,0){O}  
\tkzDefPoint(3,0){A}
\tkzDefPointBy[rotation= center O angle 210](A)
 \tkzGetPoint{AA}

\tkzDrawArc[color=red](O,AA)(A)


\tkzDefPoint[label=-60:$\widehat\alpha$](0,-10){al}   
\tkzDefPoint[label=above left:$x$](-1.272,-5.76){x} 

\tkzDefPoint[label=above:towards $\h\beta$](-2.172,-2.76){xx} 

\tkzDefPoint[label=-60:$w\cdot\widehat\alpha$](-0.3,-9){A1}
\tkzDefPoint[label=-60:$w^n\cdot\widehat\alpha$](-0.9,-7){A2}

\tkzDrawLines[color=blue!90,add=0 and 0](al,x) 

\draw[color=blue!90,dashed,->] (0.3,-11) -- (-2.172,-2.76);


\tkzTangent[from=al](O,A) 
   \tkzGetPoints{t'}{f}
   \tkzGetPoints{k}{t}
\tkzDrawLines[densely dotted, color=green!50!black, line width=0.5pt](al,t al,t')    
\tkzLabelPoints[above](t,t')

\tkzDefPointBy[homothety=center O ratio .9](t) \tkzGetPoint{tt}
\tkzDefPointBy[homothety=center O ratio .9](t') \tkzGetPoint{tt'}
\tikzset{compass style/.append style={<->}}
\tkzDrawArc[color=blue!50,double](O,tt)(tt')
\tkzLabelPoint[above](tt){\small $\mathcal C_x(\widehat\alpha)$};


\tkzTangent[from=A1](O,A) 
   \tkzGetPoints{t'_1}{f1}
   \tkzGetPoints{k1}{t_1}
\tkzDrawLines[densely dotted, color=green!50!black, line width=0.4pt](A1,t_1 A1,t'_1)    
\tkzLabelPoints[above](t_1,t'_1)


\tkzTangent[from=A2](O,A) 
   \tkzGetPoints{t'_n}{f2}
   \tkzGetPoints{k2}{t_n}
\tkzDrawLines[densely dotted, color=green!50!black, line width=0.3pt](A2,t_n A2,t'_n)    
\tkzLabelPoints[above](t_n,t'_n)

\tkzDefPointBy[homothety=center O ratio .8](t_n) \tkzGetPoint{tt_n}
\tkzDefPointBy[homothety=center O ratio .8](t'_n) \tkzGetPoint{tt'_n}
\tikzset{compass style/.append style={<->}}
\tkzDrawArc[color=blue!50,double](O,tt_n)(tt'_n)
\tkzLabelPoint[above](tt_n){\small $\mathcal C_x(w^n\cdot\widehat\alpha)$};



\tkzDefPoint[label=left:$\h\gamma$](-0.85,-9.3){ga}
\tkzDrawLine[color=blue!40, line width=0.5pt,add=0 and 6](al,ga) 
\tkzInterLC(al,ga)(O,A) \tkzGetFirstPoint{y}


\tkzDefPoint[label=90:$w\cdot\h\gamma$](-1.3,-8.3){ga1}
\tkzDrawLine[color=blue!40, line width=0.5pt,add=0 and 4](A1,ga1) 
\tkzInterLC(A1,ga1)(O,A) \tkzGetFirstPoint{y_1}


\tkzDefPoint[label=left:$w^n\cdot\h\gamma$](-1.9,-6.5){ga2}
\tkzDrawLine[color=blue!40, line width=0.5pt,add=0 and 1.5](A2,ga2) 
\tkzInterLC(A2,ga2)(O,A) \tkzGetFirstPoint{y_n}

\tkzDrawPoints[color=green!25!black](t,t')
\tkzDrawPoints[color=green!25!black](t_1,t'_1)
\tkzDrawPoints[color=green!25!black](t_n,t'_n)

  \tkzDrawPoints[color=blue](ga2,ga1,ga,al,A1,A2)
  \tkzDrawPoints[color=red](x,y,y_1,y_n)
\tkzLabelPoints[above](y,y_1,y_n)

  \end{tikzpicture}}
    \caption{Illustration of the proof of
      Proposition~\ref{prop:rank3}: in red is a part of the circle $\h
      Q$ on which lives the arc $\mathcal C_x (\h\alpha)$ of all the
      points on $\h Q$ visible from $\h\alpha$. Here we adopt the
      notation $t:=t(\h\alpha)$ and $t':=t'(\h\alpha)$ for the
      intersection points of the two tangents to the circle $\h Q$
      passing through $\h\alpha$; similarly we consider $t_1:=t(w\cdot
      \h\alpha)$ and $t_n:=t(w^n\cdot \h\alpha_n)$. Note that in the
      case where $\mpair{\alpha,\beta} = -1$, the line
      $L(\h\al,\h\bt)$ is tangent to $\h Q$ and $x$ is equal to $t$ or
      $t'$.}
    \label{fig:exproof}
 \end{figure}%

  \smallskip In order to finish to prove $(ii)$, we need to cover some
  basic facts from classical Euclidean geometry. Any point $p$ outside
  the closed disk bounded by $\h Q$ has two tangent lines to $\h Q$
  passing through $p$; these tangent lines define two tangent points
  $t(p),t'(p)$.  Let $z$ be a point of the circle, visible from
  $p$. The arc $\mathcal C_z(p)$ of the circle $\h Q$ containing $z$
  and bounded by $t(p)$ and $t'(p)$ is precisely the set of elements
  in the circle that are visible from $p$, see
  Figure~\ref{fig:exproof}. It is an easy exercise to show the
  following two statements:
  \begin{itemize}
  \item for any $q\in [p,z]$, $z$ is visible from $q$, and $\mathcal
    C_z(q)\subseteq\mathcal C_z(p)$;
  \item let $(p_n)_{n\in \mathbb N}$ be a sequence of points outside
    the closed disk, such that $z$ is visible from $p_n$ for any $n$;
    if $(p_n)$ converges to $z$, then the length of the arc $\mathcal
    C_z(p_n)$ tends to $0$, and so $\bigcap_{n\in \mathbb N}\mathcal
    C_z(p_n)=\{z\}$.
  \end{itemize}

  \smallskip
  We apply these facts to our situation. For all $n\in\mathbb N$, we
  have:
  \begin{itemize}
  \item $y_n\in \mathcal C_x(w^n\cdot \h\alpha)$, since $y_n$ is
    visible from $w^n\cdot\h\alpha$;
  \item $x$ is visible from $w^n\cdot \h\alpha$ and $\mathcal
    C_x(w^{n+1}\cdot \h\alpha)\subseteq\mathcal C_x(w^n\cdot
    \h\alpha)$, by (1) (see Figure~\ref{fig:exproof}).
  \end{itemize}
  Since $w^n\cdot \h\alpha$ converges to $x$, we have therefore that
  $y_n$ converges to a limit that lives in the set
  \[
  \bigcap_{n\in\mathbb N} C_x(w^n\cdot \h\alpha)=\{x\}.
  \]
\end{proof}

\begin{proof}[Proof of Theorem~\ref{prop:Ecovdense}]
  Since $E_2(\Phi)$ is dense in $E(\Phi)$ by \cite[Thm.~4.2]{HLR}, it
  is enough to show that any $x\in E_2(\Phi)$ is the limit of a
  sequence in $\Ecov(\Phi)$.

  Let $x\in E_2(\Phi)$ and let $\alpha,\beta\in \PP$ such that $x\in
  L(\h\alpha,\h\beta)$. We can choose $\alpha,\beta$ such that
  $\mpair{\alpha,\beta}\leq -1$ and $x$ is visible from $\h\alpha$.

  \smallskip

  We have to prove that there is a sequence $(y_n)_{n\in\mathbb
    N}\subseteq \Ecov(\Phi)$ converging to $x$. Let $\gamma\in\PP$
  such that $\alpha\dprec \gamma$; for instance, any $\g\in \Phi^{+}$ with $\al\prec \g\prec s_{\al(\bt)}$, $\g\neq \al$ and $l(s_{\g})$ is minimal amongst $\g$ with these properties will do. Since $\alpha\prec \gamma$, there
  is $y\in L(\h\alpha,\h\gamma)$ such that $y$ is visible from
  $\h\alpha$ (Proposition~\ref{prop:geodomin}). We have two cases:

  \begin{enumerate}[(i)]
  \item If $\h\gamma\in L(\h\alpha,\h\beta)$, then $y\in
    L(\h\alpha,\h\gamma)= L(\h\alpha,\h\beta)$ and therefore $y=x\in
    \Ecov(\Phi)$ is the unique point in $L(\h\alpha,\h\beta)$ that is
    visible from $\h\alpha$.

  \item Assume that $\h\gamma\notin L(\h\alpha,\h\beta)$, so $x\not =
    y$. Set $V':=\Span\{\alpha,\beta,\gamma\}$, $W':=\langle
    s_\alpha,s_\beta,s_\gamma \rangle$ and
    $\Phi':=W'(\{s_\alpha,s_\beta,s_\gamma\})$. Using \cite{dyer:ref}, there is a simple
    system such that $(\Phi',\Delta')$ is a root system of rank $3$ in
    $(V',\Mpair)$ with associated Coxeter group $W'$.  Recall that
    $\mpair{\alpha,\beta}\leq -1$, so $w:=s_\alpha s_\beta$ has
    infinite order. Since $\alpha\dprec \gamma$, $\alpha\not = \gamma$
    and $\mpair{\alpha,\gamma}\geq 1$. Therefore $s_\alpha s_\gamma$
    also has infinite order. Thus $W'$ must be irreducible, of rank
    $3$, infinite and cannot be affine (since $y\not = x$).  So the
    signature of the restriction of $\Mpair$ to $V'$ is $(2,1)$.
    Since $\alpha\dprec\gamma$ in $W$, then $\alpha\prec\gamma$ in
    $W'$. Since $x$ and $y$ are visible from $\h\alpha$,
    Proposition~\ref{prop:rank3} applies to our situation: there is a
    sequence $y_n:=w^n\cdot y$ in $E(\Phi')\subseteq E(\Phi)$ that
    converges to $x$; moreover, $y_n\in \h Q\cap L(w^n\cdot
    \h\alpha,w^n\cdot \h\gamma)$ for all $n\in\mathbb N$. Since cover
    of dominance is preserved under the action of $W$, we have
    $w^n(\alpha)\dprec w^n(\gamma)$ and therefore $y_n\in \Ecov(\Phi)$
    for all $n\in\mathbb N$.
   \end{enumerate}
\end{proof}

 \section{Faithfulness of the action on limit roots and universal Coxeter groups} \label{se:univ}
%

 Let $(\Phi,\Delta)$ be a based
root system in $(V,\Mpair)$ with associated Coxeter system $(W,S)$.  A
question that naturally arose in~\cite[Remark 3.4]{HLR} is: is the
$W$-action on the set $E$ of limit roots faithful?

Obviously, if $\Phi$ is finite, it is not the case since $E$ is
empty. If $\Phi$ is of affine type or is indefinite dihedral, then $E$ is finite whereas $W$ is
infinite, so the action cannot be faithful either.  If  $\Phi$ is not irreducible, and we write 
the decomposition in irreducible parts $\Phi=\bigsqcup_{i=1}^p
\Phi_i$, $W=W_1\times \dots \times W_p$, then one sees from
Remark~\ref{rem:irred} that the $W$-action on $E(\Phi)$ is faithful if and only if the $W_{i}$ action on each $E(\Phi_{i})$ is faithful.
 The answer to the question  is therefore given by the following result, the proof of which is one of the aims of this section. \begin{thm}\label{thm:faithful} Assume that $(\Phi,\Delta)$ is irreducible of indefinite type, of rank $\geq 3$. Then the  $W$-action on $E$ is faithful.
\end{thm}

We actually prove a stronger property: for any open set $U$ with $U\cap E\neq \eset$, if some
$w\in W$ fixes $U\cap E$ pointwise, then $w=1$
(Theorem~\ref{thm:faithful2}(c)).

\medskip

The main ingredient of the proof of Theorem~\ref{thm:faithful} is the
following existence property. Given a based root system
$(\Phi,\Delta)$ (irreducible, indefinite) with Coxeter group~$W$, one
can find a root subsystem $(\Phi',\Delta')$ of $(\Phi,\Delta)$ such
that $\Span(\Delta')=\Span(\Delta)$ and  for all $\alpha , \beta \in \Delta'$ ($\al\neq\bt$),
$\mpair{\alpha,\beta}<-1$ (Proposition~\ref{largeuniv}). This implies
the existence of non-dihedral universal  reflection subgroups of $W$  (see \S\ref{ss:refsg} for the definitions).

Refinement of the proof
 leads to a positive answer to another important question, raised
by the first author in~\cite[Question 9.8]{dyer:imc}: can one
approximate, with arbitrary precision, the set of limit roots (resp.,
imaginary convex body) of $\Phi$ with the sets of limit roots (resp.,
imaginary convex body) of its root subsystems associated to reflection
subgroups which are universal Coxeter groups. Actually, we extend the result in two directions.  First, we establish a more general result on similar  approximations of arbitrary faces of the imaginary convex body. Second,  we also  consider approximations by  finite subsets of  the $W$-orbit of a point in the imaginary convex body, and their convex closures. In order to be able to
state precisely these results, we first collect in \S\ref{ss:refsg} some
definitions and known properties of reflection subgroups and universal
Coxeter groups. In \S\ref{ss:distance} we recall the definition of the
Hausdorff metric on compact sets, and state the approximation theorem
(Theorem~\ref{ZEapprox}).

We then continue to the core of this section, starting the steps of
the proofs of Theorems~\ref{thm:faithful} and~\ref{ZEapprox}. In
\S\ref{ss:exist} we state Proposition~\ref{largeuniv} mentioned above,
concerning the existence of reflection subgroups of $W$ that are
universal Coxeter groups. In \S\ref{ss:lines} we give some notations
 and facts related to the way lines in $V_1$
intersect $\h Q$; they will be helpful in shortening the proofs in the
following subsections. In \S\ref{ss:faithful} we check the faithfulness
of the $W$-action on $E$ (Theorem~\ref{thm:faithful}), by proving a
stronger result (Theorem~\ref{thm:faithful2}); the main component of
the proof of which is the existence of universal Coxeter subgroups from
Proposition~\ref{largeuniv}. Among the consequences of this stronger
theorem are also the facts that $E$ has no isolated points, and that
$E$ is not contained in any countable union of proper affine subspaces
of $\aff(E)$. In
\S\ref{ss:cardEext}-\ref{ss:cantor} we state and prove two other
direct consequences of Theorem~\ref{thm:faithful2}: $\ol{Z}=\conv(E)$
has uncountably many extreme points, and $E$ contains a subset
homeomorphic to the Cantor set (in particular, $E$ has the cardinality of $\real$). 
In the last part,
\S\ref{ss:proofapprox}, the proof of the approximation
theorem~\ref{ZEapprox} is completed, using Theorem~\ref{thm:faithful2}
and the tools introduced in \S\ref{ss:lines}. 
Although some of the steps of the proof are  a bit technical,  they
are always constructed from a geometric intuition which is explicitly given.

 \emph{From  \S\ref{ss:faithful} on in this section, we assume unless otherwise stated that $(\Phi,\Delta)$ is an irreducible based root system of indefinite type and rank at least three.}

\subsection{Reflection subgroups and universal Coxeter groups}
\label{ss:refsg}
 Let $(W,S)$ be a Coxeter group. The set
 $T$ of reflections of $(W,S)$ is the conjugacy closure of $S$. A
 \emph{reflection subgroup} $W'$ of $(W,S)$ is a subgroup of $W$
 generated by the reflections it contains, i.e., $W'=\left< W'\cap
 T\right>$. It is easy to see that  $T=\{s_\beta\,|\,\beta\in \PP \}$, so  reflection subgroups
  are  subgroups of $W$ that naturally associated with subsets of $\PP$.
   Let us discuss this relation a bit more (see \cite[\S3.3]{dyer:ref}, or also \cite{bonnafe-dyer}, for more details).
 
 Let $(\Phi,\Delta)$ be a based root system associated to $(W,S)$. For $A\subseteq \PP$, we consider the  reflection subgroup $W'$ generated by the reflections associated to $A$: $W'= \left< s_\beta,\ \beta\in A\right>$. Set:
\[
\Phi'= \Phi_{W'}:=\{\al \in \Phi \mid s_\al \in W'\}\ \text{ and }\ \Phi_{W'}^+:=\Phi_{W'}\cap \Phi^+.
 \]
 To find a set of canonical generators for $W'$, we will first build a simple system for $\Phi_{W'}$. Let 
 \[
 \Delta'= \Delta_{W'}:=\left\{\al \in \Phi_{W'}^+ \,|\, s_\al ( \Phi_{W'}^+\setminus\{\al\}) = \Phi_{W'}^+\setminus\{\al\} \right\}.
 \]
Set $S'=S_{W'}=\{s_\alpha\,|\, \alpha\in \Delta_{W'}\}$, then $(\Phi_{W'},\Delta_{W'})$ is a based root system in $(V,\Mpair)$ with positive roots $\Phi_{W'}^+$ and associated Coxeter system $(W',S')$ (see \cite[Lemma~3.5]{bonnafe-dyer}). Any based root system arising this way will be called below a \emph{root subsystem} of $(\Phi,\Delta)$. In particular, facial root subsystems defined in \S\ref{se:fractal} are  examples of root subsystems. 

\medskip

Note that, even when $\Phi$ is the standard root system of $(W,S)$ and $S$ is finite, $\Delta_{W'}$ may be linearly dependent, and one may have $\mpair{\alpha,\beta}<-1$ for some $\alpha,\beta \in \Delta_{W'}$. This is the main reason for having considered from the beginning of this article (and already in \cite{HLR} and \cite{dyer:imc}) a larger class of root systems than the usual one (see for instance \cite[\S5.1]{HLR} for some examples). Actually $\Delta_{W'}$ may even be infinite. When it is finite, the reflection subgroup (resp., the root subsystem) is said to be of finite rank, or finitely generated.

\medskip

The following fact (from \cite[Theorem~4.4]{dyer:ref}) is fundamental: given a subset~$\Gamma \subseteq \PP$, its associated reflections $R=\{s_\gamma \mid \gamma \in \Gamma\}$ and  reflection subgroup $W'=\left< R\right>$, one has $R=S_{W'}$, i.e., $\Gamma=\Delta_{W'}$, if and only if 
\begin{equation*}
\text{for all distinct }\alpha,\beta \in \Gamma, \mpair{\alpha,\beta} \in \ ]-\oo,-1] \cup
    \{-\cos\left(\frac{\pi}{k}\right), k\in \mathbb Z_{\geq 2} \}.
\end{equation*} 
This is equivalent to saying that $\Gamma$ satisfies the axiom (ii) of a simple system seen in the introduction of~\S\ref{se:imc}. Geometrically, this means that $\h\Gamma$ 
 is the set of extreme points of $\conv(\h{\Phi'})$.

\medskip

A Coxeter group with no non-trivial braid relations, canonically
isomorphic to the free product of cyclic groups of order two generated
by its simple reflections, is called below a \emph{universal Coxeter
  group}. It is the free object for Coxeter groups.  Given the characterization above, a reflection subgroup
$W'$ of $W$ is universal if and only if $\mpair{\al,\bt}\leq -1$ for
all distinct $\al,\bt\in \Delta_{W'}$. It is easily seen that any
reflection subgroup of a universal Coxeter group is universal.  If all the $\mpair{\al,\bt}$ are equal, then a simple computation of eigenvalues shows that the root system is weakly hyperbolic. Otherwise, it is not always the case; see for instance the example of Figure~\ref{fig:nonhyp} and also \cite[Example~1.4]{dyer:imc}. We shall say that a based root system $(\Phi, \Delta)$ is \emph{generic universal} if $\mpair{\al,\bt}<-1$ for all distinct $\al,\bt\in \Delta$.
     
\subsection{Approximation of $E$ and $\ol{Z}$ using reflection subgroups}
\label{ss:distance}

Let $\mathbf{X}$ denote the set of non-empty compact subsets of
$V$. There is a natural distance on $\mathbf{X}$, called the
\emph{Hausdorff metric}, that may be defined as follows (see for
instance \cite[\S2.7]{webster} for details of the definition and
proofs of the few basic properties needed here).  Fix a norm
$\Vert\cdot \Vert$ on $V$ (inducing the standard topology on $V$). For
$K\in \mathbf{X}$ and $\epsilon\in \real_{\geq 0}$, define the
$\epsilon$-neighbourhood of $K$ (see \cite[Fig.~2.12]{webster}):
\begin{equation*} K_{\epsilon}:=\mset{v\in V\mid \Vert v-a\Vert
    \leq \epsilon \text{ \rm for some } a\in K}.
\end{equation*} The Hausdorff metric $\dist \colon \mathbf{X}\times \mathbf{X}\to \real$ on $\mathbf{X}$ is defined by 
\begin{equation*}
 \dist(K,L):=\inf(\mset{\epsilon\in \real_{\geq 0}\mid K \subseteq
   L_{\epsilon} \text{ \rm and }L \subseteq K_{\epsilon}}), \text{ \rm
   for }K,L\in \mathbf{X}.
\end{equation*}  
It is well known that $\dist$ is a metric on $\mathbf{X}$, that the resulting
topology on $\mathbf{X}$ is independent of the choice of norm on $V$ and that
\begin{equation}
\label{eq:distconv}
 \dist(\conv(K),\conv(L))\leq \dist(K,L), \text{ \rm for }K,L\in \mathbf{X}.
\end{equation} 
Another simple property used later in this section is that for $K,L\in \mathbf{X}$
and $\epsilon_{1},\epsilon_{2}\in \real_{\geq 0}$ with $L \subseteq
K_{\epsilon_{2}}$, one has $L_{\epsilon_{1}} \subseteq
K_{\epsilon_{1}+\epsilon _{2}}$. Finally, we shall use the fact that
  \begin{equation}\label{eq:uniondist} \dist(\bigcup_{i=1}^{n}K_{i},\bigcup_{i=1}^{n}L_{i})\leq \max(\mset{\dist(K_{i},L_{i})\mid i=1,\ldots, n})\end{equation} if $K_{i},L_{i}\in \mathbf{X}$ for $i=1,\ldots, n$, where $n>0$.

The second main result of this section is the theorem below. In the special case in which $F=\ol{Z}$, parts of (a)  were raised as a question in \cite[Question 9.8]{dyer:imc} and
were previously established in the case of hyperbolic $W$ by Tom Edgar in
\cite{edgar}.
 
\begin{thm}
  \label{ZEapprox}
  Assume that $(\Phi,\Delta)$ is an irreducible based root system of
  indefinite type. Abbreviate $Z:=Z(\Phi)$ and $E:=E(\Phi)$.  Fix  a face $F$  of $\ol{Z}$ (e.g. $F=\ol{Z}$) and any  $\epsilon >0$.
  
 \begin{num}
 \item There is a finite rank  based root subsystem $(\Phi',\Delta')$  with associated reflection
  subgroup $W'$ of $W$ such that, writing $E'=E(\Phi')$ and $Z'=Z(\Phi')$:
  \begin{subconds}
  \item $(\Phi',\Delta')$ is  generic  universal, so $W'$ is a universal Coxeter group;
  \item $\Span(\Delta')=\Span(\Delta)$;
  \item $\dist(E',F\cap E)<\epsilon$,
  $\dist(\h{\Phi'}\cup E',F\cap E)<\epsilon$ and  $\dist(\wh {\Delta'},F\cap E)<\epsilon$;  \item $\dist\left(\ol{Z'},F\right)<\epsilon$ and 
   $\dist\left(\conv(\h{\Delta'}),F\right)<\epsilon$.
  \end{subconds}  
  Moreover, given any non-empty set of $W$-orbits on $\Phi$, one may
  chose $\Delta'$ so it contains roots from those $W$-orbits and no
  others.
  \item  For any   $z\in \ol{Z}$, there exists  a finite subset $G\seq W\cdot z$   with  
  $\Span(G)=\Span(\Delta)$, $\dist\left(G,F\cap E\right)<
  \epsilon$ and   $\dist\left(\conv(G),F\right)<
  \epsilon$.\end{num}
If $\Phi$ has rank at  least three, one may in addition, for any 
$m\in \Nat$, choose $W'$ and~$G$ above so $W'$ has rank at 
least $m$ and~$G$ has cardinality at least  $m$. 
\end{thm}

\begin{figure}[!h]
\begin{minipage}[b]{\linewidth}
\centering
\scalebox{1}{\input{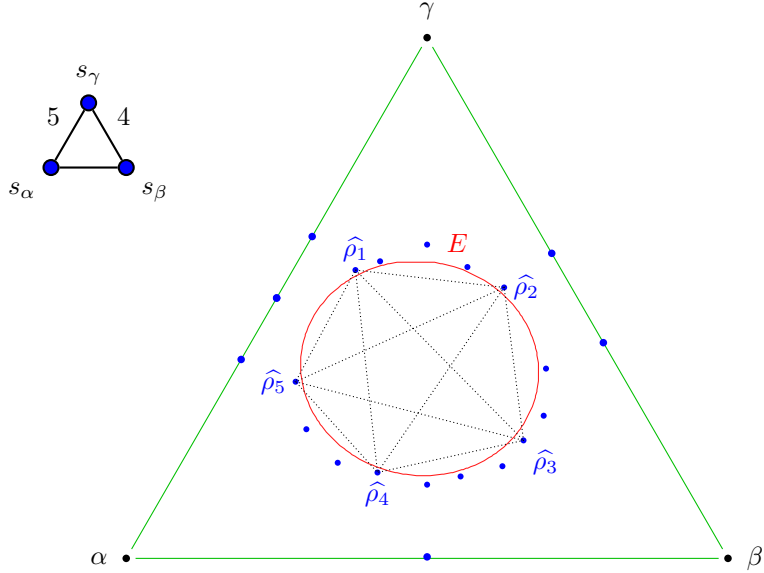}}
\end{minipage}%
   \caption{The geometric intuition behind Theorem~\ref{ZEapprox} when $F=\ol Z$, on
     a simple example (see discussion below the theorem): for any
     $\epsilon>0$, one choose a subset of positive roots
     $\Delta'=\{\rho_1,\dots, \rho_k\}$, such that $\h{\Delta'}$ is
     close enough to $E$ and numerous enough in order to verify
     $\mpair{\rho_i,\rho_j}<-1$ for all $i\neq j$, 
     and~$\dist\left(\conv(\h{\Delta'}),\conv(E(\Phi))\right)<\epsilon$.}
  \label{fig:approx}
\end{figure}

Note that $\mpair{\al,\bt}<-1$ if and only if the segment $[\h \alpha,
  \h \beta]$ intersects $\h Q$ in two distinct points (see for
instance \cite[Figure~3]{HLR}). In particular, in this case they
provide two limit points. In Figure~\ref{fig:approx} we give a
schematic representation of Theorem~\ref{ZEapprox} in the case where the face 
$F$ is $\ol{Z}$ (to which we temporarily restrict attention for
simplicity).  The ellipse represents the normalized isotropic cone $\h
Q$, which in general contains $E$, and in the example of the picture
is exactly $E$. The idea is to choose a sufficiently numerous subset
$\Delta'$ of roots such that $\h{\Delta'}$ is close enough to $\h Q$
that the line joining any two of the roots cuts $\h Q$ in two
points. It is intuitively clear that one could do this if, say, $\h Q$
really is (as in the diagram) the boundary of some strictly convex
body with the roots outside it (this is essentially the hyperbolic
case).  If one can do this, one has (i) automatically since the
segment between distinct normalized roots in $\h{\Delta'}$ cuts $\h Q$
in two points. The main subtlety is that the strict convexity need not
hold: for example, $\h Q$ may contain affine subspaces of positive
dimension (recall Proposition~\ref{prop:bounded}) and choosing roots
close to two limit roots in such an affine subspace does not guarantee
that the segment joining the roots cuts $\h Q$ (since the subspace is
flat).  In fact, even a (non-strictly) convex body need not exist with
properties as above (recall from \cite[Example 1.4]{dyer:imc} that
there are only very weak restrictions on the signature of $\Mpair$ on
$\Span(\Delta)$).  However, the part of $\h Q$ near limit roots behaves
enough like such a body for the proof to go through (see for example Figure~\ref{fig:nonhyp}).  
Since $E$ is the set of limit points of $\h{\Phi}$, it is intuitively reasonable that,
given some technical device to get around the subtleties, one should
be able to choose $\Delta'$ large enough so that $\h{\Delta'}$ is
arbitrarily close to $E$ and $\Delta'$ has the same span as $\Delta$,
giving (ii) and the last part of (iii) (with $F=\ol{Z}$).  The other
parts of (iii) hold since $E'$, $\wh \Phi'\cup E'$ and $\wh{ \Delta'}$
are automatically close for generic universal $\Phi'$ for which the
set of limit roots of rank two standard parabolic subsystems is
sufficiently close to $\wh{ \Delta'}$ (see Lemma \ref{genuni}) and the
position of the latter limit roots can be controlled.  Then part (iv)
follows from (iii) by inequality \eqref{eq:distconv}.  These intuitive
geometric arguments will be made rigorous in \S\ref{ss:exist}
and~\S\ref{ss:proofapprox}.

\begin{rem} 
   We could state an equivalent version of Theorem~\ref{ZEapprox} in
   the context of closed but possibly non-convex cones, 
   by replacing  each subset of $V_{1}$ in the statement of the theorem  by the union of rays through its points and  defining a distance on the set of closed non-empty, nonzero (possibly non-convex) cones included in
   $\cone(\Delta)$ by  $\dist'(C,C'):=\dist(C\cap V_1,C'\cap V_1)$. 
  It is easily seen that the resulting topology is independent of choice of $V_{1}$.  
\end{rem}

The rest of this section is now devoted to the proofs of
Theorem~\ref{thm:faithful} and Theorem~\ref{ZEapprox}.

\subsection{Existence of reflection subgroups that are universal Coxeter groups} 
\label{ss:exist}

The following proposition proves the existence of reflection subgroups
that are universal Coxeter groups; it corresponds to parts (i)-(ii) of
Theorem~\ref{ZEapprox}.

\begin{prop}
  \label{largeuniv}
  Assume $(\Phi,\Delta)$ is irreducible of indefinite type and rank
  two or greater. Then there exists a generic, universal based root
  subsystem $(\Lambda,\Psi)$ of $(\Phi, \Delta)$ such that
  $\Span(\Psi)=\Span(\Delta)$.  In particular $W_\Psi$ is a universal
  Coxeter group.
\end{prop} 
   
Geometrically, this means that one can always find normalized roots
$\{\h{\rho_1}, \dots, \h{\rho_k}\}$ such that for $i\neq j$, $\h{\rho_i}$ and $\h{\rho_j}$ are ``on both sides'' of $\h Q$, i.e., the segment $[\h{\rho_i},\h{\rho_j}]$ intersects $\h Q$ in two distinct points (see Figure~\ref{fig:approx}).

\begin{rem} 
  The proposition implies that the possible signatures of the
  restrictions of $\Mpair$ to $\Span(\Delta)$, for irreducible $\Phi$
  of indefinite type and rank at least three, coincide with those of
  generic universal root systems of the same rank. These are described
  in \cite[Example 1.4]{dyer:imc}; in fact, the proposition provides a
  conceptual explanation for the observation at the end of that
  example.
\end{rem}

Before proving this proposition, we need a technical lemma.  A subset
$\Psi$ of $\Phi$ is said to be \emph{indecomposable} if there is no
partition $\Psi=\Psi_{1}\sqcup\Psi_{2}$ of $\Psi$ into non-empty,
disjoint, pairwise orthogonal subsets $\Psi_{1}$, $\Psi_{2}$. If
$\Psi$ is finite, this is equivalent to indecomposability of the Gram
matrix $(\mpair{\al,\bt})_{\al,\bt\in \Psi}$ in the usual sense, and,
if $\Psi$ is a simple system, it corresponds to irreducibility of
$W_{\Psi}$.
  
\begin{lem}
  \label{condens}
  Let $\Psi$ be an indecomposable finite subset of $\PP$ such that
  $\vert \Psi\vert \geq 3$, $\mpair{\al,\bt}\leq 0$ for all distinct
  $\al,\bt\in \Psi$ and $\mpair{\al,\bt}\leq -1$ for some $\al,\bt\in
  \Psi$.  Let $N\in \real _{> 1}$. Then there exists $\Psi' \subseteq
  \PP$ with the following properties:
  \begin{conds}
  \item $W_{\Psi'} \subseteq W_{\Psi}$, and $\Span(
    \Psi')=\Span(\Psi)$.
  \item There is a bijection $\Psi\xrightarrow{\cong}\Psi'$ which maps
    each element of $\Psi$ to an element of $\Psi'$ in the same
    $W_{\Psi}$-orbit.
  \item If $\al,\bt\in \Psi'$ are distinct, then $\mpair{\al,\bt}<-N$.
   \end{conds}
  In particular,  there is a generic universal based root subsystem 
  $(\Lambda,\Psi')$ of $(\Phi,\Delta)$  with simple system $\Psi'$.  
  \end{lem}

Note that the result does not hold if $|\Psi|=2$: take for example a dihedral reflection subgroup $\left< s_\al,s_\bt\right>$ with $\mpair{\al,\bt}=-1$.

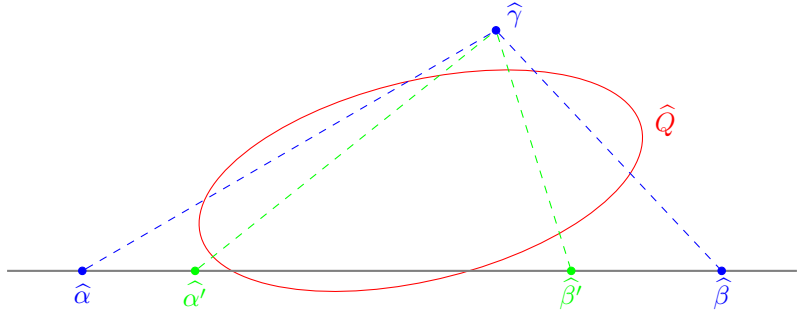
\begin{figure}[!h]
 \begin{minipage}[b]{\linewidth}
\centering
\scalebox{1}{\newcommand{\ccdot}{\!\cdot \!}

\begin{tikzpicture}
  [xscale=5,
    yscale=2,
    pointille/.style={densely dashed},
    root/.style={inner sep=1pt,circle,draw=blue,fill=blue},
    nroot/.style={inner sep=1pt,circle,draw=green,fill=green},
  ]


\begin{scope} [rotate=60]
  
  \draw[red] (0,0) ellipse (0.8cm and 0.5cm);
\end{scope}

\draw[color=red] (0.65,0.4) node{$\h Q$} ;
\draw[color=gray, thick] (-1.1,-0.6) -- (1,-0.6);

\node[root] (g) at (0.2,1) {};
\node[blue] at (0.25,1.1) {$\h{\gamma}$};

\node[root] (a) at (-0.9,-0.6) {};
\node[blue] at (-0.9,-0.75) {$\h\alpha$};

\node[root] (b) at (0.8,-0.6) {};
\node[blue] at (0.8,-0.75) {$\h{\beta}$};

\draw[color=blue, thin, dashed] (a) -- (g) -- (b);

\node[nroot] (aa) at (-0.6,-0.6) {};
\node[green] at (-0.6,-0.75) {$\h{\alpha'}$};

\node[nroot] (bb) at (0.4,-0.6) {};
\node[green] at (0.4,-0.75) {$\h{\beta'}$};

\draw[color=green, thin, dashed] (aa) -- (g) -- (bb);

\end{tikzpicture}











}
\end{minipage}%
  \caption{Condensation of a pair of roots $\{\al,\bt\}$: replacement of such a pair by a new pair $\{\alpha',\beta'\}$ which is closer to $Q$. This is the main idea of the proof of Lemma~\ref{condens}.}
  \label{fig:condens}
\end{figure}

The geometric idea of the proof is to replace progressively any pair
$(\al,\bt)$  such that $\mpair{\al,\bt}\leq -1$ by a pair  $(\al',\bt')$ such that 
$(\h{\al'}, \h{\bt'})$ is on the  
on the same line as $(\h\al, \h\bt)$ but closer to $\h Q$, and with $\al',\bt'$ conjugate to $\al,\bt$ respectively in $\pair{s_{\alpha},s_{\beta}}$ (we call this a
\emph{condensation} at $(\alpha,\beta)$): this will not increase the other inner products $\mpair{\gamma,
  \alpha}$, $\mpair{\gamma, \bt}$ (see
Figure~\ref{fig:condens}) and may decrease some of them. Properties (i)-(ii) are clearly preserved
by this replacement; it turns out one can iterate this process in
order to obtain property (iii).

\begin{proof}  Fix $N\in\real_{>1}$. We will first describe a relation $\rightarrow$ on \emph{admissible $k$-tuples} built up using condensation, and use this relation to show the existence of $\Psi'$.

For $k\in \Nat_{\geq 3}$, call a $k$-tuple $a=(\al_{1},\al_{2},\ldots, \al_{k})$ of positive roots $\al_{i}\in \Phi^{+}$ an \emph{admissible $k$-tuple} if $\mpair{\al_{i},\al_{j}}\leq 0$ for all $1\leq i< j\leq k$, $\mpair{\al_{i},\al_{j}}\leq -1$ for some 
$1\leq i< j\leq k$ and $\G_{a}:=\set{\al_{1},\ldots, \al_{k}}$ is indecomposable.  In particular, this implies that the $\al_{i}$ are pairwise distinct. Let $A_k$ be the set of all admissible $k$-tuples.   For instance, by the assumptions on $\Psi$, there  is   an admissible  $k$-tuple $a=(\al_{1},\ldots,\al_{k})\in A_k$ such that $\mpair{\al_{1},\al_{2}}\leq -1$ and   $\Psi=\G_a=\set{\al_{1},\ldots, \al_{j}}$ is indecomposable for $j=1,\ldots, k$.

Define the \emph{incompleteness index} $I(a)$ of an admissible  $k$-tuple $a\in A_k$ to be $I(a):=\mset{(j,i)\in \Nat\times \Nat \mid 1\leq i<j\leq k, \mpair{\al_{i},\al_{j}}\geq -N}$.  This index can be seen as the margin of error for $\Psi$ to be a desired $\Psi'$.

\smallskip
\noindent\emph{Condensation on $k$-tuples.} 
Consider an admissible $k$-tuple $a\in A_k$ as above.   For any distinct $i,j\in \Nat$ with  $1\leq i,j\leq k$ and $\mpair{\al_{i},\al_{j}}\leq -1$,   define another admissible $k$-tuple $b=(\bt_{1},\ldots, \bt_{k})\in A_k$, which we call a \emph{condensation} of $a$ (at 
$\set{\al_{i},\al_{j}}$), as follows. First, we set 
$\beta_{l}=\alpha_{l}$ for all $l$ with $1\leq l\leq k$ and $l\neq i,j$.  
  We shall set  
  \begin{equation*}
   	\bt_{i}:=(s_{\alpha_{i}}s_{\al_{j}})^{n}(\al_{i}),\qquad \bt_{j}:=(s_{\al_{j}}s_{\alpha_{i}})^{n}(\al_{j})
  \end{equation*} 
  for any    $n\in \Nat_{>0}$ chosen sufficiently large to satisfy the conditions described below. 

     Write $\mpair{\al_{i},\al_{j}}=-\cosh \lambda$ where $\lambda\in \real_{\geq
    0}$.  Define $p_{n}$ for $n\in \Int$ by
  \begin{equation*}
    p_{n}=\begin{cases}n,&\text{if
      $\lambda=0$;}\\ \frac{\sinh(n\lambda)}{\sinh \lambda},&\text{if
      $\lambda>0$.}\end{cases}
  \end{equation*} 
 It is well known and easily checked\footnote{See for example
    \cite[p.3]{howlett}.} that
  \begin{equation*}
    \bt_{i}=p_{2n+1}\al_{i}+p_{2n}\al_{j},\qquad
    \bt_{j}=p_{2n}\al_{i}+p_{2n+1}\al_{j}
  \end{equation*}   
  and $\mpair{\bt_{i},\bt_{j}}=-\cosh ((4n+1)\lambda)$. 
For $  \g\in \G_{a}\sm(\set{\al_{i},\al_{j}}\cup\set{\al_{i},\al_{j}}^{\perp})$ one has $\mpair{\g,\al_{i}}\leq 0$ and $\mpair{\g,\al_{j}}\leq 0$ with at least one of the inequalities being strict.
   Since $p_{m}\to
  +\infty$ as $m\to+\infty$, we may (and do) choose $n\in \Nat_{>0}$ 
 so that
\begin{equation}
\label{cond2} 
  \mpair{\bt_{i},\g}<-N, \quad  \mpair{\bt_{j},\g}<-N 
  \text{ \rm for all }\g\in \G_{\al}\sm(\set{\al_{i},\al_{j}}\cup\set{\al_{i},\al_{j}}^{\perp}). 
\end{equation}
This completes the definition of the $k$-tuple $b$.  We show now that $b\in A_k$.
  Using that  $p_{2n},p_{2n+1}\geq 1$  and $-\cosh ((4n+1)\lambda)\leq -\cosh
  \lambda$, it follows that
\begin{equation}
   \label{cond4}
  \mpair{\bt_{l},\bt_{m}}\leq \mpair{\al_{l},\al_{m}}\text{ \rm  for 
  all   $1\leq l,m\leq k$.}
\end{equation}   
This implies  that $b\in A_k$  with $I(b)\subseteq I(a)$.
  \smallskip
  
$(\star)$: Write $a\xrightarrow{i,j}b$ to indicate that $b\in A_k$ is a condensation of $a\in A_k$ at $\set{\al_{i},\al_{j}}$. This implies that \eqref{cond4} holds,  that   $W_{\G_{b}}\seq W_{\G_{a}}$,  that
 $\Span(\G_{b})=\Span(\G_{a})$  and   that
 $\al_{l}\mapsto \bt_{l}$ defines a bijection  $\G_{a}\to \G_{b}$ 
 with  the property that $\bt_{l}$ is in the  $W_{\G_{a}}$-orbit of $\al_{l}$ for 
 $l=1,\ldots, k$.     Moreover, if $J\seq \set{1,\ldots, k}$ with $\mset{\al_{i}\mid i\in J}$ indecomposable, then 
  $\mset{\bt_{i}\mid i\in J}$ is also indecomposable.
 
Write $a\xrightarrow{*,*} b$ to indicate that $b\in A_k$ is a condensation of $a\in A_k$ for some $\set{\al_{i},\al_{j}}$ and let $\rightarrow$  be the transitive closure of this relation.
 
 \medskip
 \noindent\emph{Back to the proof.} Now take $k=\vert \Psi\vert\geq 3$.
  By the assumptions on $\Psi$, we already noted at the beginning of the proof that there  is   an admissible  
  $k$-tuple $a=(\al_{1},\ldots,\al_{k})\in A_k$ such that $\mpair{\al_{1},\al_{2}}\leq -1$ and 
 $\Psi=\G_a=\set{\al_{1},\ldots, \al_{j}}$ is indecomposable for $j=1,\ldots, k$.
  Consider the set $A_k(a)$ of admissible $k$-tuples $b\in A_k$ such that
    $a\rightarrow b$.   Fix $b\in A_k(a)$ with $I(b)$ minimal under inclusion.  It will suffice to show that $I(b)=\eset$; 
    for then, setting $\Psi'=\G_b=\set{\bt_{1},\ldots, \bt_{k}}$, the above remarks $(\star)$ imply that  (i) holds 
    and that (ii) is satisfied by the bijection  $\al_{i}\mapsto\bt_{i}\colon \Psi\to \Psi'$,  and (iii) holds 
     since $I(b)=\eset$. 
     
     \smallskip
     Suppose to the contrary that $I(b)\neq \eset$. Then   there exists
 a minimum element      $(j,i)\in I(b)$   in the lexicographic total order on $I(b)\seq \Nat\times \Nat$  induced by  
 the standard  total order of $\Nat$. 
 First assume $j\leq 2$, so $(j,i)=(2,1)$. 
Since $\mpair{\bt_{1},\bt_{2}}\leq \mpair{\al_{1},\al_{2}}\leq -1$, 
we have $b\xrightarrow{1,2}g$ for some $g=(\g_1,\dots,\g_k)\in A_k(a)$.
  Since $\set{\al_{l}\mid l=1,2,3}$ is indecomposable, so is
  $\set{\bt_{l}\mid l=1,2,3}$ i.e. $\mpair{\bt_{3},\bt_{l}}<0$ for some 
  $l\in\set{1,2}$. From \eqref{cond2}, we get 
  $\mpair{\g_{3},\g_{l}}< -{N}<-1 $ for $l=1,2$. Now we may 
  define $d=(\delta_1,\dots,\delta_k)\in A_k(a)$ such that
  $g\xrightarrow {1,3}d$. Since 
  $\mpair{\g_{1},\g_{2}}\leq \mpair{\al_{1},\al_{2}}\leq -1$, we have 
  $\mpair{\delta_{1},\delta_{2}}< -N$ by \eqref{cond2}.  In 
  particular, $d\in A_k(a)$, as desired, with 
  $I(d)\seq I(b)\sm \set{(2,1)}\sneq I(\beta)$,
   contrary to minimality of $I(b)$.
   
    Hence $j\geq 3$ and $1\leq i<j$. Since 
  $\Psi'=\set{\al_{1},\ldots, \al_{j}}$ is indecomposable, there is  
  $l\in \Nat$ so $1\leq l<j$ and $\mpair{\al_{j},\al_{l}}<0$. Since 
  $j\geq 3$, we may choose $m\in \Nat$ so $1\leq m<j$,  $m\neq l$ and 
  $i\in \set{l,m}$. Since $l,m<j$,  the minimality of $(j,i)\in I(b)$
   in lexicographic order implies  $\mpair{\bt_{l},\bt_{m}}<-N< -1$. 
  Hence $b\xrightarrow{l,m}g$ for some $g\in A_k(a)$.
  Since $\mpair{\bt_{j},\bt_{i}}\leq \mpair{\al_{j},\al_{i}}<0$ with
  $i\in \set{l,m}$, it follows  by \eqref{cond2} that 
  $\mpair{\g_{i},\g_{j}}<-N$. Therefore, $\g\in A'$ with  
  $I(g)\seq I(b)\sm\set{(j,i)}\sneq I(b)$ contrary to minimality  of 
  $I(b)$. This contradiction completes the proof of Lemma~\ref{condens}.       
 \end{proof}

\begin{proof}[Proof of Proposition~\ref{largeuniv}]
  If $W$ is hyperbolic, this is proved in \cite{edgar}.  We sketch an
  argument for that case since it provides the starting point for the
  proof in general. Since $W$ is hyperbolic, from \S\ref{se:fractal}
  we get that $\h Q$ is the boundary of an ellipsoid inside $\conv(\h
  \Delta)$ and $\Eext=\ol{\Eext}= \h Q$ (see
  Corollary~\ref{cor:Eexthyp} and Theorem~\ref{thm:Qincl}). Choose a
  subset $\set{e_{1},\ldots,e_{n}}$ of $\h Q$ which is affinely
  independent and which has affine span equal to $\aff(\widehat
  \Delta)$.  By the Cauchy-Schwarz inequality, one has
  $\mpair{\frac{1}{2}e_{i}+\frac{1}{2}e_{j},\frac{1}{2}e_{i}+
    \frac{1}{2}e_{j}}<0$ for $i\neq j$. Choose positive roots
  $\rho_{i}$ such that $\widehat{\rho_{i}}$ is sufficiently close to
  $e_{i}$ that $\widehat\rho_{1},\ldots, \widehat \rho_{n}$ are
  affinely independent and
  $\mpair{\frac{1}{2}\widehat\rho_{i}+\frac{1}{2}\widehat\rho_{j},
    \frac{1}{2}\widehat\rho_{i}+ \frac{1}{2}\widehat\rho _{j}}<0$ for
  $i\neq j$. This implies that $\set{\rho_{1},\ldots, \rho_{n}}$ spans
  $\Span(\Delta)$, and a quick calculation shows that for $i\neq j$,
  $\mpair{\rho_{i},\rho_{j}} <-1$, as required in this case.
  
  Next we prove the result in the case $\Delta$ is linearly
  independent, by induction on $\vert \Delta\vert$. If $\Phi$ is of
  rank two, the required conditions are satisfied by taking
  $\Psi=\Delta$. Assume henceforward $\Phi$ is of rank three or
  greater and is not hyperbolic. We claim that $\Phi$ contains a
  proper irreducible standard parabolic subsystem of indefinite type.
  If $\mpair{\al,\bt}<-1$ for some $\al,\bt\in \Delta$ then
  $\Phi_{\left\langle s_{\al},s_{\bt}\right\rangle}$ is such a
  subsystem, so we may assume $\mpair{\al,\bt}\geq -1$ for all
  $\al,\bt\in \Delta$. Then the root system $\Phi$ is the standard
  root system associated to $W$ in \cite[Chapter 5]{humphreys}.  Since
  $W$ is not hyperbolic, by \cite[\S6.8]{humphreys}, there is a
  maximal proper standard parabolic subsystem of $\Phi$ which is not of
  ``positive type'' i.e. which has a component of indefinite
  type. This proves the claim.
  
  By the claim and irreducibility of $\Phi$, we may choose $\al\in
  \Delta$ such that $\Delta\sm\set{\al}$ is irreducible and contains
  the simple roots of an irreducible standard parabolic subsystem of $\Phi$ of
  indefinite type.  Hence $\Delta\sm\set{\al}$ is itself the set of
  simple roots of an irreducible standard parabolic subsystem $\Psi$
  of $\Phi$ of indefinite type. By induction, there are roots
  $\al_{1},\ldots, \al_{n-1}\in \Psi$ such that
  $\Span(\al_{1},\ldots,\al_{n-1})=\Span(\Delta\sm\set{\al})$ and
  $\mpair{\al_{i},\al_{j}}\leq -1$ for all $i,j$ with $1\leq i< j\leq
  n-1$.  Since $\al_{i}\in \cone(\Delta\sm\set{\al})$, one has
  $\mpair{\al,\al_{i}}\leq 0$ for all $i=1,\ldots, n-1$. The
  inequality must be strict for some~$i$ by irreducibility of
  $\Delta$.  Hence $\Psi:=\mset{\al_{1},\ldots, \al_{n-1},\al}$
  satisfies the conditions of Lemma~\ref{condens}. Then any subset
  $\Psi'$ of $\PP$ satisfying the conditions in the statement of
  Lemma~\ref{condens} meets the requirements on $\Psi'$ here.  This
  completes the proof in the special case in which $\Delta$ is
  linearly independent.  Finally, the case in which $\Delta$ is not
  linearly independent immediately reduces to that in which $\Delta$
  is linearly dependent by use of \cite[\S1.4]{dyer:imc}.
 \end{proof} 
  
 \subsection{Intersection of lines with the isotropic cone}
 \label{ss:lines}
 The proofs of Lemmas~\ref{genpos} and~\ref{mainlem}) require a slight
 extension of a computation in \cite[\S4.2]{HLR} that describes the
 intersection points of a line cutting $Q$.  We give the details here
 for ease of reference. In these results, $(\Phi,\Delta)$ can be an arbitrary based root system.

 Let $u,v\in V$. If $u,v$ are distinct, the line $L(u,v)$ passing
 through $u,v$ consists of all points $(1-t)u+tv$ for $t\in \real$,
 whereas if $u=v$, $(1-t)u+tv=u$ for all $t\in \real$.  In any case,
 for $t\in \real$, the point $(1-t)u+tv\in Q\cap L(u,v)$ if and only
 if
 \begin{equation}\label{quadeq}
 \mpair{u+t(v-u),u+t(v-u)}=t^{2}\mpair{v-u,v-u}+2t\mpair{u,v-u}+\mpair{u,u}=0.
\end{equation}
 The above equation, regarded as an equation for $t\in \real$,  has exactly two  distinct solutions if and only if 
 \begin{equation}\label{disccond}
 \mpair{v-u,v-u}\neq 0 \text{ \rm  and } \mpair{u,v}^{2}-\mpair{u,u}\mpair{v,v}>0.
 \end{equation} 
 In that case, the solutions are
\begin{equation*}
t=\frac{\mpair{u-v,u}\pm\sqrt{\mpair{v,u}^{2}-\mpair{u,u}\mpair{v,v}}}{\mpair{v-u,v-u}}
 \end{equation*} 
which we shall denote as  $t=\tmin(u,v)$, $t=\tmax(u,v)$ where $\tmin(u,v)<\tmax(u,v)$. So 
\[
L(u,v)\cap Q=\{(1-\tmin(u,v))u+\tmin(u,v) v, (1-\tmax(u,v))u+\tmax(u,v) v\}.
\]
Observe that by the symmetry between $u$ and $v$, one has 
 \begin{equation}\label{rootsym}
 \tmax(u,v)+\tmin(v,u)=1.
 \end{equation}
So by setting 
\begin{equation*}
  u_{Q}(u,v):=(1-\tmin(u,v))u+\tmin(u,v)v
\end{equation*}
one obtains $u_{Q}(v,u)=(1-\tmax(u,v))u+\tmax(u,v)v$. Therefore for pairs $(u,v)$ verifying~ \eqref{disccond} we have
\[
L(u,v)\cap Q=\{u_{Q}(u,v), u_{Q}(v,u)\},
\]
and $u_{Q}(u,v)< u_{Q}(v,u)$ for the order induced by the oriented line $\overrightarrow{(uv)}$. 
In other words, $u_Q(u,v)$ is the point of intersection of $L(u,v)$ with $Q$ that is seen from $u$ when looking at $Q$; 
whereas $u_Q(v,u)$ is the point of intersection of $L(u,v)$ with $Q$ that is seen from $v$ when looking at $Q$. 
For example, in Figure~\ref{fig:dihedraldom}, we have $x=u_Q(\h\al,\h\bt)=x$ and $y=u_Q(\h\bt,\h\al)$.
\smallskip

Let $U_Q$ be the open set of all pairs $(u,v)$ such that
$L(u,v)\cap Q$ consists of two distinct points, i.e., such that $(u,v)$ satisfies \eqref{disccond}.  The functions $\tmin,\tmax:U_Q\to\real$ are
continuous.  Therefore the function $u_Q\colon U_Q\to V$ given by
$(u,v)\mapsto u_Q(u,v)$ is also continuous.

\smallskip 
The following lemma will be needed in the proof of Lemma
  \ref{mainlem}; it is however natural to state and prove it here. 
  
\begin{lem} \label{quad}
  Suppose that $u,v\in Q$ and $\mpair{u,v}<0$.  Then there exist an
  open neighbourhood $\Omega_u$ of $u$ and an open neighbourhood
  $\Omega_v$ of $v$ with the following properties:
  \begin{conds}
  \item  $\Omega_{u}\times \Omega_{v} \subseteq  U_Q$. 
  \item If $x \in \Omega_{u}$ and $y \in \Omega_{v}$ with
    $\mpair{x,x}>0$ and $\mpair{y,y}>0$, then
    $0<\tmin(x,y)<\tmax(x,y)<1$. That is, $x$, $u_{Q}(x,y)$,
    $u_{Q}(y,x)$, $y$ are distinct and aligned in this order for the
    order induced by the oriented line $\overrightarrow{(xy)}$.
  \end{conds}
\end{lem}

Figure~\ref{fig:quad} illustrates this lemma.

\begin{figure}[!h]
\begin{minipage}[b]{\linewidth}
\centering
\scalebox{0.5}{\input{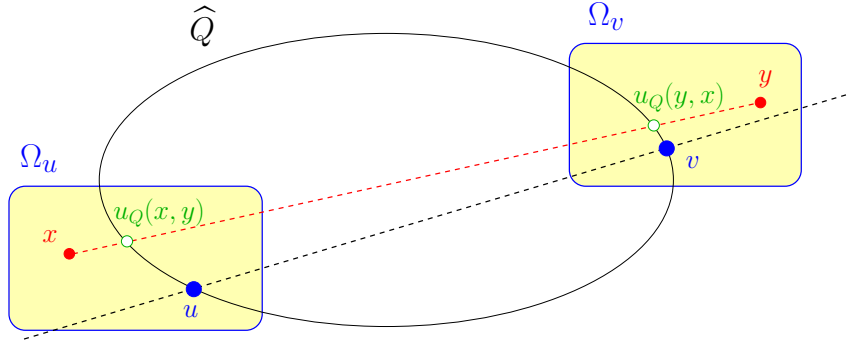}}
\end{minipage}%
\caption{Illustration of Lemma~\ref{quad}: for $u,v\in Q$, there are
  neighborhoods $\Omega_u$ and $\Omega_v$ of $u$ and $v$ such that for
  any $x,y$ in the ``outer side'' of $Q$, if $x\in \Omega_u$ and $y\in
  \Omega_v$, then $x$, $u_{Q}(x,y)$, $u_{Q}(y,x)$, $y$ are distinct
  and aligned in this order.}
\label{fig:quad}
\end{figure}

\begin{proof}   
  From the definitions, $(u,v)\in U_Q$, and $U_Q$ is open. Therefore
  there exist open neighbourhoods $\Omega_{u}$ of $u$ and $\Omega_{v}$
  of $v$ in $V$ such that $\Omega_{u}\times \Omega_{v} \subseteq U_Q$
  and $\mpair{x,y}<0$ for $(x,y)\in \Omega_{u}\times \Omega_{v}$.
  Observe that from Equation~\eqref{quadeq} one has, for all $(x,y)\in
  U_Q$:
  \begin{equation}\label{rootprod}
    \tmax(x,y)+\tmin(x,y)=\frac{2\mpair{x,x-y}}{\mpair{y-x,y-x}}
  \end{equation}
  \begin{equation}\label{rootsum} 
    \tmax(x,y) \tmin(x,y)=\frac{\mpair{x,x}}{\mpair{y-x,y-x}}.
  \end{equation}
  If $(x,y)\in \Omega_{u}\times \Omega_{v}$ are such that
  $\mpair{x,x}>0$ and $\mpair{y,y}>0$, then 
  $\mpair{y-x,y-x}>0$ (since $\mpair{x,y}<0$) and $\tmin(x,y)$ and
  $\tmax(x,y)$ have positive sum and positive product, by
  \eqref{rootprod} and \eqref{rootsum}. Hence $\tmin(x,y)$ and
  $\tmax(x,y)$ are both positive. Similarly, $\tmin(y,x)$ and
  $\tmax(y,x)$ are also positive. Using \eqref{rootsym}, we get that
  they all lie between $0$ and $1$ exclusive. This implies that
  $0<\tmin(x,y)<\tmax(x,y)<1$. The rest of (ii) follows directly.
\end{proof}

Although it is not needed in the proof of the main results of this section, we note also the following consequence of Lemma \ref{quad}. 

\begin{prop}
  \label{dominc}
  ~
  \begin{num}
  \item  Let $(\rho_{n})_{n\in \Nat}$ and 
 $(\tau_{n})_{n\in \Nat}$ be  sequences of  positive roots. Suppose  that for each $n\in \Nat$, $\rho_{n}\prec \tau_{n}$. Let $u\in E$ (resp., 
  $v\in E$) be a limit point of  $(\h\rho_{n})_{n\in \Nat}$ (resp., 
  $(\h\tau_{n})_{n\in \Nat}$). Then 
 $\mpair{u,v}=0$.
  \item Let $(\rho_{n})_{n\in \Nat}$  be a  sequence of  positive roots which is strictly increasing in dominance order:   for each $n\in \Nat$, $\rho_{n}\prec \rho_{n+1}$.  Then the   limit points of $(\h\rho_{n})$ are pairwise orthogonal (and isotropic).
  \end{num} 
\end{prop}

 \begin{proof}  For the proof of (a), one may assume, by passing to subsequences of $(\rho_{n})$, $(\tau_{n})$ if necessary, that 
 $\h\rho_{n}\to u$ and $\h\tau_{n}\to v$ as $n\to \infty$.  One has $u,v\in E$ by definition of $E$. Hence $u,v\in Q$ since $E\seq \h Q$. If $u=v$ then $\mpair{u,v}=0$ so assume $u\neq v$.  From  \eqref{eq:ineqICB}, it follows that $\mpair{u,v}\leq 0$.
 Suppose for a contradiction that $\mpair{u,v}<0$, and choose
 $\Omega_{u}$ and $\Omega_{v}$ as in   Lemma \ref{quad}.
 Choose $m$ sufficiently large that $\h\rho_{n}\in \Omega_{u}$ and $\h\tau_{n}\in \Omega_{v}$ for all $n\geq m$. Then   Proposition \ref{prop:geodomin} says that $[\wh \rho_{m},\wh\tau_{m}]\cap Q=\eset$ while Lemma \ref{quad}(ii) says that $[\wh \rho_{m},\wh\tau_{m}]\cap Q$ consists of two points. This contradiction completes the proof of (a). Part (b) follows by taking $\tau_{n}:=\rho_{n+1}$ in (a). 
    \end{proof}

\subsection{Decomposition of $\conv(\h\Delta)$ in the case of generic universal based root systems}

The function $u_Q$ gives, in the case of a generic universal based root system, a very nice decomposition of the polytope $\conv(\h\Delta)$.
 For instance, in the right-hand side picture in Figure~\ref{fig:ICB}, it looks like $\conv(\h\Delta)$ is the union of $Z$ with the union of the triangle $\conv(\{\gamma, u_Q(\gamma,\beta),u_Q(\gamma,\alpha)\})$, of the triangle $\conv(\{\alpha, u_Q(\alpha,\beta),u_Q(\alpha,\gamma)\})$ and of the triangle $\conv(\{\beta, u_Q(\beta,\gamma),u_Q(\beta,\alpha)\})$. The same kind of geometric intuition holds  in the case of generic universal root systems that are weakly hyperbolic, since in this case the transverse hyperplane can be chosen for $\h Q$ to be a sphere (Proposition~\ref{prop:sphere}). However in the general case it is not that obvious, see for instance Figure~\ref{fig:nonhyp}.  Still, this phenomenon is true in general and was observed first in the framework of the imaginary cone in \cite[\S9.12]{dyer:imc}. 
 
 Suppose that $(\Phi,\Delta)$ is a generic universal based root system.  Then for any distinct $\al,\bt\in \Delta$, we have $(\h\al,\h\bt)\in U_{Q}$, so $u_{Q}(\h\al,\h\bt)\in \conv(\set{\h\al,\h\bt})$ is defined. It is remarked  after the statement of Theorem~\ref{ZEapprox} that, quite generally, for   distinct $\al,\bt\in\Phi^{+}$ with $\mpair{\al,\bt}<-1$,  the line  joining  $\h \al$ to $\h\bt$ cuts $\h Q$ in the two limit roots $u_Q(\h\alpha,\h\beta)\not = u_Q(\h\beta,\h\alpha)$, which are  in the   open line segment joining $\h\al$ and $\h \bt$. For $\al\in \Delta$, let   
 \[
 D_{\al}:=\conv\left(\set{\h\al}\cup\set{u_{Q}(\h\al,\h\bt)\mid \bt\in\Delta\sm\set{\al}}\right) \seq \conv(\h \Delta).
 \] 
 It is clear that $D_\al$ is a polytope that spans the same affine space, and therefore is of the same dimension as $\conv(\h\Delta)$.

\begin{prop} 
\label{prop:decompGeneric}
 Suppose that $(\Phi,\Delta)$ is a generic universal based root system. 
 The polytope $\conv(\h\Delta)$ is the union of the imaginary convex set $Z$ together with the polytopes $D_\alpha$, $\alpha\in\Delta$:
 \[
  \conv(\h\Delta) = Z\cup \bigcup_{\alpha\in\Delta} D_\alpha .
\]
Moreover $Z\cap Q=\eset$ and $D_{\al}\cap D_{\bt}=\eset$ for distinct $\al,\bt\in \Delta$.
\end{prop}

This proposition is a direct reformulation in affine terms of (parts of)  \cite[Lemma 9.11 and Lemma 9.12]{dyer:imc}. We give a proof here for convenience. 

\begin{proof} Assume that $\Delta$ is linearly independent. The general case is dealt with by choosing a lift up as in \cite[\S5.3]{HLR}, see the details in \cite[Proof of Lemma 9.12 (d)]{dyer:imc}. Let $P$ be the polytope $P:=\conv\{u_Q(\h\alpha,\h\beta)\,|\, \alpha\not = \beta \in\Delta\}$. Since $\Delta$ is linearly independent and the root system is universal, the set $\set{u_{Q}(\h\al,\h\bt)\mid \bt\in\Delta\sm\set{\al}}$ is affinely independent of cardinality $|\Delta|-1$ for any $\alpha\in \Delta$. Therefore, for any $\alpha\in\Delta$, the polytope $\conv(\set{u_{Q}(\h\al,\h\bt)\mid \bt\in\Delta\sm\set{\al}})$ is a facet of both $P$ and $D_\al$; the other facets of $P$ and $D_\alpha$ being contained in the facets of $\conv(\h\Delta)$.  In particular, it is not difficult to see that
\[
\conv(\h\Delta)=P\cup \bigcup_{\alpha\in\Delta} D_\alpha.
\]
Since $u_Q(\h\alpha,\h\beta)$ is a limit root, we have that $P\subseteq \conv(E)=\ol Z$, by Theorem~\ref{thm:imc-closure}. Therefore $\ri(P)\subseteq \ri(\ol Z)=\ri (Z)\subseteq Z$, where $\ri(X)$ denotes the relative interior of $X$ (see for instance \cite[Appendix A and Eq. A.2.2]{dyer:imc}).

To show that $\conv(\h\Delta) = Z\cup \bigcup_{\alpha\in\Delta} D_\alpha$, we show by induction on the dimension of $\conv(\h\Delta)$ that $P\subseteq Z\cup \bigcup_{\alpha\in\Delta} D_\alpha$.

If the dimension $\conv(\h\Delta)$ is $2$ then $\Delta=\{\alpha,\beta\}$ and $\conv(\h\Delta)=[\h\al,\h\bt]$. Moreover 
 \[
 D_\al=[\h\al,u_Q(\h\al,\h\bt)],\ D_\bt=[\h\bt,u_Q(\h\bt,\h\al)]\textrm{ and } P=[u_Q(\h\al,\h\bt),u_Q(\h\bt,\h\al)].
 \]
But  $Z=]u_Q(\h\al,\h\bt),u_Q(\h\bt,\h\al)[$ which concludes the dimension $2$ case. Assume now that the dimension of $\conv(\h\Delta)$ 
 is strictly greater than $2$. Let $z\in P$, we may assume $z\in P\setminus \ri (P)$, since $\ri(P)\subseteq Z$.
  So $z$ is in a facet $P'$ of $P$.  We know from the discussion above that either $P'$ is also  a facet of $D_\alpha$, for $\alpha\in\Delta$, and $z\in D_\alpha$; or $P'$ is contained in a facet $\conv(\h\Delta\setminus\{\h\gamma\})$, $\gamma\in\Delta$, of $\conv(\h\Delta)$. In this last case, it is easy to see that, for any $\alpha\in \Delta'=\Delta\setminus\{\gamma\}$,  the set $D_\alpha\cap P'$ is the equivalent object to $D_\alpha$ for the based root subsystem associated to~$\Delta'$. Moreover, $Z\cap \conv(\h{\Delta'})$ is the imaginary convex set for for the root subsystem associated to $\Delta'$, as noted at the beginning of \S\ref{se:dom} (see also~\cite[Lemma~3.4]{dyer:imc}) and 
  \[
  P'=\conv\{u_Q(\h\alpha,\h\beta)\,|\, \alpha\not = \beta \in\Delta'\}.
  \]
So by induction we have that
\[
z\in P'\subseteq \big(Z\cap\conv(\h{\Delta'})\big) \bigcup_{\alpha\in\Delta'} (D_\alpha\cap \conv(\h{\Delta'}))\subseteq Z\cup \bigcup_{\alpha\in\Delta} D_\alpha.
\]
This concludes the first part of the proof.  
  \smallskip

 Now let us prove the ``Moreover'' part of the proposition.   If $z\in \h Q\cap Z$ then there is $w\in W$ such that $w\cdot z \in K\cap \h Q$. Since $w\cdot z\in K\subseteq \conv(\h\Delta)\subseteq \cone(\Delta)$, we can write $w\cdot z=\sum_{\alpha\in\Delta} a_\alpha \alpha$ with $a_\alpha\geq 0$. Let $\Delta'=\{\alpha\in\Delta\,|\,a_\alpha > 0\}$. This set is non-empty, since $w\cdot z\not = 0$, and $|\Delta'|\geq 2$ since $B(w\cdot z,w\cdot z)=0$ and $B(\al,\al)>0$ for any root $\al$. Since $w\cdot z\in K\cap Q$, we have
  \[
  0=B(w\cdot z,w\cdot z)=\sum_{\alpha\in\Delta'} a_\alpha B(\alpha,w\cdot z)\leq 0.
  \]
This forces $ B(\alpha,w\cdot z)=0$, for all $\alpha\in\Delta'$, since $a_\alpha\not =0$. Hence  $w\cdot z$ is in the radical of $B$ restricted to $\Span(\Delta')$. 
Since $w\cdot z=\sum_{\alpha\in\Delta'} a_\alpha \alpha$ with all $a_\alpha\neq 0$ and $\Delta'$ is connected, this implies $\Delta'$ is the simple system of an irreducible affine standard parabolic subsystem of $\Phi$ (see for instance \cite[\S4.5]{dyer:imc}). But since  $(\Phi,\Delta)$ is a \emph{generic} universal based root system, there are no such subsystems, a contradiction which completes the proof that $Q\cap Z=\eset$. The last parts of the proposition hold since for $\alpha\not =\beta$ in $\Delta$ we have $B(\h\al,\h\al)>0$, $B(\h\al,u_Q(\h\al,\h\bt))>0$ and $B(\h\bt,u_Q(\h\al,\h\bt))<0$, whereas for $\gamma\not = \alpha$ in $\Delta$ we have $B(\h\gamma,\al)<0$ and $B(\h\gamma,u_Q(\h\al,\h\bt))<0$.
\end{proof}

\smallskip

\begin{rem} 
  ~
  \begin{enumerate}
  \item This decomposition still holds in the case of universal based root systems that are not necessarily generic, see for an example the right-hand side picture in Figure~\ref{fig:fractalConj}.  But in this case, $Z$ may meet $Q$ and the $D_\al$'s may not be disjoint. For instance $Z$ meets $Q$ in limit points arising from facial affine dihedral based root subsystems, i.e., when $B(\al,\bt)=-1$.
  \item Let $Q^+=\{v\in V\,|\, B(v,v)\geq 0\}$.  If the based root system is generic universal, we have necessarily $Q^+\cap\conv(\h\Delta)\subseteq \bigcup_{\alpha\in\Delta} D_\alpha$, since $Z\subseteq Q^-$ and $Z\cap Q=\eset$,  and the union is disjoint.
  \end{enumerate}
\end{rem}

As a direct consequence of the last statement in the remark above is the following lemma that will be required for the proof of Theorem~\ref{ZEapprox}. 

\begin{lem} 
\label{genuni}
Suppose that $(\Phi,\Delta)$ is a generic universal based root system. Then $\wh \Phi^{+}\cup (\h Q\cap \conv(\h \Delta))\seq \bigcup_{\al\in \Delta} D_{\al}$, and the union is disjoint.
\end{lem}

\subsection{Faithfulness of the $W$-action}
\label{ss:faithful}

For the remainder of Section \S\ref{se:univ}, we assume unless otherwise stated that $(\Phi,\Delta)$ is an irreducible based root system of indefinite type and rank at least three. We aim to
prove next that the $W$-action on $E$ is faithful. Here is roughly the idea
of the proof. Assume $w\in W$ is such that for all $x\in E$, $w\cdot
x=x$. By the definition of the $W$-action, this means that any $x$ in
$E$ is an eigenvector for $w$ with positive eigenvalue (for the linear action of $w$ in
$V$). Hence, $E$ is contained in the union of the eigenspaces of
$w$. It would then be easy to conclude provided we prove the following
fact:
\begin{equation} \label{eq:notcontained} E \textnormal{ is not contained in any finite
union of proper linear subspaces of }\Span(\Delta).\end{equation} It
is clear that $E$ cannot be contained in \emph{one} proper linear
subspace of $\Span(\Delta)$, since we know that
$\Span(E)=\Span(\Delta)$ (see Proposition~\ref{prop:aff}). To prove \eqref{eq:notcontained}, we need to show that even after removing a proper subspace, there are still enough points in $E$ to span $\Span(\Delta)$, and we must be able to repeat this process indefinitely. It would be sufficient to prove that for any $x$ in $E$, there exists an open neighbourhood $U$ of $x$ in $V$ such that $U\cap E$ is enough to span $\Span(E)$. This is what is stated in the theorem below, together with the logical consequences. 

\begin{thm} 
  \label{thm:faithful2} 
   Let $U$ be an open subset of $V$
  with $U\cap E(\Phi)\neq \eset$. Then
  \begin{num}
  \item $\aff(U\cap E(\Phi))=\aff(E(\Phi))$.
  \item $U\cap E(\Phi)$ is not contained in the union of any countable
    collection of proper affine subspaces  of $\aff(E(\Phi))$.
  \item If for all $x\in U\cap  E(\Phi)$, we have $w\cdot x = x$, then $w=1$.
 \end{num}
\end{thm}
 
\begin{rem} 
  \label{rk:faithful2}
  Part (c) is a stronger version of the faithfulness of the $W$-action
  on $E$, i.e., implies Theorem~\ref{thm:faithful}. Together with part (b), it has
  several consequences on~$W$-orbits in $\ol{Z}$ and on the cardinality of $E$ and $\Eext$, statements that we
  postpone to \S\ref{ss:acc-imc}--\ref{ss:cantor}. From (a) we
  already have that $E$ is perfect (i.e. contains no isolated points), so in particular
  it is infinite.
\end{rem}

In the proof of Theorem~\ref{thm:faithful2}, part (b) will be
naturally deduced from (a) and will imply quite easily part (c) (as
explained before the theorem). The difficult part
is to prove (a). The idea is to exhibit enough points in $U\cap E$ to
affinely generate $\aff(E)$. This will be a consequence of the
following fact.

\begin{lem}
  \label{genpos}
 There exists a finite subset $P$ of $E$
  which is not contained in the union of any two proper affine
  subspaces of $\aff(E)$.
\end{lem}
 
We give below a proof for Lemma~\ref{genpos}, and deduce afterwards the proof of Theorem~\ref{thm:faithful2}.

\begin{proof} 
  Note that for any roots $\al,\bt\in \PP$ with $\mpair{\al,\bt}<-1$,
  one has $(\widehat\al,\widehat\bt)\in U_{Q}$ and so $u_{Q}(\widehat
  \al,\widehat \bt)$ and $u_{Q}(\widehat \bt,\widehat \al)$ are
  defined and distinct.  Using Proposition~\ref{largeuniv}, one can
  find $\al_{1},\ldots, \al_{n}\in \PP$ with
  $\mpair{\al_{i},\al_{j}}<-1$ if $1\leq i<j\leq n$, and which form a
  basis for $\Span(\Delta)$ (by taking a subset of $\Psi$ if
  necessary).  The set $A=\mset{\widehat{\al_{i}}\mid i=1,\ldots, n}$ is
  affinely independent and $\aff(A)=\aff(\widehat \Delta)$.  Set
  $\al_{i,j}=u_{Q}(\widehat\al_{i},\widehat\al_{j})\in E$. Then
  \begin{equation}
    P:=\mset{\al_{i,j}\mid 1\leq i\leq n,1\leq j\leq n,i\neq j}
  \end{equation} is a set of $n(n-1)$ 
  points in the simplex $A':=\conv(A)$ with $A$ as vertex set. Note
  that $A$ does not contain any vertex of $A'$ and contains exactly
  two points on each of the $\binom{n}{2}$ edges of $A'$.  In
  particular, $\aff(\set{\al_{i,j},\al_{j,i}} =
  \aff(\set{\h{\al_{i}},\h{\al_{j}}}$ for $i\neq j$.
 
  Clearly, $\aff(P)=\aff(A)=\aff(\h\Delta)$. Denote
  $V':=\aff(\h\Delta)$. It suffices to show that if $H$ is any affine
  hyperplane in $V'$, then $\aff(P\sm H)=\aff(P)$. This follows by a
  simple general argument as follows.  One has $\h\al_{i}\not\in H$ for
  some $i$, say for $i=1$ by reindexing.  Then for each $j=2,\ldots,
  n$, $H$ cannot contain both $\al_{1,j}$ and $\al_{j,1}$ since their
  affine span contains $\h\al_{1}$. We now consider two cases as
  follows.  Assume first that for some $k$ with $2\leq k\leq n$, $P\sm
  H$ contains both $\al_{1,k}$ and $\al_{k,1}$. Then $\aff(P\sm H)$
  contains $\h\al_{1}$. For each $j=2,\ldots, n$, $P\sm H$ contains
  either $\al_{1,j}$ or $\al_{j,1}$, so $\aff(P\sm H)$ contains
  $\h\al_{j}$. Thus, $\aff(P\sm H)$ contains $A$ and so is equal to $V'$
  as required in this case.  In the other case, for each $k$ with
  $2\leq k\leq n$, $H$ contains exactly one of $\al_{1,k}$ and
  $\al_{k,1}$. Then $H$ strictly separates $\h\al_{1}$ from the other
  vertices $\h\al_{2},\ldots,\h\al_{n}$ of $A'$, and so $\al_{j,k}\in P\sm
  H$ for all distinct $j,k$ in $\set{2,\ldots, n}$. Obviously
  $\aff(P\sm H)\supseteq \aff(\set{\h\al_{2},\ldots,\h \al_{n}})$.  Since
  also either $\al_{1,2}$ or $\al_{2,1}$ is in $P\sm H$, we therefore
  get $\h\al_{1}\in \aff(P\sm H)$ as well and so $\aff(P\sm H)=\aff(A)$
  as required in this case too.
\end{proof}
 
\begin{proof}[Proof of Theorem~\ref{thm:faithful2}]
  (a) Choose $x\in U\cap E$. By Proposition~\ref{prop:isoface2},
  $x^{\perp }\cap\ol{Z}$ is a proper face of $\ol{Z}=\conv(E)$ and
  in particular, $x^{\perp}\cap \aff(E)$ is a proper affine subspace
  of $\aff(E)$. By Lemma \ref{genpos}, there is a finite subset $P$ of
  $E\sm x^{\perp}$ with $\aff(P)=\aff(E)$.  By Theorem
  \ref{thm:fractact}, there exists $w\in W$ such that $w\cdot P
  \subseteq E\cap U$. Since the geometric action of $W$ on $V$ is by
  invertible linear maps, we have $ \dim (\Span (w(P))) = \dim (\Span
  (P))=\dim (\Span (E))$. But $\Span (w(P))=\Span (w\cdot P) \subseteq
  \Span(E)$, so we get $ \Span (w\cdot P) =\Span(E)$. Hence
  $\aff(E)=\aff(w\cdot P)\subseteq \aff(E\cap U)$, which completes the
  proof of (a).
  
  \smallskip

  (b) It will suffice to show that if $(H_{n})_{n\in \Nat}$ is a
  family of proper affine subspaces of $\aff(E)$, then there is a
  point $x\in (U\cap E)\sm ( \cup_{n\in \Nat} H_{n})$. By (a), there
  is a point $x_{1}\in (U\cap E)\sm H_{1}$. Choose an open
  neighbourhood $U_{1}$ of $x_{1}$ in $V$ with compact closure
  $\ol{U_{1}} \subseteq U\sm H_{1}$. Since $x_{1}\in E\cap U_{1}$,
  there exists by (a) a point $x_{2}\in (U_{1}\cap E)\sm H_{2}$.
  Choose an open neighbourhood $U_{2}$ of $x_{2}$ with compact closure
  $\ol{U_{2}} \subseteq U_{1}\sm H_{2}$.  Continuing to use (a) in
  this way, choose for each $n\in \Nat_{\geq 3}$ a point $x_{n}\in
  (E\cap U_{n-1})\sm H_{n}$ and an open neighbourhood $U_{n}$ of
  $x_{n}$ in $V$ with compact closure $\ol{U_{n}} \subseteq U_{n-1}\sm
  H_{n}$.  Since $E$ is compact, the sequence $(x_{n})_{n\in \Nat}$ in
  $E$ has a limit point $x\in E$.  Let $n\in \Nat_{\geq 1}$. Since
  $x_{m}\in U_{m} \subseteq U_{n}$ for all $m\geq n$, it follows that
  $x\in \ol{U_{n}}$. Since $\ol{U_{n}}\cap H_{n}=\eset$, $x\not \in
  H_{n}$.  Since $x\in \ol{U_{1}} \subseteq U$, this completes the
  proof of (b).
 
  \smallskip

  (c) Assume that $w\in W$ fixes $E\cap U$ pointwise.  Then
  for each $x \in E\cap U$, $wx=\lambda_{x}x$ for some $\lambda_{x}\in
  \real_{>0}$.  For each $\lambda\in \real$, let $V_{\lambda}$ be the
  $\lambda$-eigenspace of $w$ on $V$.  The above shows that $E\cap U
  \subseteq \cup_{\lambda\in \real}(V_{\lambda}\cap \aff(E))$. Only
  finitely many affine subspaces $V_{\lambda}\cap \aff(E)$ of
  $\aff(E)$ are non-empty, so by (b) we must have $\aff(E) \subseteq
  V_{\lambda}$ for some $\lambda\in \real$. Hence $\Phi \subseteq
  \Span(E) \subseteq V_{\lambda}$. It follows that $w$ is an homothety
  on $V$. Since $W$ is a reflection group, we have $\det(w)=\pm 1$, so
  $w=\pm 1$. But it is well known that for an infinite Coxeter group
  $W$, $-1$ cannot be an element of $W$ (for example by using the
  characterization of the length function in terms of roots, see
  \cite[\S5.6]{humphreys}). Therefore $w=1$.
\end{proof}

\subsection{Accumulation points} 
\label{ss:acc-imc} The fact that  $E$ is perfect  enables  us to prove the following results on accumulation points of $W$-orbits on $\overline{Z}$.  
\begin{cor}
\label{cor:acc-imc}
~
\begin{num}
\item If $z\in E$, then $\Acc(W\cdot z)=E$.
\item If $z\in \ol{Z}$ then $\Acc(W\cdot z)\sreq E$.
\item If $(\Phi,\Delta)$ is of hyperbolic (resp., weakly hyperbolic) type, then
$\Acc(W\cdot z)=E$ for all  $z\in\ol{Z}$ (resp., for all $z\in Z$).  
\end{num}
\end{cor}

\smallskip

\begin{rem}
~
\begin{num}
\item
If $(\Phi,\Delta)$ is irreducible of affine type, (a)--(c) all fail.
If it is hyperbolic dihedral, then  for $z\in \ol{Z}$, one has $\Acc(W\cdot z)\sreq E$ if and only if $\Acc(W\cdot z)= E$. 
\item We do not know if $E=\Acc (W\cdot z)$ for  arbitrary irreducible, non-dihedral    $\Phi$ of indefinite type and  all $z\in \ol{Z}$ (or even just all $z\in Z$).  
\end{num}
\end{rem}

\begin{proof}
Let $z\in E$. Then  $ \ol{W\cdot z}= E$ by Theorem \ref{cor:minimal}. Hence $E\sm W\cdot z=\ol{W\cdot z}\sm W\cdot z\seq \Acc(W\cdot z)\seq \ol{W\cdot z}=E$.   The equality $ \Acc(W\cdot z)=E$ therefore  holds since $E$ has no isolated points, as 
observed  in Remark \ref{rk:faithful2}. This proves (a). By (a), it 
suffices to prove (b)  for $z\in \ol{Z}\setminus E$. Then $W\cdot z \cap E=\eset$, so
\[ \Acc (W\cdot z) \supseteq \ol{W\cdot z} \setminus W\cdot z \supseteq E \] by Theorem~\ref{cor:minimal}. This proves (b).
To prove (c), assume $\Phi$ is hyperbolic (resp., weakly hyperbolic). By (b), it suffices to prove that if $z\in \ol{Z}$ (resp., $z\in Z$), then $\Acc(W\cdot z)\seq E$. Note that if $(\Phi,\Delta)$ is hyperbolic, then $\ol{Z}\seq Z\cup (\ol{Z}\cap \h Q)$. Hence in both the hyperbolic and  weakly hyperbolic cases, 
 Theorem~\ref{thm:Wact} gives the inclusion $\Acc (W\cdot z)\subseteq \h Q$. So, using Theorem~\ref{thm:fractal2}, we have $\Acc (W\cdot z)\subseteq \h Q\cap \ol{Z}=E$. This gives the required equalities.
\end{proof}

\subsection{Cardinality of $\Eext$}
\label{ss:cardEext}
The fact that $E$ is not contained in any countable union of affine
proper subspaces of $\aff(E)$ (implied by
Theorem~\ref{thm:faithful2}(b)) has the following easy consequence.

\begin{cor}
  \label{cor:cardEext} 
  ~
  \begin{num}
  \item The imaginary convex body $\ol{Z(\Phi)}$ and the closed
    imaginary cone $\ol{\imc(\Phi)}$ both have uncountably many faces;
  \item  $\Eext(\Phi)$ is uncountable.
  \end{num}
\end{cor}

\begin{rem}
  \label{rk:cardEext}
 Part (a) is a consequence of (b) (since extreme points are particular
 faces), but will be the first step in the proof of (b).  In particular $E(\Phi)$ is strictly bigger than its countable (dense) subset
 $E_2$. In \S\ref{ss:cantor} we prove a stronger property, namely, any
 open neighbourhood of a point in $E$ is uncountable, by constructing
 a Cantor set inside such a neighbourhood
 (Corollary~\ref{cor:cantor}).
\end{rem}

\begin{proof}  
  (a) The set $\ol{Z}=\conv(\Eext)$ is a convex body living in the
  affine space $A:=\aff(Z)$. We call relative interior $\ri(\ol{Z})$
  of $\ol{Z}$ the interior of $\ol{Z}$ for the induced topology on
  $A$, and relative boundary of $\ol{Z}$ the set $\rb(\ol{Z})=\ol{Z}
  \setminus\ri(\ol{Z})$. It is well known that $\rb(\ol{Z})$ is equal
  to the union of the proper faces of $\ol{Z}$ (see for example
  \cite[Thm.~2.4.12]{webster}). In particular, $\Eext$ is contained in
  $\rb(\ol{Z})$. Moreover, $\rb(\ol{Z})$ is closed, and $E=\ol{\Eext}$
  from Theorem~\ref{thm:Eext}. Thus $E \subseteq \rb(\ol{Z})=\bigcup_F
  F \subseteq \bigcup_F \aff(F)$, where $F$ runs over the proper faces
  of $\ol{Z}$. Using Theorem~\ref{thm:faithful2}(b), this implies that
  $\ol{Z}$ has uncountably many faces, as well as
  $\ol{\imc}=\cone(\ol{Z})$.

  \smallskip

  (b) We first prove that any proper face can be constructed from a
  finite number of extreme points.  Let $F$ be a face of $\ol{Z}$.
  Thus $F=\ol{Z}\cap \aff(F)$. Denote by $X$ the set of
  extreme points of $F$. We have $F=\conv(X)$, so $\aff(X)=\aff
  (F)$. Since $\aff(F)$ is finite-dimensional, one can choose in $X$ a
  finite number of points $x_1,\dots, x_p$ such that $\aff(\{x_1,\dots,
  x_p\})=\aff(F)$. Note that since $F$ is a face of $\ol{Z}$, extreme
  points of $F$ are also extreme points of $\ol{Z}$, so the points
  $x_i$ are in $\Eext$. Thus we can associate to any face $F$ of
  $\ol{Z}$ a finite subset $\{x_1,\dots, x_p\}$ of $\Eext$ such that
  $F=\ol{Z}\cap \aff(\{x_1,\dots, x_p\})$. This construction is clearly
  injective.

  Now suppose by contradiction that $\Eext$ is countable. Then there
  are also countably many finite subsets of $\Eext$. From the
  injection constructed above, this would imply that the set of faces
  of $\ol{Z}$ is countable, contradicting (a).
\end{proof}

\subsection{A Cantor space inside $E$} 
\label{ss:cantor}

We know from Corollary~\ref{cor:cardEext} that $E$ is uncountable. We prove below (as another corollary of
Theorem~\ref{thm:faithful2}) that  any open
neighbourhood of a point in $E$ is also uncountable. In order to do so, we construct, for
any open subset $U$ of $V$ such that $U\cap E\neq \eset$, a Cantor
space living inside $U\cap E$.  Recall that a \emph{Cantor space} is a
topological space that is homeomorphic to the classical (ternary)
Cantor set, or, equivalently, to a  product $\prod_{n\in \Nat}\set{0,1}$ of
 countably infinitely 
many copies of a discrete two-point space. A space is a Cantor space if and only if it is non-empty,
compact, metrizable, totally disconnected (i.e., it has no non-trivial
connected subsets), and perfect,
 see
\cite[\S30]{willard}. A Cantor space has the cardinality of $\real$.

\begin{cor} 
  \label{cor:cantor} 
    Let $U$ be an open subset of $V$
  with $U\cap E\neq \eset$. Then $U\cap E$ contains a subset
  homeomorphic to the Cantor set. Consequently, $U\cap E$ has the
  cardinality of $\real$.
\end{cor}
\begin{proof} We give a proof using well-known facts of general topology.
Recall that a topological space is said to be \emph{topologically 
complete} if its  topology is induced by a complete metric. Any 
closed or open subspace of a topologically complete space is  
topologically complete \cite[\S 7.2, Exercise 6]{munkres}.   It is known that any non-empty,
perfect, topologically complete space contains a Cantor space as 
a subspace \cite[Proposition 3.2.8]{katok}. 
Now we have already seen that $E$ is perfect (see Remark~\ref{rk:faithful2}).
Let $U$ be open in $V$ with $E\cap U\neq \eset$. Since $E$ is perfect and $U$ is open,  $E\cap U$ is perfect.  As $V$ is topologically complete and $E$ is closed (in fact, compact) in $V$,   $E\cap U$ is topologically complete. Hence $E\cap U$ contains a Cantor subspace. \end{proof}
\begin{rem}   \label{rk:cantor}    For   root systems of universal 
Coxeter groups with no affine dihedral subgroups, it is known  that 
$E$ itself is a Cantor set in some cases and it is conjectural that it 
has a Cantor set as quotient space,
with topological balls (of unknown dimensions, possibly all $0$), as 
the fibers of the quotient map more generally
 (see \cite[\S9]{dyer:imc}). For general irreducible root systems, 
 Corollary \ref{cor:fractal1} gives a description of $E$ as  the 
 closure of the union of countably many 
 (not necessarily pairwise disjoint) topological spheres and points
 (see Figure \ref{fig:ICB}). In particular, $E$ is not necessarily a Cantor set. 
 \end{rem}

\subsection{Proof of Theorem~\ref{ZEapprox}} 
\label{ss:proofapprox}

The following technical statement contains the main content of
the proof of Theorem~\ref{ZEapprox}.

\begin{lem}\label{mainlem}
  Let $\al_1, \dots, \al_n \in \Phi$, $x_1, \dots, x_n\in E$, and for
  any $i=1,\dots, n$, let $U_{i}$ be an open neighbourhood of $x_{i}$
  in $V$. Then
  \begin{num}
  \item There exist some $y_{i}\in E\cap U_{i}$, for $i=1,\ldots, n$,   such
    that $\mpair{y_{i},y_{j}}<0$ for all $i\neq j$.
  \item For $i=1,\ldots,n$, there exists some $\rho_{i}\in W\al_{i}\cap\PP$ 
    with the following properties:
    \begin{subconds}
    \item $\h{\rho_i}\in U_{i}$ for all $i$.
    \item $(\h{\rho_i},\h{\rho_j})\in U_{Q }$ and
      $u_{Q}(\h{\rho_i},\h{\rho_j})\in U_{i}$ for all
      $i\neq j$.
    \item $\mpair{\rho_{i},\rho_{j}}<-1$ for $i\neq j$.
    \end{subconds}
  \item If $\aff(\set{x_{1},\ldots, x_{n}})=\aff(E)$, one may further
    require $\aff(\set{y_{1},\ldots, y_{n}})=\aff(E)$ in $\text{\rm
      (b)}$ and $\aff(\set{\wh \rho_{1},\ldots, \wh\rho_{n}})=\aff(E)$
    in $\text{\rm (c)}$.
  \end{num}
 \end{lem}

The lemma is straightforward for $n=1$, and the case $n=2$ can be done
using the construction of Lemma~\ref{quad}; the technical part is to
be able to approximate all the $x_i$'s at the same time, 
preserving the properties we want. We illustrate this construction in
Figure~\ref{fig:lemma}.

\begin{figure}[!h]
\begin{minipage}[b]{\linewidth}
\centering
\scalebox{0.5}{\input{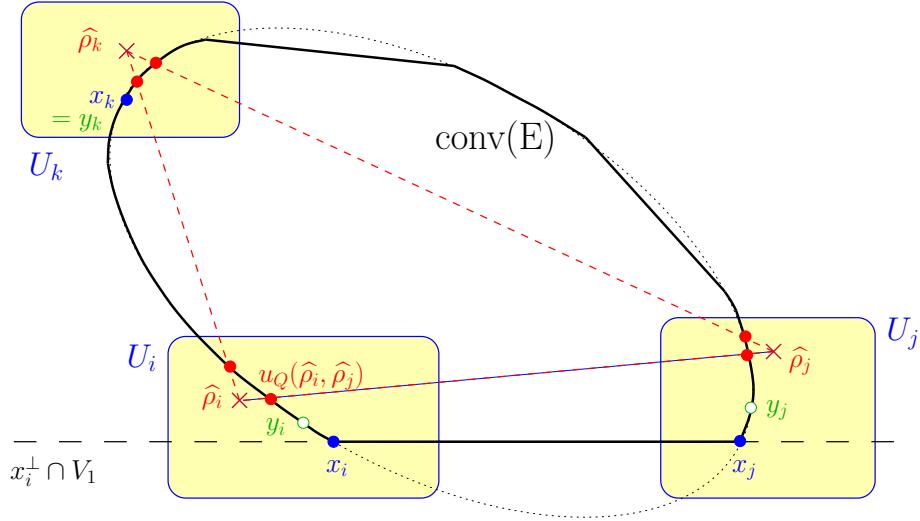}}
\end{minipage}%
\caption{Illustration of Lemma~\ref{mainlem}: $\h Q$ is drawn in black dotted line, and the boundary of
  $\ol{Z}=\conv(E)$ in full line. Given three limit roots $x_i$,
  $x_j$, $x_k$ and their respective neighborhoods $U_i$, $U_j$, $U_k$,
  some of the constructions given by Lemma~\ref{mainlem} are
  depicted. Here we can choose $y_k=x_k$, but, since
  $\mpair{x_i,x_j}=0$, we have to find other $y_i,y_j \in E$ such
  that $y_i \in U_i$, $y_j\in U_j$ and $\mpair{y_i,y_j}<0$. The
  normalized roots $\h{\rho_i}$, $\h{\rho_j}$ and $\h{\rho_k}$
  illustrate item (b) of Lemma~\ref{mainlem}. }
\label{fig:lemma}
\end{figure}

 \begin{proof} 
   The assertions are easily checked if $\Phi$ is dihedral, so we
   assume $\Phi$ has rank at least three.  

   (a) We proceed by induction on $n$.  For $n=1$, it holds trivially.
   Assume that $n>1$ and choose by induction $y_{1},\ldots,y_{n-1}$
   with $y_{i}\in U_{i}$ and $\mpair{y_{i},y_{j}}<-1$ for all $1\leq
   i<j\leq n-1$.  We show by induction on $m$ that for $m=0,1,\ldots,
   n-1$ there is a point $z_{m}\in U_{n}\cap E$ such that
   $\mpair{z_{m},y_{i}}<0$ for $i=1,\ldots, m$. For $m=0$, one may
   take $z_{0}=x_{n}$. Assume that $1\leq m\leq n-1$ and $z_{m-1}$
   exists with $z_{m-1}\in U_{n}\cap E$ and $\mpair{z_{m-1},y_{i}}<0$
   for $i=1,\ldots, m-1$. There is an open neighbourhood $\Omega
   \subseteq U_{n}$ of $z_{m-1}$ in $V$ such that $\mpair{v,y_{i}}<0$
   for all $v\in \Omega$ and all $i=1,\ldots, m-1$.  By
   Theorem~\ref{thm:faithful2}(a), $\Span(E\cap \Omega)
   =\Span(E)$. Since $y_{m}\not\perp E$ by Lemma~\ref{lem:exp}(a),
   there is some $z_{m}\in E\cap \Omega$ with $\mpair{z_{m},y_{m}}\neq
   0$. By Equation~\eqref{eq:ineqICB}, one has $\mpair{z_{m},y_{m}}<0$.
   Since $z_{m}\in \Omega$, $\mpair{z_{m},y_{i}}<0$ for $i=1,\ldots,
   m-1$.  Hence $z_{m}$ has the required properties, and the induction
   on $m$ is complete. In particular, $z_{n-1}$ is defined. We set
   $y_{n}:= z_{n-1}\in E\cap U_{n}$.  Then $y_{i}\in E\cap U_{i}$ for
   $i=1,\ldots, n$ and $\mpair{y_{i},y_{j}} <0$ for $1\leq i<j\leq n$
   as required. This completes the induction on $n$ and hence the
   proof of (a).
   
   \smallskip

   (b) For each $i=1,\ldots, n$, choose by
   Theorem~\ref{cor:minimal}(c) a sequence $(\rho_{i,k})_{k\in \Nat}$
   in $W\al_{i}\cap \PP$ such that $\h{\rho_{i,k}}\to y_{i}$ as $k\to
   \infty$. It follows immediately from the definitions that for all
   distinct $i,j$ in $\set{1,\ldots, n}$, one has $(y_{i},y_{j})\in
   U_{Q}$. By Lemma~\ref{quad}, by passing to subsequences of the
   sequences $(\rho_{i,k})_{k\in \Nat}$ if necessary, we may assume
   without loss of generality that
   \begin{enumerate}
   \item $\h{\rho_{i,k}}\in U_{i}$ for all $i=1,\ldots, n$ and all
     $k\in \Nat$
   \item $(\h{\rho_{i,k}},\h{\rho_{j,\ell}})\in U_{Q}$ for all distinct
     $i,j$ and all $k,\ell\in \Nat$.
   \item $u_{Q}( \h{\rho_{i,k}},\h{\rho_{j,\ell}})$ (resp.,
     $u_{Q}(\h{\rho_{j,\ell}},\h{\rho_{i,k}}))$ is in the open
     interval with endpoints $\h{\rho_{i,k}}$ and
     $u_{Q}(\h{\rho_{j,\ell}},\h{\rho_{i,k}}))$ (resp., $u_{Q}(
     \h{\rho_{i,k}},\h{\rho_{j,\ell}})$ and $\h{\rho_{j,\ell}}$) for
     all distinct $i,j$ and all $k,\ell\in \Nat$.
   \end{enumerate}
   Condition (3) implies that $\mpair{ \rho_{i,k},\rho_{j,\ell}}<-1$
   since the closed interval with the normalized roots
   $\h{\rho_{i,k}}$ and $\h{\rho_{j,\ell}}$ as endpoints cuts the
   isotropic cone in two (distinct) points.  As $k\to \infty$,
   $u_{Q}(\h{\rho_{i,k}},\h{\rho_{j,k}}) \to u_{Q}(y_{i},y_{j})= y_{i}$
   for distinct $i,j$, by continuity of $u_{Q}$.  Hence we may choose
   a sufficiently large integer $N$ such that $u_{Q}(\h{\rho_{i,N}},
   \h{\rho_{j,N}})\in U_{i}$ for all distinct $i,j=1,\ldots, n$.
   Setting $\rho_{i}:=\rho_{i,N}$ for $i=1,\ldots, n$, the conditions
   (i)--(iii) above hold as required. This proves (b).

   \smallskip
   
   (c) Choose open neighbourhoods $U'_{i} \subseteq U_{i}$ of $x_{i}$
   such that for any $v_{i}\in U'_{i}\cap \aff(E)$ for $i=1,\ldots,
   n$, one still has $\aff(\set{v_{1},\ldots, v_{n}})=\aff(E)$. Then
   one can apply (a)--(b) with $U_{i}$ replaced by $U'_{i}$.
  \end{proof}

 \begin{proof}[Proof of Theorem~\ref{ZEapprox}]  
 We assume for simplicity that $\Phi$ has rank at least three, 
 leaving the  indefinite  dihedral case to the 
 reader.   Fix $\epsilon >0$ and $m\in \Nat$. Set $\ek=\epsilon/5$. First we prove (a).  Note that $F$ is closed in $\ol{Z}$, hence compact and convex.  Therefore $F$ is the 
 convex hull of its extreme points, which are extreme points of 
 $\ol{Z}$  and hence are contained in $E\cap F$.
Thus, $F=\conv(E\cap F)$.
   Since $E\cap F$ is compact, there is a finite subset
   $Y$ of $E\cap F$ such that
   $F\cap E\seq Y_{\ek}$.  Since 
   $Y\seq F\cap E\seq (F\cap E)_{\ek}$, it follows that
   $\dist(F\cap E, Y)\leq \ek$.  Since $E$  has no isolated points
   and the affine span of any non-empty open subset of $E$ coincides with $\aff(E)$, one may choose a subset 
   $X=\mset{x_{1},\ldots, x_{n}}$ of $E$, where $\vert X\vert =n\geq
   \max(m,2)$, $\aff(X)=\aff(E)$ and $\dist(X,Y)<\ek$.
   Hence $\dist(F\cap E,X)\leq 2\ek$.
   Set $U_{i}$ to be the open ball with center $x_i$ and radius
   $\ek$. Choose $\rho_{i}\in \PP$ as in
   Lemma~\ref{mainlem}(b)--(c) (for any choice of the $\al_{i}\in
   \Phi$). Define the reflection subgroup $W':=\left< s_{\rho_{i}}\mid
   i=1,\ldots, n\right>$, and denote by $(\Phi',\Delta')$ the
   associated based root system. To simplify notation, write $E'':=\wh {\Phi'}\cup E'$  for  the closure of  $\wh{\Phi'}$. By
   Lemma~\ref{mainlem}(b)--(c),    $\Delta'=\mset{\rho_{1},\ldots, \rho_{n}}$ and (a)(i)--(ii) hold.   
  For $i=1,\ldots, n$, let $D_{i}'=\set{\rho_{i}}\cup
  \mset{u_{Q}(\h{\rho_i},\h{\rho_{j}})\mid
       j\in \set{1,\ldots, n}, j\neq i}$ and $D_{i}:=\conv(D_{i}')$.
       Also set $D:=\bigcup_{j=1}^{n}D_{j}$.  One has $D_{i}'\seq U_{i}$ by  Lemma~\ref{mainlem}(a),(b). Hence $D_{i}\seq U_{i}$ since $U_{i}$ is convex. Note that 
       $ \h{\Delta'}\cap D_{i}$, $E''\cap D_{i}$ and $ E'\cap D_{i}$ are all non-empty. In fact,  the first  is equal to
       $\set{(\h\rho_{i})}$ (using Lemma  \ref{genuni}), the second contains the first  and the last contains 
       $u_{Q}(\h\rho_{i},\h\rho_{j})$ for all $1\leq j\leq n$ with $j\neq i$ (recall $n\geq 2$). 
           This implies  $\dist(\h{\Delta'}\cap D_{i}, \set{x_{i}})\leq\ek$,     $\dist(E''\cap D_{i}, \set{x_{i}})\leq \ek$ and $\dist(E'\cap D_{i}, \set{x_{i}})\leq \ek$. Note  that,  $\h{\Delta'}\seq D$ (trivially)  and, by Lemma \ref{genuni}, 
                $\h{\Phi'}\seq D$   and  $E'\seq \h Q\cap \conv(\h \Delta')\seq  D$.  Hence $E''=\h{\Phi'}\cup E'\seq D$ also.   It follows from \eqref{eq:uniondist}  that       $\dist(\h{\Delta'},X)\leq\ek$, $\dist(E'',X)\leq \ek$ and  $\dist(E',X)\leq\ek$. 
           Together with $\dist(X,F\cap E)\leq 2\ek$ as already established, the triangle inequality now implies  (a)(iii) (with $\epsilon$ replaced by $4\ek<\epsilon$). Then (a)(iv) follows using \eqref{eq:distconv} and $\conv(F\cap E)=F$.      
 
   The final claim in (a) (with statement beginning by ``Moreover'') is proved
   by taking~$X$ sufficiently large (which is possible since $E$ is
   infinite, see Remark~\ref{rk:faithful2}) and choosing the
   $\al_{i}\in \Phi$ above so $\set{\al_{1},\ldots,\al_{n}}$
   contains at least one root from each of the specified $W$-orbits,
   and no roots from the other $W$-orbits.
   \smallskip
   
   Now we prove (b). Choose $X$ as in the proof of (a). By
   Corollary \ref{cor:acc-imc}(b), one may choose a finite subset $G$ of $W\cdot Z$, with $\vert G\vert\geq m$ and $\aff(G)=\aff(X)$, such that $\dist(X,G)<\ek$. Then $\dist(G,F\cap E)<3\ek<\epsilon$  and, using  \eqref{eq:distconv} again, $\dist(\conv(G),F)<\epsilon$. 
\end{proof}

\section{Open problems}
\label{se:questions}

\subsection{Geometric characteristics of $E(\Phi)$ and $\ol{Z(\Phi)}$}
We already formulated an important open question in
\S\ref{se:fractal}, about whether the equality $E=\h Q\cap \conv(E)$
(valid for the weakly hyperbolic case by Theorem~\ref{thm:fractal2})
is true in general irreducible systems. This would provide a nice ``fractal'' description
of $E$, see Proposition~\ref{prop:F}.

 In view of Remark~\ref{rk:cantor}, a
natural problem would be to understand the root systems $\Phi$,
$\Phi'$ such that $E(\Phi)$ and $E(\Phi')$ are homeomorphic. A more
general question is the following: to which extent does the
set of limit roots $E(\Phi)$ characterize the root system $\Phi$?
 It would be interesting to characterize the  root systems
 for which~$E$ is connected, or locally connected, or 
  totally disconnected.

Some questions asked in \cite[\S9.7]{dyer:imc}, on the geometry of the
imaginary convex body, also remain unanswered. For example, we do not
know whether the equality~$\Eexp=\Eext$ holds (see
\S\ref{ss:extpoint}). 

We present in this section other avenues of research and open problems that should be investigated.  The questions raised above and below are generally of greatest interest for irreducible $\Phi$,  even if we do not explicitly make that assumption.

\subsection{Facial structure for subsets of $E$}
\label{ss:facialinter}
In the same way as many combinatorial properties of a Coxeter group
behave well through restriction to parabolic subgroups, the geometric
properties of a based root system usually behave as expected through
restriction to facial root subsystems. Given a based root system
$(\Phi,\Delta)$ with Coxeter group $(W,S)$, recall that $I\subseteq S$
is said to be facial if $\conv(\h \Delta_I)$ is a face of $\conv(\h
\Delta)$. The root subsystem $(\Phi_I,\Delta_I)$ is then said to be
facial (when~$\Delta$ is a basis for $V$, this construction
corresponds to standard parabolic subgroups of~$W$; see
\S\ref{ss:facial} for details). For $I$ facial, denote by $F_I$ the
face $\conv(\h \Delta_I)$. The root system respects the facial
structure: for $I$ facial, we have $\Phi_I=\Phi \cap \Span
(\Delta_I)$, so~$\h{\Phi_I} = \h \Phi \cap F_I$ (hence $\h{\Phi_I}\cap
\h{\Phi_J}=\h{\Phi_{I\cap J}}$ for $I$, $J$ facial). 

 Consider a mapping $E_*$ which associates to any based root system $(\Phi,\Delta)$ a subset $E_*(\Phi)$ of $E(\Phi)$. We say that $E_*$ is a \emph{functorial subset of $E$} if for a based root system $(\Phi,\Delta)$ and a facial root subsystem $(\Phi_I,\Delta_I)$, $E_*(\Phi_I)$ is contained in~$E_*(\Phi)$. Obviously $E$ itself is a functorial subset of $E$. In addition, all the subsets constructed in \S\ref{se:dom} are also functorial: $E_f$ (as well as its $W(\Phi)$-orbit $E_2$), $\Efcov$ (also its orbit $\Ecov$), and $\Eelem$ (and its orbit). Let us say that a functorial subset $E_*$ of~$E$ respects the facial structure if in addition,
\begin{equation} \label{eq:resfac}
E_*(\Phi_I)=E_*(\Phi)\cap F_I, \quad  \text{for any } I \text{ facial.} 
 \end{equation} 
 All the six subsets mentioned above have this property, by Theorem~\ref{thm:intersection}. However, as it was already noted,~$E$ does not satisfy  \eqref{eq:resfac} (see \cite[Ex.~5.8]{HLR}). We ask the general question about how to characterize the functorial subsets of~$E$ which respect the facial structure. A first direction to follow would be to explore what happens in the case of the $W$-orbit $W\cdot x$ of a point $x$.

\subsection{Facial restriction for $E(\Phi)$} 

As mentioned above, 
$E$ does not respect the facial structure as in \eqref{eq:resfac}. 
We would still like to understand the relation, for~$I$ facial, 
between $E(\Phi_I)$ and $E(\Phi)\cap F_I$. Let us describe an 
approach towards  understanding this. The counterexample in
 \cite[Ex.~5.8]{HLR} can be generalized in the following way. 
 Suppose $(\Phi,\Delta)$ is irreducible, and $I$ is facial such that
  $(\Phi_I, \Delta_I)$ is not irreducible. Write
   $\Phi_I=\Phi_1\sqcup \Phi_2$, with 
   $\Phi_1$ $\Mpair$-orthogonal to $\Phi_2$ (this corresponds to 
   taking two subsets of $S$ which are not connected in the 
   Coxeter diagram of $W$). Then we have
    $E(\Phi_I)=E(\Phi_1)\sqcup E(\Phi_2)$. Calculations suggest  
 the possibility   that for any $x\in E(\Phi_1)$, $y\in E(\Phi_2)$, the segment 
    joining $x$ and $y$ is contained in $E(\Phi)$. 
This would create many counterexamples to the facial restriction 
formula \eqref{eq:resfac} for $E$, provided $E(\Phi_1)$ 
and~$E(\Phi_2)$ are non-empty. We do not know whether this property 
 is the only obstruction to the facial formula, i.e., whether in this 
 setting $E(\Phi)\cap V_I$ is exactly the join of all the 
 $E(\Phi_{I,k})$ where the $\Phi_{I,k}$ are the irreducible 
 components of $\Phi_I$. If so, this would imply in particular that 
 when $\Phi_I$ is irreducible, $E(\Phi_I)=E(\Phi)\cap F_I$.

\subsection{Relation with hyperbolic geometry and geometric group theory}\label{sse:RelationGeo}

The relations of our setting with relevant topics in hyperbolic geometry or  geometric group theory are mainly unexplored, but look fertile. 
For instance, consider $(\Phi,\Delta)$ a based root system of rank 3 or 4, and of indefinite type which is  weakly hyperbolic. The Coxeter group~$W$ acts on $E$ as a group generated by hyperbolic reflections, so can be seen as a Fuchsian or Kleinian group, which explains the shape of Apollonian gasket obtained in the figures (see Remark~\ref{rk:klein}). Some of our results are generalizations of known theorems in Kleinian group theory, such as the minimality of the action. 

\smallskip
In \cite{HPR}, which was written after a first version of this paper was circulated, the authors explore some of the relations between hyperbolic geometry and our setting. If $\Phi$ is weakly hyperbolic, it means that $\Phi$ is a root system in the Lorentzian space $(V,B)$, which contains models for the hyperbolic space $\mathbb H^n$, where $n+1=\dim(V)$. In particular, each root is Lorentzian-normal to a hyperbolic hyperplane, so~$W$ turns out to be a discrete subgroup of isometries of $\mathbb H^n$ generated by reflections. Moreover, the set of limit roots $E(\Phi)$ is precisely the limit set $\Lambda(W)$ of $W$ seen as a Kleinian group. 

Starting from this point, a dictionary between our terminology and the terminology commonly used in hyperbolic geometry can be developed. As an example, we may interpret  the {\em convex core} associated to $W$ as follows,  see for instance~\cite[p.637]{Ra06} for the definition of convex core. 

From Proposition~\ref{prop:sphere} and the remark that follows, we
know that, in the weakly hyperbolic type, the transverse hyperplane can
be chosen so that $\h Q$ is a sphere.  Therefore $\conv(\h Q)$ is a
$W$-invariant ball, and its interior $\mathcal B^n$
is a $W$-invariant open ball of dimension $n$.
 Recall from \S\ref{ss:imc}  and Proposition~\ref{rk:action}, that the imaginary convex set $Z(\Phi)=W\cdot K$, i.e., the projective version of the imaginary cone, is the $W$-orbit of the fundamental convex polytope $K$, and that  the closure of $Z(\Phi)$ is $\overline{Z(\Phi)}=\conv(E(\Phi))$, which is contained  in the ball $\conv(\h Q)$.  So the {\em convex core of $\mathcal B^n/W$} is by definition 
\[
C(\mathcal B^n/W)=\big(\conv(E(\Phi))\cap \mathcal B^n\big)/W .
\]

\smallskip

We also point out that $W$ is of finite covolume if and only if the fundamental polyhedron for $W$ in $\mathbb H^n$ is contained in the conical hull of the simple roots, see \cite[\S3.5.2]{HPR}. Our results and framework presented here are valid for all discrete reflection groups generated by reflections in the isometry group of $\mathbb H^n$, so in particular for all discrete reflection groups of infinite covolume. 

Actually, our framework (limit roots $E$, imaginary convex body $Z$, fundamental convex polytope $K$) and many of the results are valid for any Coxeter group geometrically represented as a subgroup of an orthogonal group $O_B(V)$, where $B$ is a (not necessarily non-degenerate) symmetric bilinear form; in this sense our work could be relevant for the community studying infinite covolume actions of discrete groups in more generality. An interesting approach would be then to try to generalize in our framework other classical properties of limit sets of Kleinian groups relative to the dynamics of the $W$-action.

\subsection{Dynamics of the projective action of $W$} Another natural question concerns the dynamics of the projective action of $W$ on all directions of the vector space $V$, not only on the roots and the imaginary cone. After a first version of this paper was circulated, H.~Chen and J.-P.~Labb\'e gave some answers to this question for    $W$ associated to a weakly hyperbolic root system. It turns out that in this case, $E(\Phi)$ is also equal to the set of limit directions arising from the projective action on the \emph{weights} of the root system \cite{chenlabbe1}, but some directions outside $E(\Phi)$ can occur in limit sets of orbits of another direction \cite{chenlabbe2}.

\subsection{Ergodic theory for the $W$-action on $E$}

It is a classical question in ergodic theory of discrete groups, given
a limit set of a group, whether there exists a (unique) invariant
measure on this set, and how to construct it (see for example
\cite[Ch.~3]{nicholls}). Thus a natural problem in our framework is
the search for $W$-invariant measures on $E$. When $W$ acts on a
hyperbolic space (in the context of generalizations of Kleinian
groups), these are well-known questions (see \cite{sullivan}).

When the root system is of indefinite type, and not hyperbolic, $E$
can be qualified as ``fractal'' (see Theorem~\ref{thm:fractact},
Corollaries~\ref{cor:fractal1} and~\ref{cor:cantor} for some fractal
properties). Thus a natural question is to compute the Hausdorff
dimension of $E$. When the root system corresponds to the universal
Coxeter group of rank 4, $E$ is the usual Apollonian gasket inscribed
in a sphere (see \cite[Fig.~9]{HLR}), and its Hausdorff dimension is
about 1.3057 (see \cite{mcmullen}).

\subsection{Construction of converging sequences, combinatorics and dominance order}
Many questions on the precise way in which the normalized roots
converge to~$E$ have been left open. For example, the rate of
convergence (as a  function of the depth of roots) is unknown. Also it
would be interesting to describe explicitly for which sequences of
roots the associated normalized roots converge. More precisely, given
a sequence of positive roots $(\rho_1\leq \rho_2 \leq \dots)$ (increasing  
in the root poset, see~\S\ref{ss:elem}), when does the sequence
 $(\h{\rho_n})_{n\in \Nat}$ converge? This comes down to study the possible limit points of
a sequence $(s_k s_{k-1}\dots s_1(\alpha_0))_{k\in \Nat}$, where
$\alpha_0\in \Delta$ and $(\dots s_k \dots s_1)$ is a (left-)infinite
reduced word of $W$. The case where the word is periodic is of special
interest. When the period is 2, it will provide limit roots 
in~$E_2(\Phi)$. In general, this question requires the precise study of the
asymptotics of sequences of the form $(w^n(\alpha))_{n\in \Nat}$ for
$w\in W$, $\alpha\in \Delta$.

Other related questions are as  follows. Consider  a sequence
 $(\rho_{n})_{n\in \Nat }$  of positive roots, with $\rho_1\prec \rho_2
\prec \dots$ i.e. strictly  increasing in the
dominance order (see~\S\ref{def:dom}).   We do not know   if  $(\h {\rho_n})_{n}$
 has a unique limit root. However, it follows from  Proposition~\ref{dominc}(b)   that any two limit roots
 of a fixed such sequence are
 orthogonal (compare Proposition~\ref{prop:isoface}).  It can be shown  that   in general, not every  limit root in~$E$ is  a  limit root from   
  a  dominance increasing  sequence $(\rho_{n})_{n}$, but it is a limit root of some sequence $(\h {\tau_{n}})$
  where $(\tau_{n})$ is related to some dominance increasing $(\rho_{n})$ as in Proposition~\ref{dominc}(a). 
 
  This suggests a way to associate  subsets of~$E$ to ends of dominance order. The result 
  \cite[Proposition 7.10(c)]{dyer:imc} also suggests  an approach to attaching isotropic faces 
  of the imaginary cone  to ends of weak order.  These ideas have been worked out most fully 
  for generic universal root systems (see~\cite[9.9--9.16]{dyer:imc}) but  basic questions 
  remain open even  in  that specially simple case.
  Clarifying these ideas and their relationships  in general
would contribute to a better understanding of the relation between the
combinatorics of the root system and the distribution of the
normalized roots.  

As shown in \cite{dyer:imc}, the Coxeter system $(W,S)$,
specified by its  Coxeter graph with 
vertex set  $\Delta$, together with the set of  facial subsets of $\Delta$,    suffice to 
determine the face lattice  of the imaginary cone (as lattice with $W$-action). 
 One might  speculate that this information together 
 with the additional data given by the set  of  ``affine edges'' 
$\mset{\set{\alpha,\beta}\mid \al,\bt\in \Delta,\ \mpair{\al,\bt}=-1}$ may 
determine the face lattice of the closed imaginary cone combinatorially.
It  is not  incompatible with what is currently known  that the set  of limit roots  
 may admit a combinatorial description  which  determines it  as a set 
(or perhaps even up to homeomorphism) in terms of this data.

\subsection{Generalization to other frameworks} The concept of root system has many different incarnations, depending on the framework: Coxeter groups, semi-simple Lie algebras, Kac-Moody Lie algebras, extended affine Lie algebras, reductive algebraic groups...; see the many references in the introduction of \cite{loos-neher:root}, where Loos-Neher developed a general framework in order to clarify all these structures (see also H\'ee \cite{hee:root} and the recent work of Fu \cite{fu:pair}). In most of these contexts, the limit roots can still be defined, and in some cases, the isotropic cone as well. We expect that a part of the results in \cite{HLR} and in the present work generalize well to these other settings.

For example, some classes of based root systems appear naturally in the
context of quiver representation, where the positive roots can be
interpreted as dimension vectors for the indecomposable
representations (see \cite{weyman}). The question of an interpretation of the limit roots in this setting is intriguing.

\begin{figure}[h!]
\captionsetup{width=\textwidth} 
\begin{minipage}[b]{\linewidth}
\centering
\scalebox{0.9}{\input{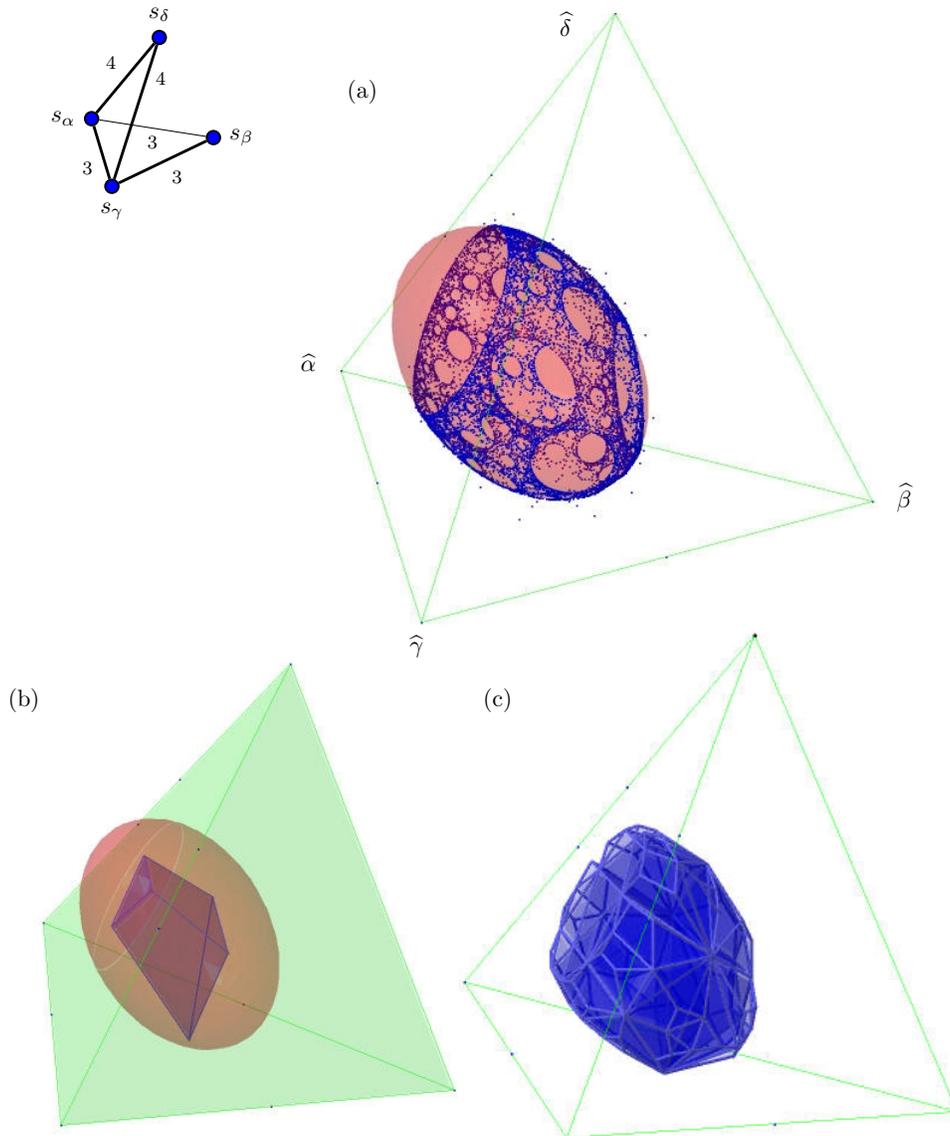}}
\end{minipage}%

\caption{An example of weakly hyperbolic root system (with diagram on the top left corner). (a) Normalized roots (blue dots, drawn until depth 11), which quickly tend to an Apollinian gasket-like shape living on $\h Q$ (in red), as explained in \S\ref{se:fractal}; the sets $\mathcal F_0$ and $\mathcal F$ described in \S\ref{ss:facial}-\ref{ss:fractal2} appear clearly. (b) The polytope $K$ defined in \S\ref{ss:imc}; note how it is truncated by the left face of $\conv(\h\Delta)$ and it touches the bottom face exactly on the limit root of the affine root subsystem generated by $\{\alpha,\beta,\gamma\}$. (c) The first steps of construction of the imaginary convex set $Z$, defined in \S\ref{ss:imc}. We draw all the polytopes $w\cdot K$ for $w$ of Coxeter length $\leq 3$ ($\h Q$ is not drawn, to lighten the picture).}
\label{fig:gasket}
\end{figure}

\newpage

\appendix
\section{Relation of limit roots to Benoist's limit sets}\label{se:Benoist}

Let $\Phi$ be a based root system, associated to a Coxeter group $W$. 
We assume here that $\Phi$ is irreducible, of indefinite type and of rank at least three, and that $\Phi$ spans $V$ linearly. When $\Phi$ is non-degenerate (i.e., the associated bilinear form is non-degenerate), we explain in this appendix how the set of limit roots $E(\Phi)$ can be identified with one of the \emph{projected limit sets} of Benoist \cite{benoist}, which are limit sets associated to a Zariski dense subgroup of a connected reductive algebraic group. Benoist's framework is described in \S\ref{ss:benoistsum}. Constructing the identification involves generalizing first a result by Benoist-De la Harpe \cite{BenHar} on the Zariski closure of a Coxeter group (\S\ref{ss:Zariski}). In \S\ref{ss:non-deg}, we prove the identification of $E$ with Benoist's limit set (Theorem~\ref{benthm}) and we obtain this way a new characterization of the set of limit roots in the non-degenerate case (Corollary~\ref{cor:AppA}).

These results do not extend directly to the case where $\Phi$ is degenerate, because the natural ambient algebraic group is not reductive (\S\ref{reductive?}). However, in this case $E(\Phi)$ will project onto some $E(\Psi)$ with $\Psi$ non-degenerate, as explained in \S\ref{radquot}.

\subsection{}\label{ss:benoistsum} This subsection describes, somewhat informally and imprecisely, a part of the  results of   \cite{benoist}, referring to   \cite{Bor} and \cite[Chapter 1]{Mar} for the necessary background.   Let $\Gamma$ be a Zariski dense subsemigroup  of the group of $k$-points $G(k)$ of  a connected, reductive algebraic group $G$ defined over a local field $k$. Benoist attaches
to $\Gamma$ certain (equivalent) notions of ``limit set'' for $\Gamma$ in $G$. We discuss the realization of the limit set as a  subset  $\Lambda$  of a suitable flag variety  $Y$.
 Our applications involve only   the special case in which $k=\mathbb{R}$, $G$ is semisimple  and $\Gamma$  is a group, and we assume this henceforward for simplicity. 
 Below,  the  set of $k$-points $X(k)$ of a  complex algebraic variety $X$ defined over $k$ is always considered as an analytic $k$-variety (in particular, it is taken  to have  the standard  Hausdorff topology induced from that of $k$).
 
The  standard  parabolic $k$-subgroups of $G$ may be naturally indexed   as $P_{\theta}$
for  subsets $\theta$ of the set $ \Pi $ of restricted simple roots, so that $P_{\theta}\supseteq P_{\theta'}$ if $\theta\subseteq \theta'$.  Attached to   $P_\theta$, one has a     ``flag variety''   $Y_{\theta}:=G(k)/P_\theta(k)$  which is a compact analytic $k$-manifold  on which a   maximal compact subgroup $K$   of $G(k)$ acts transitively.   There is a $K$-invariant probability measure $\mu_{\theta}$ on $Y_{\theta}$. Define $\Lambda_{\theta}$ to be the set of all points $x$ in $Y_{\theta}$  such that there is a sequence $(\gamma_{n})_{n\in \mathbb{N}}$ in $\Gamma$ such that the sequence $\gamma_{n}^{*}(\mu_{\theta})$ of pullback measures converges to a Dirac mass concentrated at $x$.

 The following facts  are from   \cite[\S3]{benoist} (see especially the first paragraph of \S3 and 3.5--3.6).  The set $\Lambda_{\theta}$ is a  closed, $\Gamma$-invariant subset of   $Y_\theta$.
 Let us denote by $Y:=Y_ \Pi $ the flag variety associated to the minimal parabolic $k$-subgroup,  and $\Lambda:=\Lambda_ \Pi $ the set of associated limit points  in $Y$ for $\Gamma$. Because of our assumption $k=\mathbb{R}$, one has $\Lambda\neq \eset$ and  any non-empty  $\Gamma$-invariant closed subset of $Y$  contains $\Lambda$. For $\theta\subseteq  \Pi $, $\Lambda_{\theta}$ is the image of  $\Lambda$ under the natural  projection  $Y\to Y_{\theta}$.

 The results of the preceding paragraph  imply that  for all $\theta\subseteq  \Pi  $, $\Lambda_{\theta}$ is a non-empty, closed, $\Gamma$-invariant subset of $Y_{\theta}$, and that any non-empty  $\G$-invariant  closed subset of $Y_{\theta}$ contains $\Lambda_{\theta}$.
 These properties uniquely characterize $\Lambda_{\theta}$ and may be     taken  as definitions for our purposes below.  They imply  in particular that  the $\Gamma$-action on each  $\Lambda_{\theta}$ is minimal.
 We call $\Lambda$  the limit set of $\Gamma$. By a  \emph{projected limit set}, we mean  a set $\Lambda_{\theta}$ for some $\theta\subseteq  \Pi  $.

\subsection{} \label{ss:Zariski} For non-degenerate, spanning based root systems  (associated to a Coxeter group $W$),   the set of limit roots  will be identified below with a suitable projected limit set.  First we need to understand what is the right algebraic group to consider. The Zariski closure of $W$ in its standard reflection representation is described by Benoist-De la Harpe in \cite{BenHar}. We extend below their result to the more general class of  reflection representations considered in our paper.  Although this  result is used  here only for non-degenerate forms, we state it in natural  generality corresponding to that in op. cit. 

Fix a based root system $(\Phi,\Delta)$ in $(V,B)$, together with its Coxeter group $W$.  We assume  from now on  that $\Phi$ is irreducible of indefinite type and  of rank at least three, and that  $\Phi$ spans  $V$ linearly. Let $O(V)=O(V,B)$ denote the orthogonal group of $(V,B)$; that is,   
\begin{equation*}
O(V):=\{g\in \text{\rm GL}(V)\mid B(gv,gv')=B(v,v')\text{ \rm for all $v,v'\in V$}\}. 
\end{equation*} Let   $V^\perp  =\{v\in V\mid B(v,v')=0 \text{ \rm for all $v'\in V$}\}$ denote the  radical of $(V,B)$, and define the   following  subgroup  of $O(V)$:

\begin{equation*}
 H(B):=\{g\in  O(V)\mid  g(v)=v \text{ \rm for all $v\in  V^\perp  $}\}.
\end{equation*}
  Let $(V_{\mathbb{C}},B_{\mathbb{C}})$ denote the quadratic space arising as the  complexification of $(V,B)$ (i.e., $V_{\mathbb{C}}:=V\otimes_{\mathbb{R}}\mathbb{C}$ and $B_{\mathbb{C}}$ is the symmetric bilinear form on 
$V_{\mathbb{C}}$ arising by extension of scalars to $\mathbb{C}$ from $B$ on $V$). Similarly as above, we define the orthogonal group $O(V_{\mathbb{C}})=O(V_{\mathbb{C}},B_{\mathbb{C}})$ and its subgroup  $H(B_{\mathbb{C}})=H(V_{\mathbb{C}},B_{\mathbb{C}})$. Regarding
the natural map $\text{\rm GL}(V,\mathbb{R})\to \text{\rm GL}(V_{\mathbb{C}},\mathbb{C})$
as an inclusion, we regard $H(B)$ as a subgroup of $H(B_{\mathbb{C}})$. 
Note that $H(B_{\mathbb{C}})$ is a linear algebraic group, since it is closed in the Zariski topology of $\text{\rm GL}(V_{\mathbb{C}},\mathbb{C})$. More precisely, we view 
$H:=H(B_{\mathbb{C}})$ as a complex linear algebraic group defined over $k$ with
$H(k):=H(B)$ as its (Zariski dense) group of $k$-points.

The main result of \cite{BenHar} extends to this setting as follows:

\begin{prop}\label{ss:Zariskithm} The Zariski closure of $W$   is $H$ (i.e., $W$ is Zariski dense in $H(k)$). \end{prop} 
\begin{proof} In case $\Delta$ is linearly independent the 
argument is the same, mutatis mutandis, as that in \cite{BenHar}. The general case can be reduced to that case  as follows. Choose a subset $\Delta'$ of $\Delta$ which is  inclusion maximal 
subject to the requirements that $\Delta'$ is linearly independent and  the corresponding 
standard parabolic subgroup  $W_{\Delta'}=\langle s_\alpha\,|\, \alpha\in \Delta'\rangle$ is irreducible. We claim that $\Delta'$ spans  $V$. Otherwise, there is some $\alpha\in \Delta\setminus \Span(\Delta')$. By irreducibility of $W$, one may suppose without loss of generality that $B(\alpha,\beta)\neq 0$
for some $\beta\in \Delta\cap \Span(\Delta')$.  This implies $B(\alpha,\gamma)\neq 0$ for some 
$\gamma\in \Delta'$. Then $\Delta'':=\Delta'\cup\set{\alpha}$ is linearly independent and $W_{\Delta''}$ is irreducible, contrary to maximality of $\Delta'$.
Note $W_{\Delta'}$ is of  rank at least three and is  of indefinite type,   since its  type is determined by the signature of $B$.  By the case of linearly independent simple roots, $W_{\Delta'}$ is Zariski dense in $H(k)$ and hence so is $W\supseteq W_{\Delta'}$.\end{proof}

\subsection{} \label{reductive?}  To apply Benoist's results, the ambient algebraic group should be connected and reductive. Since $H$ is not connected, we will first need to replace $H$ with $H^0$, its connected component of the identity, and $W$ with $W\cap H^0(k)$. We therefore need  the following simple fact.

\begin{prop} The algebraic group $H^0$ is reductive if and only if 
$V^\perp =\{0\}$.
\end{prop}

\begin{proof}
Choose a complementary subspace $U  $ to $V^\perp$ in $V$ and let $B_{U} $ 
be the restriction of $B$ to a symmetric bilinear form on $U$. 
Let $r$ denote the 
dimension of $V^\perp$ and $m$ that of $U  $.  Let $A$ denote the $m\times m$ matrix (with 
respect to 
some basis) of $B_{U} $ on $U  $. Then (with respect to a basis obtained by extending that 
basis by a basis of $V^\perp$) the matrix of $B$ on $V$  is a diagonal block matrix 
$\text{\rm diag}(A,0_r)$ , where  $0_r$ is the
 $r\times r$ zero matrix. 
Then    $H(k)$ (resp., $H$)   identifies with the group of all real (resp., complex) block  matrices
of the form 
\begin{equation}
\begin{bmatrix}\label{eq:matrix}X&0\\Y&\Id_{r}\end{bmatrix}
\end{equation} 
where $Y$, of size $r\times m$, is arbitrary and  $X$  satisfies   $X^{t}AX=A$. The subgroup of such (complex) matrices with
$Y= 0$ identifies with the complex semisimple  algebraic group 
$(O(U  )_{\mathbb{C}},(B_{U} )_{\mathbb{C}})\cong O(m,\mathbb{C})$. On the other hand, 
the set of   complex matrices  \eqref{eq:matrix} with  $X=\Id_{m}$ is a unipotent  normal (abelian) 
subgroup   of $H$; it is the unipotent radical $R_{u}H^{0}$. 
It follows that  $H^{0}$ is 
reductive if and only if $ V^\perp =\{0\}$, in which case $H^{0}$ is semisimple.
\end{proof}

\subsection{}  The following notation will prove convenient below. For any finite-dimensional real  vector space $U'$, let $\mathbb{P}(U')$
  denote the projective space with points the real 
 lines in $U'$, in the usual (Hausdorff) topology. 
 For $X\subseteq U'$, let $[X]\subseteq \mathbb{P}(U')$ denote the set of lines  spanned by non-zero points of   $X$.

\subsection{}
\label{ss:non-deg}   We assume in this subsection that $(\Phi,\Delta)$ is  (spanning and)  non-degenerate, i.e. $V^\perp=\{0\}$. 
Let us denote as usual by $Q$ the isotropic cone of $B$, $Q:=\mset{v\in V\mid B(v,v)=0}$. In the following we explain how to identify the set of isotropic lines $[Q]$ with some partial flag variety $Y_\theta$, for some $\theta\subseteq \Pi  $ as in \S\ref{ss:benoistsum}.

 The assumed  non-degeneracy of  $B$  implies  that  $H= O(V_{\mathbb{C}}, B_{\mathbb{C}})$, so the 
connected component  $G:=H^0= SO(V_{\mathbb{C}}, B_{\mathbb{C}})$ of  the identity of  $H$ is a  
semisimple  complex algebraic $k$-group.  Let $n=\dim V$, so that $G\cong SO(n,\mathbb{C})$.  Its group $G(k)$ of $k$-points identifies with 
$SO(V,B)\cong SO(p,q)$ where $(p,q)$ is the signature of  $(V,B)$.  Let $r:=\min(p,q)$ be the 
Witt index of $(V,B)$.  Since  $(\Phi,\Delta)$  is of indefinite type and rank at least three, we have $r\geq 1$ and $p+q\geq 3$.    Fix a   
choice of maximal $k$-split  torus in $G$ and a set $ \Pi  $ of simple relative roots for the  
corresponding relative root system for $G$, which is  of type $B_{r}$ if $p\neq q$ and $D_{r}$ (interpreted as $A_{1}\times A_{1}$, $A_{3}$ for $r=2,3$) if $p=q$ (see  \cite[23.4]{Bor}). The  standard  minimal parabolic $k$-subgroup $P_{ \Pi  }$ 
identifies (see loc. cit.) with the stabilizer in $G$ of a standard maximal flag, defined over $k$, 
 of totally isotropic subspaces of $(V_{\mathbb{C}},B_{\mathbb{C}})$. 
 The standard parabolic 
 $k$-subgroups $P_{\theta}$, where $\theta\subseteq  \Pi  $,  of $G$ are precisely  the 
 subgroups of $G$ which contain $P_{ \Pi  }$. The standard  $k$-parabolic subgroups all have  interpretations similar to that of $P_{ \Pi  }$, as 
 stabilizers of standard isotropic  flags in $V_{\mathbb{C}}$ defined over $k$, but there are complications  in 
 type $D$ because there are two $G$-orbits of maximal isotropic $\mathbb{C}$-subspaces of 
 $V_{\mathbb{C}}$. 
 
For  purposes here, it suffices to note that    
there is a  standard parabolic $k$-subgroup of $G$, which we write as $P_{\theta}$ for some 
$\theta\subseteq  \Pi  $, given by the stabilizer of the isotropic line in that standard maximal 
isotropic flag.  (We do not need the explicit description of $\theta$ as a subset of $\Pi$, but it may easily be determined for each type of root system). The corresponding (partial) flag variety  $Y_\theta=G(k)/P_\theta(k)$  identifies with the $G(k)$-orbit 
in $\mathbb{P}(V)$ of the corresponding (real)  line.
Now all isotropic lines in $(V,B)$ are in the same $G(k)=SO(V,B)$ orbit, since (by Witt's 
theorem) they  are in the same orbit for $O(V,B)$, and any one of them is stabilized by the 
reflection in some  non-isotropic vector orthogonal to that line (such a line always exists since 
$p+q\geq 3$). Hence $Y_{\theta}$ naturally identifies (as homogeneous spaces for  $G(k)$) 
with the set $[Q]$ of all isotropic lines in $[V]=\mathbb{P}(V)$. The above identification 
$[Q]=Y_{\theta}$ can be made as analytic manifolds, but it suffices here   to make it as 
topological spaces (which is straightforward since both are compact Hausdorff spaces).

Let $\Gamma:=W\cap G(k)=\{w\in W\mid l(w) \text{ \rm is even}\}$ be   the
 ``rotation subgroup''  of $W$, regarded as Zariski dense subgroup of $G(k)$.   Denote the projected limit set for $\Gamma$ in $Y_{\theta}$  as  $\Lambda_{\theta}$, as in \S\ref{ss:benoistsum}.

\begin{thm}\label{benthm} Assume  $(V,B)$ is non-degenerate and $\Delta$ spans $V$, and make the  identification $[Q]=Y_{\theta}$,   with the specific $\theta$ defined in the previous paragraphs.
 Then $\Lambda_{\theta}=[E(\Phi)]$, i.e., the projected limit set $\Lambda_\theta\subseteq Y_{\theta}$ for $\Gamma$ as a  Zariski dense subgroup  of $ G(k)$   identifies with the set of limit roots $E(\Phi)$, as subsets of $[Q]$. In particular,  $\Lambda_{\theta}$  is $W$-stable.  \end{thm}

\begin{proof}
Note $\Gamma$ is a  normal subgroup of $W$, which acts on $[Q]$ by restriction of the natural $O(V,B)$-action given by  $g[\mathbb{R}\alpha]=[\mathbb{R}g\alpha]$  for any $g\in O(V,B)$  and non-zero $\alpha\in Q$.  For any $w\in W$, $w\Lambda_{\theta}$ is 
a minimal non-empty closed  $\Gamma$-invariant subset of $[Q]$ (since $\Gamma=w\Gamma
w^{-1}$) and therefore coincides with $\Lambda_{\theta}$ (which is the unique minimal such subset). Hence $\Lambda_{\theta}$ is stable under the $W$-action on $[Q]$.

Since $[E({\Phi})]$ is a non-empty, closed, $\Gamma$-invariant subset of $[Q]$, one has 
$ \Lambda_{\theta}\subseteq [E(\Phi)]$, by minimality of $\Lambda_{\theta}$ amongst sets with those properties.  But then $\Lambda_{\theta}$ is a non-empty, closed $W$-invariant  subset of $[E(\Phi)]$, and  the minimality of the $W$-action (Theorem \ref{cor:minimal}) on $[E(\Phi)]$ forces equality in the inclusion.  
\end{proof}

\begin{cor}\label{cor:AppA} If  $(V,B)$ is degenerate and $\Delta$ spans $V$, any non-empty, closed $W$-invariant  subset of $[Q]$ contains $[E(\Phi)]$.  
\end{cor}
\begin{proof}
This follows since the previous theorem implies it holds  with $[E(\Phi)]$ replaced by $\Lambda_{\theta}$ and $W$ by its subgroup $\Gamma$. 
\end{proof}

\begin{rem}
\label{rk:5}
~
\begin{enumerate}
\item  Although \cite{benoist} and (our extension of) \cite{BenHar}  easily imply as above the existence of a unique  non-empty closed  $W$-invariant  subset of  $[Q]$  on which the $W$-action is minimal,  we do not know how  to prove that set identifies with $[E(\Phi)]$ except as above, i.e., by use of  our  Theorem \ref{cor:minimal}.  In particular, we do not have a way to relate the two notions of limit sets ($\Lambda_\theta$ and $E(\Phi)$) directly from their definitions, without using their characterizations via minimality.

\item We do not know how to prove Corollary~\ref{cor:AppA}   without use of \cite{benoist}.   The related result we have (Theorem~\ref{thm:fractal2}) assumes that the root system is weakly hyperbolic and states only that any non-empty, closed $W$-invariant  subset of $[Q]$, that is \emph{also contained in} $[\conv(\Delta)]$, is equal to $[E(\Phi)]$. 

\item  For $\Phi$ non-degenerate, Theorem \ref{benthm} provides an  interpretation of $[E(\Phi)]$ as a  projected limit set, and \cite{benoist} then yields many additional facts about $E(\Phi)$ which seem  likely to have significant applications (see for example Remark \ref{rk:6}(3)). 

\item It is an interesting question whether other projected limit sets $\Lambda_{\theta'}$ for $\Gamma$, and especially the limit set  $\Lambda$  itself, can be given interpretations similar  to those in the theorem in terms of the root system. The corresponding flag varieties $Y_{\theta'}$  involve flags containing higher dimensional totally  isotropic spaces,  and such isotropic subspaces already appear naturally in the study of limit roots
(see for instance Proposition  \ref{dominc}).
\end{enumerate}
\end{rem}

\subsection{}\label{radquot}  We now consider the situation for possibly degenerate root bases.  Let us explain a classical way (after Krammer) to obtain from a degenerate root system a non-degenerate one, with the  same Coxeter group. Let  $\pi\colon V\to V/V^\perp$ 
be the natural map. The restriction of $\pi$ to  $U $  identifies $U  $ isomorphically with 
$V/ V^\perp $ as real vector space, and we further identify $(U  ,B_{U} )$ with the quotient of $(V ,B)$ by its 
radical $ V^\perp$.  
Since $(\Phi,\Delta)$ is of indefinite type, one  has
$\cone(\Delta)\cap V^\perp=\{0\}$, and there is a (non-degenerate, spanning) based root system 
$(\Psi, \Sigma)$ for $(U  ,B_{U} )$ where  $\Psi=\pi(\Phi)$ and $ \Sigma=\pi(\Delta)$ 
(see \cite[6.1]{krammer}).
The Coxeter system attached to  $(\Psi, \Sigma)$  identifies canonically with $W$, with its natural 
action on the quotient space $U  =V/V^\perp$.

Let $V_{1}$ be the fixed affine subspace  of $V$ transverse to $\Phi$ and  
$V_{0}$ be its  translate   through $0$, as  in \S\ref{ss:limitroots}. The map   $v \mapsto [\{v\}]$ 
for $v\in V_{1}$   identifies $V_{1}$ homeomorphically with  
   $[V_{1}]=\mathbb{P}(V)\setminus \mathbb{P}(V_{0})$, an open subset of 
    $\mathbb{P}(V)$. The  rule
$[\{v\}]\mapsto [\{\pi(v)\}]$ for $v$ in $\cone(\Delta)\cap V_{1}$
  defines a continuous  surjective map $[\cone(\Delta)]\to [\cone( \Sigma)]$ of compact Hausdorff 
  spaces (hence it is  a closed, proper quotient map). Since $v$ is isotropic if and only if 
  $\pi(v)$ is isotropic, it  easily follows that this map restricts to a surjective continuous (closed, proper, quotient) map
   $[E(\Phi)]\to [E(\Psi)]$
(where $E(\Phi)$ and $E(\Psi)$ are the sets of limit roots for $\Phi$ and $\Psi$ respectively) 
 and that this latter map is in addition $W$-equivariant.

By  Theorem~\ref{cor:minimal}(b), the $W$-actions on  $[E(\Phi)]$ and $[E(\Psi)]$ are minimal. One easily 
sees that minimality  on $[E(\Phi)]$ directly implies that on $[E(\Psi)]$, but we do not know any 
direct argument for the converse implication.  Therefore, results on the limit sets for non-degenerate root systems do not easily extend to degenerate ones.

\subsection{} We give a simple example to show that the above map $[E(\Phi)]\to [E(\Psi)]$ is not bijective in general.  Let $(\Phi,\Delta)$ be the standard based root system attached to the following  Coxeter graph, in which vertices are labeled by the corresponding simple roots:
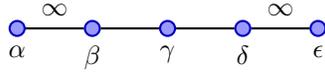
\begin{center}
\begin{tikzpicture}[sommet/.style={inner sep=2pt,circle,draw=blue!75!black,fill=blue!40,thick},] 
      \coordinate (ancre) at (0,-1.5);
      \node[sommet,label=below:$\alpha$] (t1) at (ancre) {};
      \node[sommet,label=below :$\beta$] (t2) at ($(ancre)+(1,0)$) {} edge[thick] node[auto,swap] {} (t1);
      \node[sommet,label=below :$\gamma$] (t3) at ($(ancre)+(2,0)$) {} edge[thick] node[auto,swap] {} (t2);
    	 \node[sommet,label=below :$\delta$] (t4) at ($(ancre)+(3,0)$) {} edge[thick] node[auto,swap] {} (t3);
      \node[sommet,label=below :$\epsilon$] at ($(ancre)+(4,0)$) {} edge[thick] node[auto,swap] {} (t4);
            \node[black] at (0.5,-1.25) {$\infty$};
            \node[black] at (3.5,-1.25) {$\infty$};
\end{tikzpicture} 
\end{center}

Denote the associated quadratic space as $(V,B)$.
 One easily checks that $\alpha+\beta-\delta-\epsilon$ is  in $V^{\perp}$. 
One has $\widehat{\alpha+\beta}=\lim_{n\to \infty}\widehat{(s_{\alpha}s_{\beta})^{n}\alpha}\in E(\Phi)$ and similarly 
$\widehat{\delta+\epsilon}\in E(\Phi)$. 
Hence the above map $[E(\Phi)]\to [E(\Psi)]$ sends the distinct  elements $[\{\alpha+\beta\}]$ and $[\{\delta+\epsilon\}]$ of $[E(\Phi)]$ to the same element of $[E(\Psi)]$. 

Note also that although $(\Phi,\Delta)$ is a standard based root system,
we are not able to deduce the minimality of the $W$-action on $E(\Phi)$ from Benoist's results.

\begin{rem}
\label{rk:6} 
~
\begin{enumerate}
\item We do not know  if, for  degenerate, spanning  $(\Phi,\Delta)$, with $ V^\perp$ defined as  the radical of  $(V,B)$,  any closed non-empty $W$-invariant subset of $[Q]\setminus [V^\perp]$ contains $[E(\Phi)]$ (though the corresponding  statement with  $[Q]\setminus [ V^\perp]$ replaced by $[Q]$ obviously fails in general).

\item In the case $\Phi$ is degenerate, \S \ref{radquot} gives a $W$-equivariant surjection from $E(\Phi)$ to some $E(\Psi)$ with $\Psi$ non-degenerate. This construction may allow one to transfer some of the properties known  in the non-degenerate case to the degenerate case, but not all. 
Remark (3) below illustrates both  this point and Remark \ref{rk:5}(3).  

\item We sketch  another proof of  faithfulness of the $W$-action on $E(\Phi)$  in the setting of Theorem~\ref{thm:faithful2}  ($\Phi$ indefinite of rank at least 3, and irreducible)  as follows.  From \S\ref{radquot} one always has a surjective $W$-equivariant map from $E(\Phi)$ to some $E(\Psi)$ where $\Psi$ is non-degenerate. So it is sufficient to prove the faithfulness property in the non-degenerate case.

Thus, assume now that $\Phi$ is non-degenerate. Using the notations and result of Theorem~\ref{benthm}, $[Q]$ identifies to $Y_\theta$ and $E(\Phi)$ to $\Lambda_\theta$. 
Using the projection  $Y \to Y_{\theta}$, the Zariski density of  $\Lambda$ in $Y$ 
(see \cite[Lemma 3.6]{benoist}) implies that of $\Lambda_{\theta}$ in $Y_{\theta}$.  Therefore, if $w\in W$ acts as the identity on $[E(\Phi)]$, it acts as the identity  on $[Q]$. This implies $w$ fixes each isotropic line in $V$. Since $\Phi$ is irreducible and non-degenerate, this readily  implies  that $w$ acts as the identity on $V$ and hence $w=1$ by faithfulness of the $W$-action on~$\Phi$.
\end{enumerate}\end{rem}

\bibliographystyle{alpha}
\bibliography{LimitRoots2}

\end{document}